\documentclass[12pt]{amsart}
\usepackage{amsthm,amssymb,amsmath,epsfig}
\newtheorem{thm}{Theorem}[section]
\newtheorem{prop}[thm]{Proposition}
\newtheorem{lem}[thm]{Lemma}
\newtheorem{cor}[thm]{Corollary}

\newtheorem{rem}[thm]{Remark}

\begin{document}

\title[A generalization of the Dijkgraaf-Witten invariants]%
      {A generalization of the Dijkgraaf-Witten invariants for cusped 3-manifolds}
\author[N. Kimura]{Naoki Kimura}{}
\address{Department of Mathematics, Graduate School of Fundamental Science and Engineering, Waseda University, 3-4-1 Okubo, Shinjuku, Tokyo 169-8555, Japan}
\email{noverevitheuskyk@toki.waseda.jp}
\maketitle
\thispagestyle{empty}


\begin{abstract}
We introduce a generalization of the Dijkgraaf-Witten invariants for cusped or compact oriented 3-manifolds.  
We show that the generalized DW invariants distinguish some pairs of cusped hyperbolic 3-manifolds 
with the same hyperbolic volumes and with the same Turaev-Viro invariants.
We also present an example of a pair of cusped hyperbolic 3-manifolds with the same hyperbolic volumes and with the same homology groups,
whereas with distinct generalized DW invariants.
\end{abstract}

\section{Introduction}

In 1990 Dijkgraaf and Witten \cite{Dijkgraaf} introduced a topological invariant of closed oriented 3-manifolds using a finite group and its 3-cocycle.  
Let $M$ be a closed oriented 3-manifold, $G$ a finite group and $\alpha \in Z^3(BG,U(1))$.  
Then the Dijkgraaf-Witten invariant $Z(M)$ (we abbreviate it to the DW invariant in this paper) is defined as follows:
\begin{equation*}
Z(M) = \frac{1}{|G|}\sum_{\gamma\in {\rm Hom}(\pi_1(M),G)} \langle\gamma^*[\alpha],[M]\rangle \in \mathbb{C} .
\end{equation*}

\noindent The topological invariance of $Z(M)$ is obvious from the definition and it is also evident that $Z(M)$ is a homotopy invariant
since $M$ only appears at the fundamental group and the fundamental class in the definition of $Z(M)$.

Dijkgraaf and Witten reformulated the invariant by using a triangulation of $M$ in the following way.  
Let $K$ be a triangulation of $M$.  Then the fundamental class of $M$ is described by the sum of the tetrahedra of $K$ 
and $\gamma \in {\rm Hom}(\pi_1(M),G)$ is represented by assigning an element of $G$ to each edge of $K$.  
$Z(M)$ is described as follows:
\begin{equation*}
Z(M) = \frac{1}{|G|^a}\sum_{\varphi\in {\rm Col}(K)} \prod_{{\rm tetrahedron}}  \alpha(g,h,k) ^{\pm1},
\end{equation*}

\noindent where $a$ is the number of the vertices of $K$ and $g,h,k \in G$ are colors of edges of a tetrahedron of $K$.  
Wakui \cite{Wakui1} proved the topological invariance of the DW invariant in this combinatorial construction.  
Due to the above construction of $Z(M)$ by using a triangulation, we can view the DW invariant as the \lq\lq{Turaev-Viro type}\rq\rq invariant.

This construction by using a triangulation enable us to define the DW invariant for a compact oriented 3-manifold $M$ with $\partial M \neq \emptyset$.  
However, for $\partial M \neq \emptyset$ case, the DW invariant of $M$ is determined 
not only by $M$ but also by a triangulation of $\partial M$ and its coloring.  

Here we construct another version of the DW invariant, which we call the generalized DW invariant.  
For a compact oriented 3-manifold $M$ with $\partial M \neq \emptyset$, the generalized DW invariant of $M$ does not need a triangulation of $\partial M$ nor its coloring.  
We can achieve that by using an ideal triangulation of a compact oriented 3-manifold with non-empty boundary or a cusped oriented 3-manifold.  
This is an analogy of the construction of the Turaev-Viro invariant in \cite{Benedetti1} for a compact 3-manifold with non-empty boundary or a cusped 3-manifold.


We calculate the generalized DW invariants for some examples and show that 
the invariants actually distinguish some pairs of cusped hyperbolic 3-manifolds with the same hyperbolic volumes and with the same Turaev-Viro invariants.
We also give an example of a pair of cusped hyperbolic 3-manifolds with the same hyperbolic volumes and with the same homology groups, 
meanwhile with distinct generalized DW invariants.

\vspace{0.35cm}
 
\begin{acknowledgment}
The author would like to express his gratitude to Professor Jun Murakami for his many useful suggestions.  
He would like to thank Professor Yuji Terashima for useful communications. 
He also would like to thank members of Murakami\rq s laboratory, especially Toshihiro Asami, Yuta Kadokura and Atsuhiko Mizusawa for helpful conversation.
\end{acknowledgment}

\section{Definition of the generalized Dijkgraaf-Witten invariant}

First we review the group cohomology briefly.  Let $G$ be a finite group and $A$ a multiplicative abelian group.  
The $n$-cochain group $C^n(G,A)$ is defined as follows:

\[
 C^n(G,A) = 
  \begin{cases}
   A & (n = 0)\\
   \{ \alpha : \overbrace{G\times \cdots \times G}^{n} \rightarrow A \} & (n\ge 1).
  \end{cases}
\]

\noindent The group operation of $C^n(G,A)$ is a multiplication of maps induced by the multiplication of $A$ and then $C^n(G,A)$ is a multiplicative abelian group 
since $A$ is so.  
The $n$-coboundary map $\delta^n :  C^n(G,A) \rightarrow  C^{n+1}(G,A)$ is defined by
$$(\delta^0 a)(g) = 1 \quad (a \in A,\;g \in G),$$
$$(\delta^n \alpha)(g_1,\ldots,g_{n+1}) =$$ 
$$\alpha (g_2,\ldots,g_{n+1}) (\prod_{i = 1}^n  {\alpha (g_1,\ldots,g_i g_{i+1},\ldots,g_{n+1})}^{{(-1)}^i}) {\alpha (g_1,\ldots,g_n)}^{{(-1)}^{n+1}},$$ 
$$(\alpha \in C^n(G,A),\;g_1,\ldots,g_{n+1} \in G,\;n\ge 1).$$

\noindent Then we can confirm by the above definition that $\{(C^n(G,A), \delta^n)\}_{n=0}^{\infty}$ is a cochain complex.  
Hence the $n$-cocycle group $Z^n(G,A)$ and the $n$-th cohomology group $H^n(G,A)$ are defined as usual.

An $n$-cochain $\alpha \in C^n(G,A)$ is said to be {\it normalized} if for any $g_1,\ldots,g_n \in G$, $\alpha$ satisfies
$$\alpha(1,g_2,\ldots,g_n) = \alpha(g_1,1,g_3,\ldots,g_n) = \cdots =\alpha(g_1,\ldots,g_{n-1},1) = 1 \in A.$$

\noindent If $\alpha$ and $\beta$ are normalized $n$-cochains, $\alpha \beta$ and $\alpha^{-1}$ are also normalized $n$-cochains 
and $\delta^n \alpha$ is a normalized ($n+1$)-coboundary.  Eilenberg and MacLane proved the following proposition \cite[Lemma 6.1 and Lemma 6.2]{Eilenberg}. 

\begin{prop}
For any cochain $\alpha$, there exists a normalized cochain $\alpha\rq$ which is cohomologous to $\alpha$.  
For any normalized n-coboundary $\alpha$, there exists a normalized (n-1)-cochain $\beta$ such that $\alpha = \delta^{n-1} \beta$.
\end{prop}

\noindent Hence we assume that any $n$-cochain is normalized.  
As we only consider 3-cocycles in the rest of this paper, we restate the cocycle condition for a 3-cocycle $\alpha$.
$$\alpha (h,k,l) \alpha (g,hk,l) \alpha (g,h,k) = \alpha (gh,k,l) \alpha (g,h,kl)   \quad(g,h,k,l \in G).$$

\noindent The cocycle condition takes an important role in the proof of the invariance of the generalized DW invariant in Section 3.

We can define the generalized DW invariant by using any multiplicative abelian group $A$, nevertheless we usually use $U(1)$ in the definition of the original DW invariant.  
Hence we only consider $U(1)$-valued 3-cocycles in the rest of this paper. 
 
In this paper we suppose that a triangulation $K$ of a 3-manifold is not necessarily a decompositon as a simplicial complex.  
(A triangulation in this paper means a singular triangulation in \cite{Matveev} and \cite{Turaev}.)  
For given four vertices of $K$, $K$ may have more than one tetrahedron with the given four vertices.  
For given two vertices of $K$, there may exist more than one edge connecting the given two vertices.  
If a decomposition forms a simplicial complex, we call the decomposition {\it a simplicial triangulation}. 

Let $M$ be a compact oriented 3-manifold with boundary.  
We consider a triangulation of $M$ with ideal vertices such that each boundary component of $M$ converges at an ideal vertex.  
We call such a triangulation of $M$ with ideal vertices {\it a generalized ideal triangulation} of $M$ in this paper.  
In general, a generalized ideal triangulation $K$ of $M$ has both interior vertices and ideal vertices.  
If $\partial M = \emptyset$, $K$ has no ideal vertices, that is, $K$ is an ordinary triangulation of a closed 3-manifold $M$.  
On the other hand, an ideal triangulation is a generalized ideal triangulation without interior vertices.  

Now we explain a coloring and a local order of a triangulation.  

Fix a generalized ideal triangulation $K$ of $M$. Give an orientation to each edge and each face of $K$.
{\it A coloring} $\varphi$ of $K$ is a map 
$$\varphi:\{{\rm oriented \:edges \:of} \:K\}\to G$$ 
satisfying 
$$\varphi(E_3)^{\epsilon_3}\varphi(E_2)^{\epsilon_2}\varphi(E_1)^{\epsilon_1} = 1 \in G$$ 
for oriented edges $E_1$, $E_2$ and $E_3$ of any oriented 2-face $F$ and
\[
 \epsilon_i = 
  \begin{cases}
   1 & {\rm the \:orientation \:of} \:E_i {\rm \:agrees \:with \:that \:of}\:\partial F\\
   -1 & {\rm otherwise.}
  \end{cases}
\]
(Note that the three edges $E_1$, $E_2$ and $E_3$ of $F$ are chosen along the orientation of $F$ as Figure 1.)  
The above condition for a coloring $\varphi$ is required because a coloring $\varphi$ originally comes from $\gamma \in {\rm Hom}(\pi_1(M),G)$.  
Let Col($K$) be the set of the colorings of $K$.
Note that a coloring $\varphi$ of $K$ is independent of the choice of orientations of edges and faces of $K$.

\begin{figure}
\centering
\includegraphics[width=3cm]{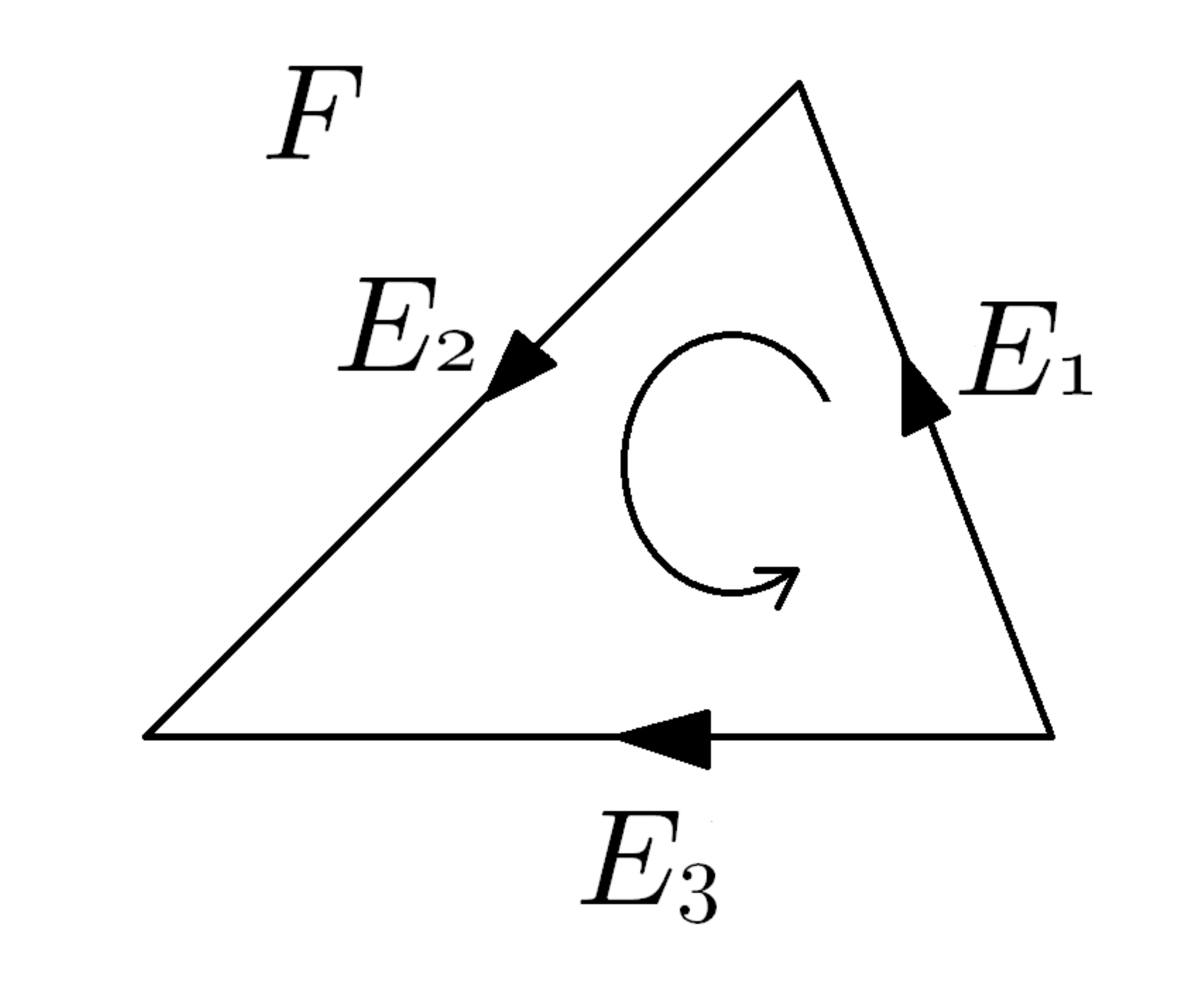}

\begin{small}
$\epsilon_1 = 1, \epsilon_2 = 1, \epsilon_3 = -1.$
\end{small}
\caption{The sign of edges.}
\end{figure}

Fix a generalized ideal triangulation $K$ of $M$.  
Give an orientation to each edge of $K$ such that 
for any 2-face $F$ of $K$, the orientations of the three edges of $F$ are not cyclic (as the left hand side of Figure 2).  
We call such a choice of the orientations of edges of $K$ {\it a local order of $K$} (or {\it a branching of $K$}).
Then each tetrahedron $\sigma$ of $K$ has one of each vertex incident to $i$ outgoing edges of $\sigma$ and 
to $(3-i)$ incoming edges of $\sigma$ for $i = 0,1,2,3$ (as the right hand side of Figure 2).  
Let $v_i$ be the vertex of $\sigma$ incident to $i$ outgoing edges of $\sigma$.  
Then the order $v_0 < v_1 < v_2 < v_3$ of the vertices of $\sigma$ settles an orientation of $\sigma$.  
We define the sign $\epsilon_{\sigma}$ of $\sigma$ as follows:
\[
 \epsilon_{\sigma} = 
  \begin{cases}
   1 & {\rm the \:orientation \:of} \:\sigma {\rm \:by \:the \:local \:order \:agrees \:with \:that \:of}\:M\\
   -1 & {\rm otherwise.}
  \end{cases}
\]  

\begin{figure}
\centering
\includegraphics[width=3.4cm]{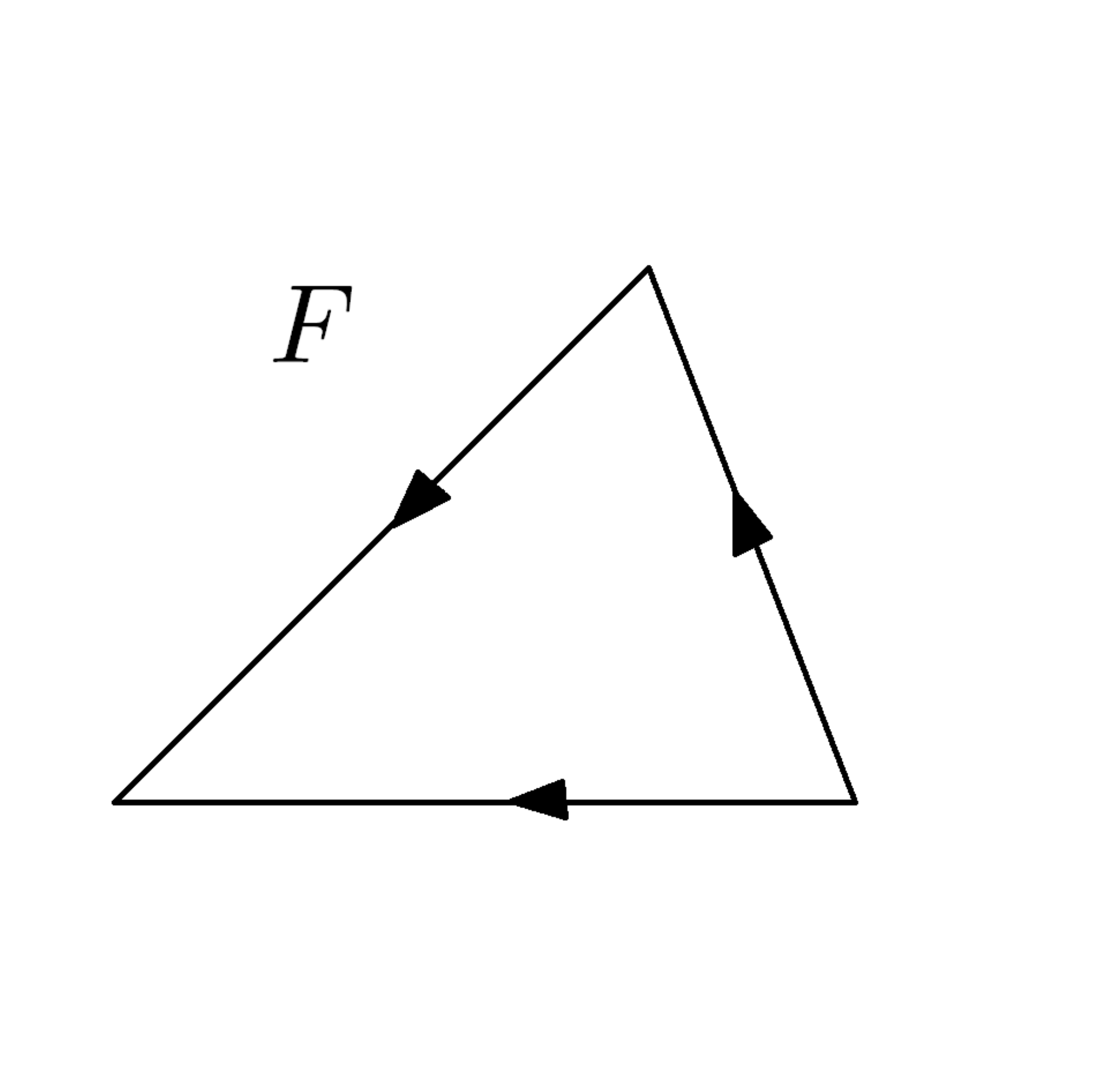} \hspace{1.2cm} \includegraphics[width=2.4cm]{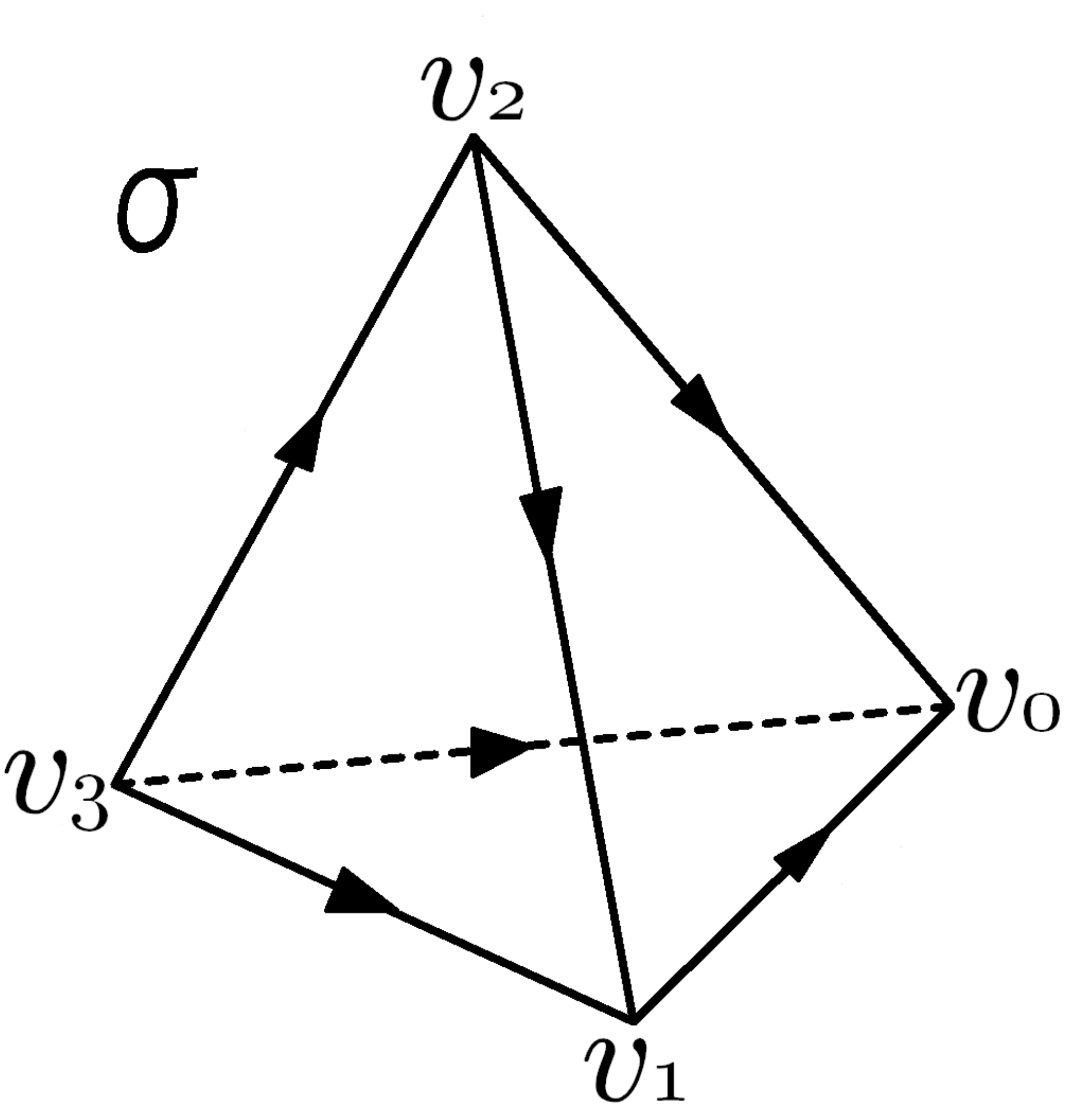}
\caption{A local order for a face and for a tetrahedron.}
\end{figure} 


Now we define the generalized DW invariant.  
Let $M$ be a compact or cusped 3-manifold,  $G$ a finite group and $\alpha \in Z^3(G,U(1))$.  
Fix a generalized ideal triangulation $K$ of $M$ with a local order.  
Then for each tetrahedron $\sigma$ of $K$ the sign $\epsilon_{\sigma}$ is determined by the local order.  
Put a coloring $\varphi$ of $K$, and then some element $\varphi(E)$ of $G$ is assigned to each oriented edge $E$ of each tetrahedron $\sigma$.  
We call $\varphi(E)$ the color of $E$ and such a tetrahedron $\sigma$ the colored tetrahedron, denoted by $(\sigma, \varphi)$. 

Let $v_0, v_1, v_2, v_3$ be the vertices of $\sigma$ with $v_0 < v_1 < v_2 < v_3$ by the local order ($v_i$ is incident to $i$ outgoing edges of $\sigma$).  
Put $\varphi(\langle v_0 v_1\rangle) = g$, $\varphi(\langle v_1 v_2\rangle) = h$, $\varphi(\langle v_2 v_3\rangle) = k$.  
Correspond $\alpha(g,h,k)^{\epsilon_{\sigma}} \in U(1)$ to the colored tetrahedron $(\sigma, \varphi)$.  
We call $W(\sigma, \varphi) = \alpha(g,h,k)^{\epsilon_{\sigma}}$ the symbol of the colored tetrahedron $(\sigma, \varphi)$.

\begin{figure}
\centering
\includegraphics[width=3cm]{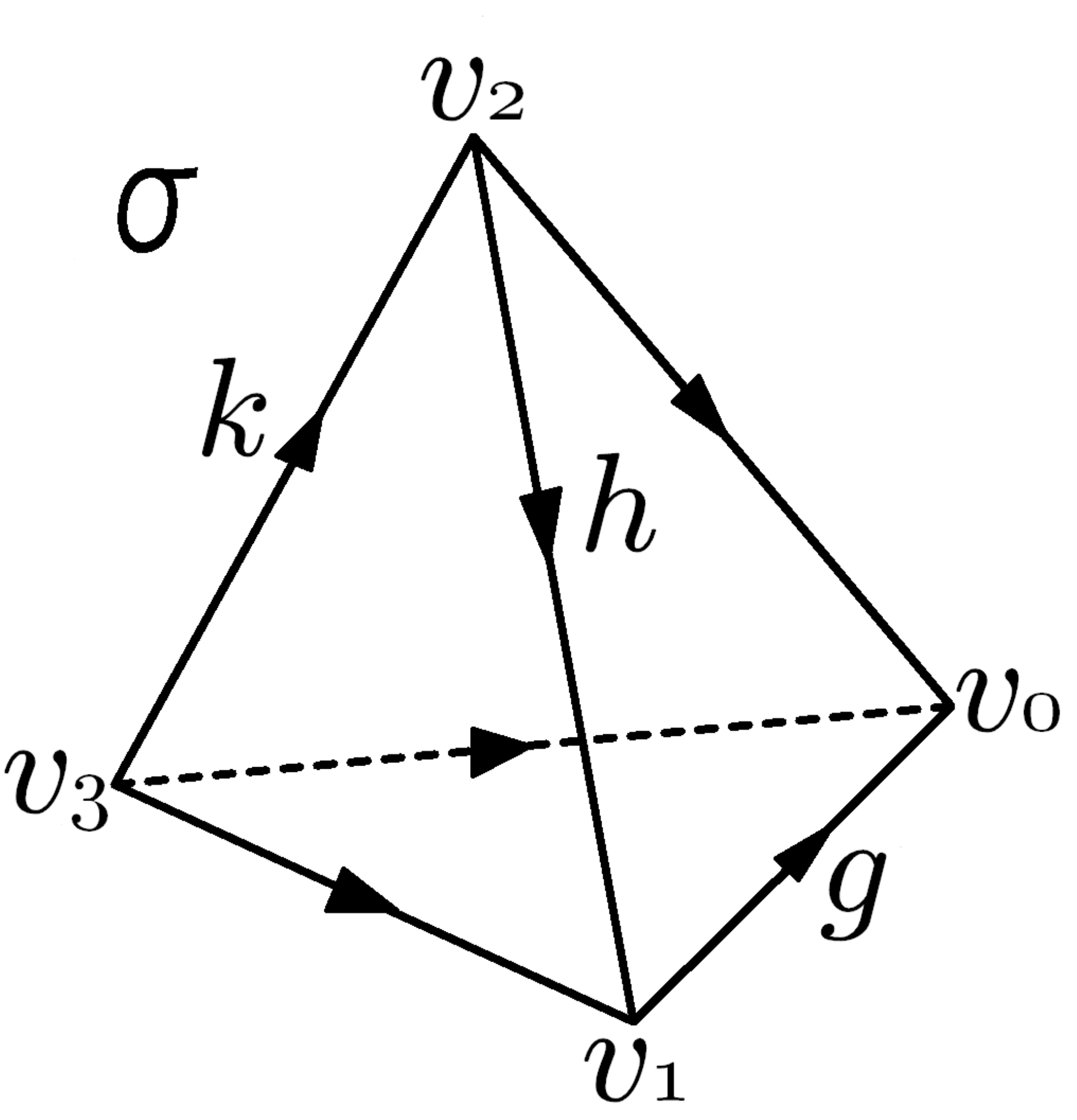}
\caption{A colored tetrahedron.}
\end{figure}

\begin{thm}
Let $M$ be a compact or cusped 3-manifold,  $G$ a finite group and $\alpha \in Z^3(G,U(1))$.  
Let $K$ be a generalized ideal triangulation of $M$ with a local order.  
Let $\sigma_1, \ldots ,\sigma_n$ be the tetrahedra of $K$ and $a$ the number of the interior vertices of $K$.
The generalized Dijkgraaf-Witten invariant $Z(M)$ is defined as follows:
$$
Z(M) = \frac{1}{|G|^a}\sum_{\varphi\in {\rm Col}(K)} \prod_{i=1}^{n} W(\sigma_i, \varphi).
$$
Then $Z(M)$ is independent of the choice of a generalized ideal triangulation $K$ of $M$ with a local order.
\end{thm}

By using a generalized ideal triangulation $K$ of $M$, each component of $\partial M$ corresponds to an ideal vertex of $K$.  
Hence, even if $\partial M \neq \emptyset$, the generalized DW invariant of $M$ does not need a triangulation of $\partial M$ nor its coloring.  
For a closed 3-manifold $M$,  since $K$ has no ideal vertices, 
the generalized DW invariant of $M$ is no other than the original DW invariant of $M$.

\begin{rem}
{\rm In general some generalized ideal triangulation} $K$ {\rm of} $M$ {\rm does not admit a local order}.  
{\rm Nevertheless the following lemma holds}.
\end{rem}

\begin{lem}
Any compact or cusped 3-manifold $M$ has a generalized ideal triangulation which admits a local order.
\end{lem}

\begin{proof}
For any given generalized ideal triangulation $K$ of $M$, let $K^{\prime \prime}$ be the generalized ideal triangulation of $M$ 
obtained by applying the barycentric subdivision twice to each tetrahedron of $K$.  
For given four vertices of $K^{\prime \prime}$ (which form a tetrahedron of $K^{\prime \prime}$), 
there exists a unique tetrahedron of $K^{\prime \prime}$ with the given four vertices.  
Hence $K^{\prime \prime}$ can be dealed in the same way as a simplicial triangulation of a closed 3-manifold.  
We choose an arbitrary total order on the set of the vertices of $K^{\prime \prime}$ and then the total order determines a local order of $K^{\prime \prime}$. 
\end{proof}

\section{Invariance of the generalized Dijkgraaf-Witten invariant}

In this section, we prove Theorem 2.2.  
First we show that $Z(M)$ is independent of the choice of a local order of a fixed generalized ideal triangulation $K$ of $M$.  
Then we prove that $Z(M)$ is independent of the choice of a generalized ideal triangulation $K$ of $M$.  

Let $K$ be a generalized ideal triangulation of $M$ with a local order.  
$\check{K}$ denotes the generalized ideal triangulation without considering a local order in this section.  
We define $Z(K)$ by
$$
Z(K) = \frac{1}{|G|^a}\sum_{\varphi\in {\rm Col}(K)} \prod_{i=1}^{n} W(\sigma_i, \varphi).
$$ 

\begin{lem}
Let $K_1$ and $K_2$ be generalized ideal triangulations with local orders of a compact or cusped 3-manifold $M$.  
If $\check{K_1} = \check{K_2}$, then $Z(K_1) = Z(K_2)$, 
i.e. $Z(K)$ is independent of the choice of a local order.
\end{lem}

\begin{proof}
Let $K$ be a generalized ideal triangulation of $M$ with a local order.  
Let $K^b$ be the generalized ideal triangulation of $M$ obtained by applying the barycentric subdivision once to each tetrahedron of $K$ 
with the following local order:

\noindent (vertex of $K$)  $<$ (midpoint of an edge of $K$)  $<$ (center of a face of $K$) 

\hspace{2cm} $<$ (center of a tetrahedron of $K$).

\noindent We call the above local order {\it the barycentric local order} in this paper.  
We prove that $Z(K) = Z(K^b)$, which implies the independence of the choice of a local order.  
We prove this claim by the following three steps. 

Step 1 : Divide each tetrahedron $\sigma$ of $K$ into four tetrahedra 
by adding four edges connecting the center of $\sigma$ (denoted by $b$) and (four) vertices of $\sigma$. 
This division is the number of the tetrahedra of $K$ times of (1,4)-Pachner moves.  See Figure 4.  
$K\rq$ denotes the generalized ideal triangulation of $M$ obtained by Step 1 with the local order 

\noindent (vertex of $K$) $<$ (center of a tetrahedron of $K$).

Step 2 : Divide each tetrahedron $\sigma$ of $K\rq$ into three tetrahedra by adding three edges as follows.  
$\sigma$ has three vertices of $K$ (the other vertex of $\sigma$ is $b$).  
Let $F$ be the face of $\sigma$ with three vertices of $K$.  
Add three edges connecting the center of $F$ (denoted by $c$) and (three) vertices of $F$.  See Figure 5.  
$K^{\prime \prime}$ denotes the generalized ideal triangulation of $M$ obtained by Step 2 with the local order 

\noindent (vertex of $K$) $<$ (center of a face of $K$) 

\hspace{2cm} $<$ (center of a tetrahedron of $K$).

Step 3 : Divide each tetrahedron $\sigma$ of $K^{\prime \prime}$ into two tetrahedra as follows.  
Let $v_0$, $v_1$ be two vertices of $\sigma$ which are vertices of $K$ (the other vertices of $\sigma$ are $b$ and $c$).  
Let $E$ be the edge of $\sigma$ connecting $v_0$ and $v_1$, and $d$ the midpoint of $E$.  
Divide $\sigma = \langle v_0v_1cb\rangle$ into $\langle v_0dcb\rangle$ and $\langle v_1dcb\rangle$.  See Figure 6.  
The generalized ideal triangulation of $M$ obtained by Step 3 is $\check{K^b}$.

Hence it suffices to show that $Z(K) = Z(K\rq) = Z(K^{\prime \prime}) = Z(K^b)$.  
We show these equalities in the above three steps.

\begin{figure}
\centering
\includegraphics[width=3.5cm]{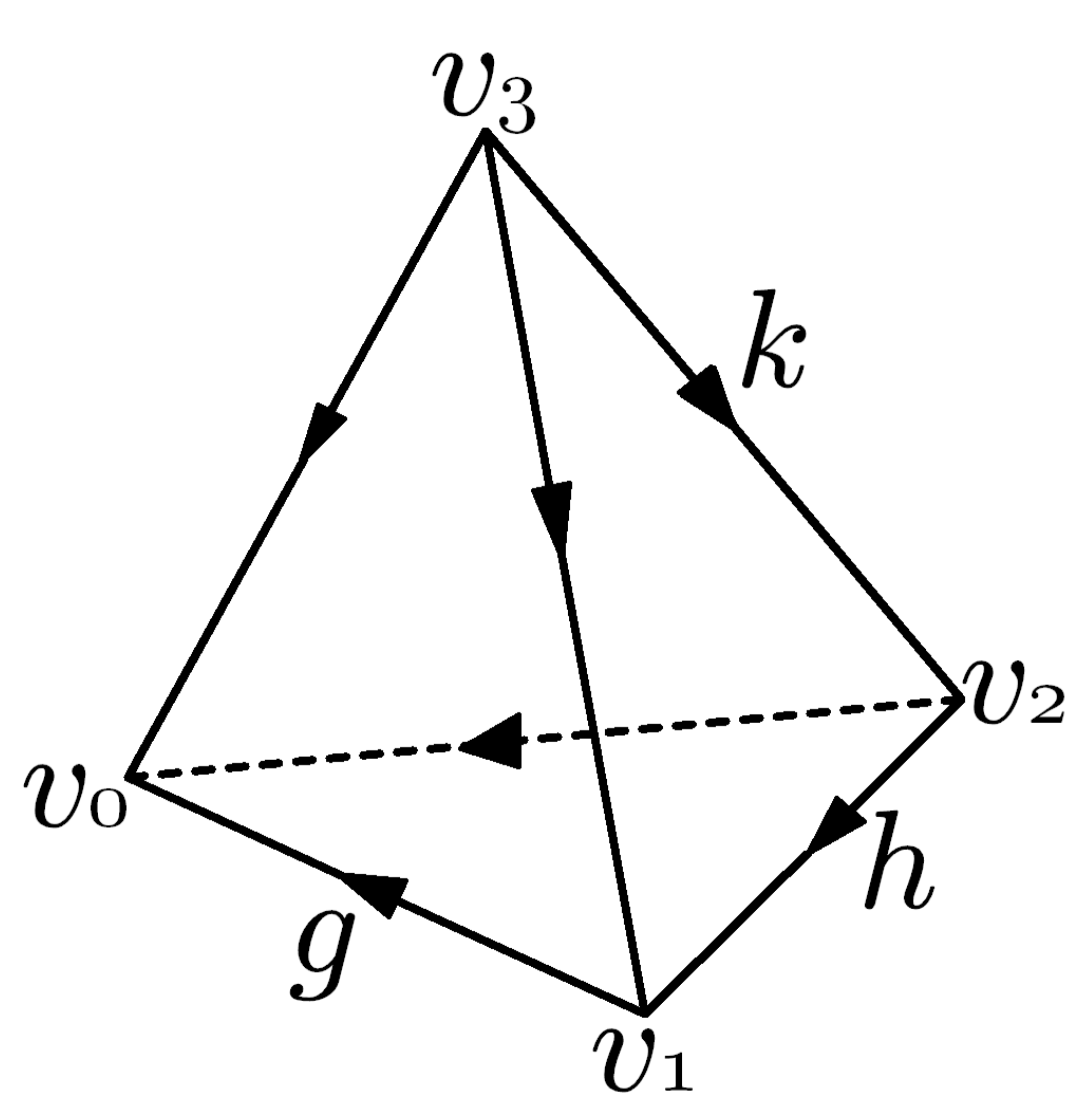} 
\includegraphics[width=1.2cm]{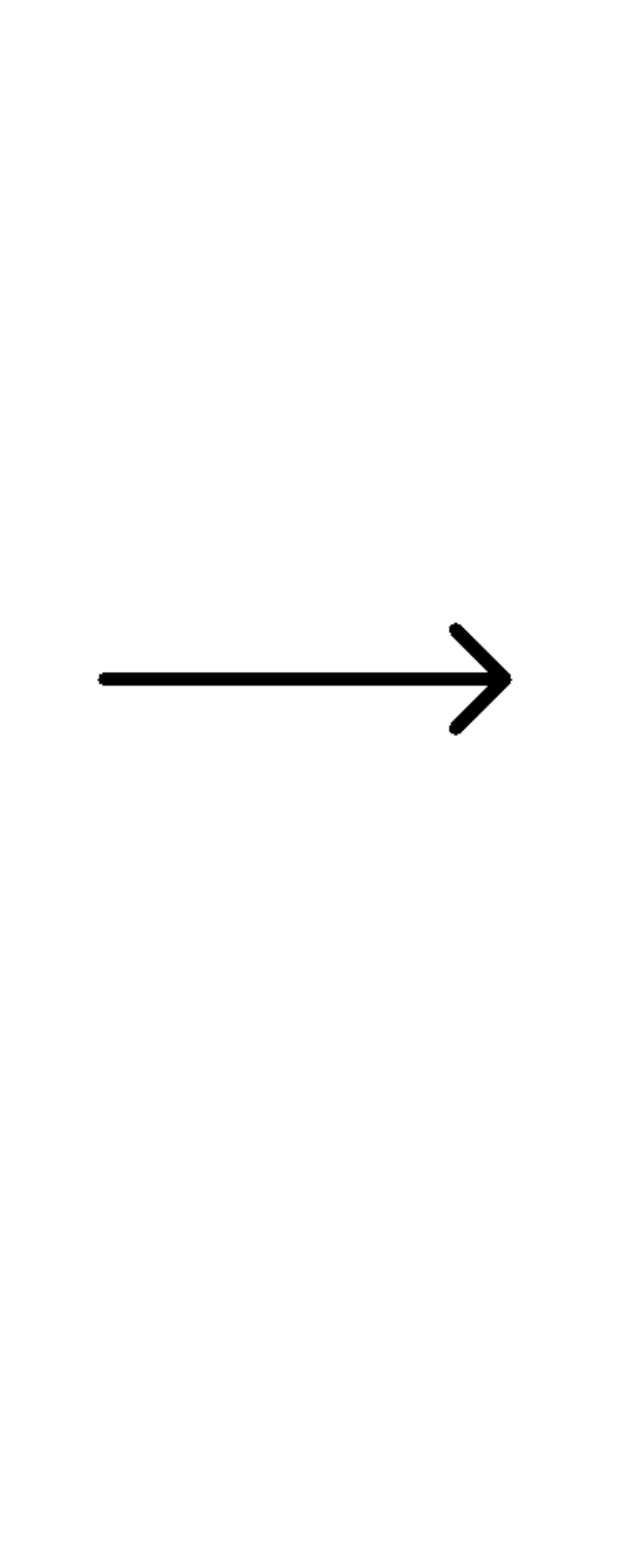}
\includegraphics[width=3.5cm]{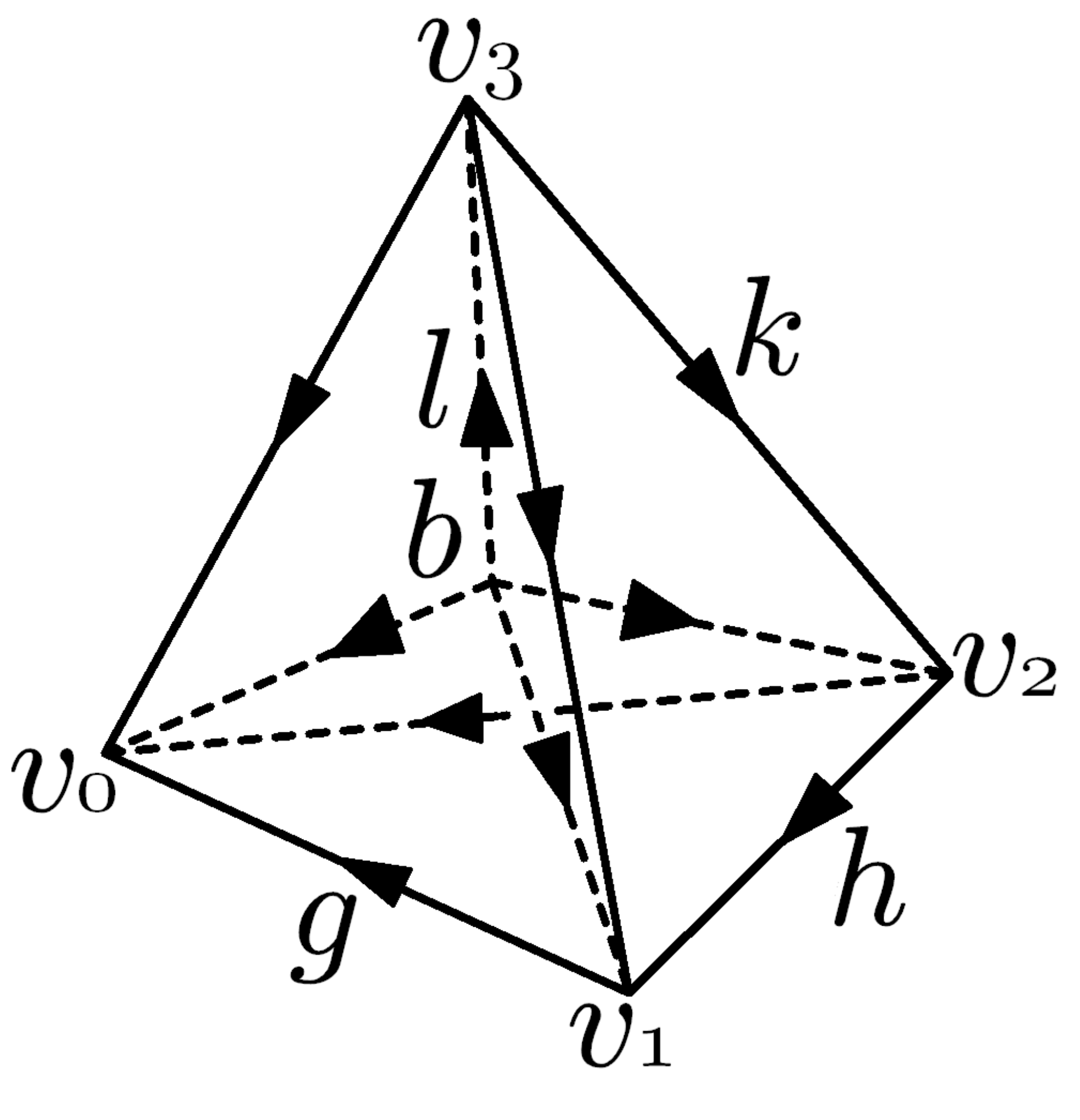}
\caption{The division in Step 1 ((1,4)-Pachner move).}
\end{figure}

Step 1: Set $\sigma = \langle v_0 v_1 v_2 v_3\rangle$, $\sigma_0 = \langle v_1 v_2 v_3 b\rangle$, $\sigma_1 = \langle v_0 v_2 v_3 b\rangle$, 
$\sigma_2 = \langle v_0 v_1 v_3 b\rangle$ and $\sigma_3 = \langle v_0 v_1 v_3 b\rangle$.  
Set $\epsilon_i = \epsilon_{\sigma_i}$.  Then $\epsilon_{\sigma} = -\epsilon_0 = \epsilon_1 = -\epsilon_2 = \epsilon_3$.  
The restriction of $\varphi \in {\rm Col}(K)$ on $\sigma$ is also denoted by $\varphi$ 
and Col($\sigma$) denotes the set of the restrictions of the colorings of $K$ on $\sigma$.  
$${\rm Col}(\sigma) \times G \ni (\varphi, l) \mapsto \varphi_l \in {\rm Col}(\sigma_0 \cup \sigma_1 \cup \sigma_2 \cup \sigma_3)$$
is a bijection, where $\varphi_l$ is the coloring determined by the following:
\[
 \varphi_l (E) = 
  \begin{cases}
   \varphi (E) & E {\rm \:is \:an \:edge \:of}\:K\\
   l & E = \langle v_3b\rangle.
  \end{cases}
\]

\noindent Therefore, it suffices to prove that for any $\varphi \in {\rm Col}(\sigma)$,
$$W(\sigma, \varphi) = \frac{1}{|G|}\sum_{l \in G} \prod_{i=0}^{3} W(\sigma_i, \varphi_l).$$

\noindent Set $\varphi(\langle v_0 v_1\rangle) = g$, $\varphi(\langle v_1 v_2\rangle) = h$, $\varphi(\langle v_2 v_3\rangle) = k$.  
The right hand side of the above formula equals to
$$\frac{1}{|G|}\sum_{l \in G} \alpha(h,k,l)^{-\epsilon_{\sigma}} \alpha(gh,k,l)^{\epsilon_{\sigma}} \alpha(g,hk,l)^{-\epsilon_{\sigma}} \alpha(g,h,kl)^{\epsilon_{\sigma}}$$
by the cocycle condition for ($g,h,k,l$),  
$$= \frac{1}{|G|}\sum_{l \in G} \alpha(g,h,k)^{\epsilon_{\sigma}} \hspace{7cm}$$
$$= \alpha(g,h,k)^{\epsilon_{\sigma}} = W(\sigma, \varphi).\hspace{6.2cm}$$

\noindent As this equality holds for each tetrahedron $\sigma$ of $K$, $Z(K) = Z(K\rq)$ holds.

Step 2: Set $\sigma = \langle v_0 v_1 v_2 b\rangle$, $\sigma_0 = \langle v_1 v_2 c b\rangle$, $\sigma_1 = \langle v_0 v_2 c b\rangle$, 
$\sigma_2 = \langle v_0 v_1 c b\rangle$.  
$\epsilon_{\sigma} = \epsilon_0 = -\epsilon_1 = \epsilon_2$.  
$${\rm Col}(\sigma) \times G \ni (\varphi, l) \mapsto \varphi_l \in {\rm Col}(\sigma_0 \cup \sigma_1 \cup \sigma_2)$$
is a bijection, where $\varphi_l$ is the coloring determined by the following:
\[
 \varphi_l (E) = 
  \begin{cases}
   \varphi (E) & E {\rm \:is \:an \:edge \:of}\:K\rq\\
   l & E = \langle c b\rangle.
  \end{cases}
\]

\noindent In the same way as Step 1, for any $\varphi \in {\rm Col}(\sigma)$, we compare 
$$W(\sigma, \varphi) \quad{\rm and}\quad \frac{1}{|G|}\sum_{l \in G} \prod_{i=0}^{2} W(\sigma_i, \varphi_l).$$

\noindent Set $\varphi(\langle v_0 v_1\rangle) = g$, $\varphi(\langle v_1 v_2\rangle) = h$, $\varphi(\langle v_2 b\rangle) = k$.  The right hand side equals to
$$\frac{1}{|G|}\sum_{l \in G} \alpha(h,kl^{-1},l)^{\epsilon_{\sigma}} \alpha(gh,kl^{-1},l)^{-\epsilon_{\sigma}} \alpha(g,hkl^{-1},l)^{\epsilon_{\sigma}}.$$

\noindent By the cocycle condition for ($g,h,kl^{-1},l$),
$$\alpha (h,kl^{-1},l) \alpha (gh,kl^{-1},l)^{-1} \alpha (g,hkl^{-1},l) = \alpha (g,h,k) \alpha (g,h,kl^{-1})^{-1}.$$

\begin{figure}
\centering
\includegraphics[width=3.5cm]{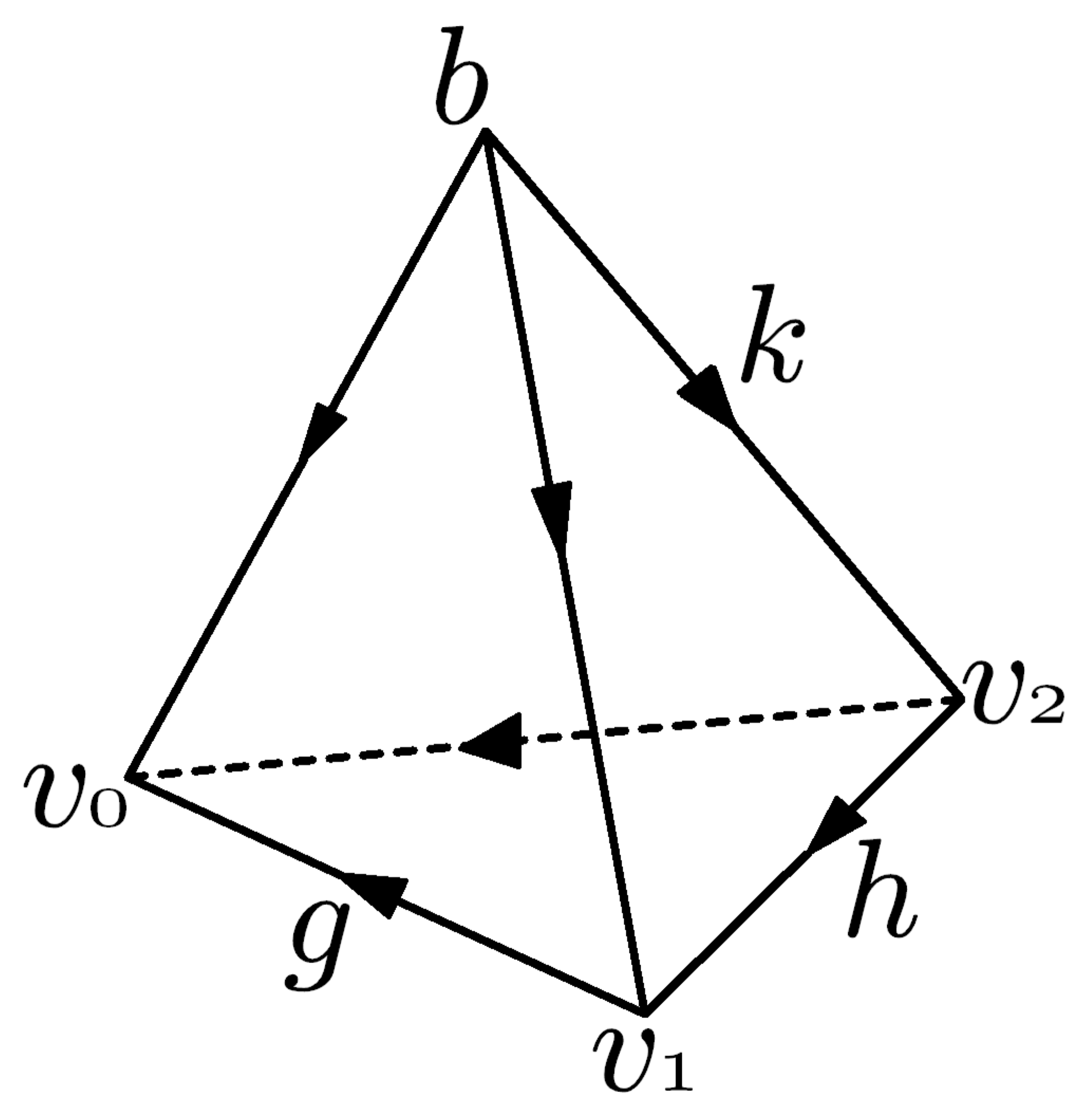} 
\includegraphics[width=1.2cm]{arrow1a3.pdf}
\includegraphics[width=3.5cm]{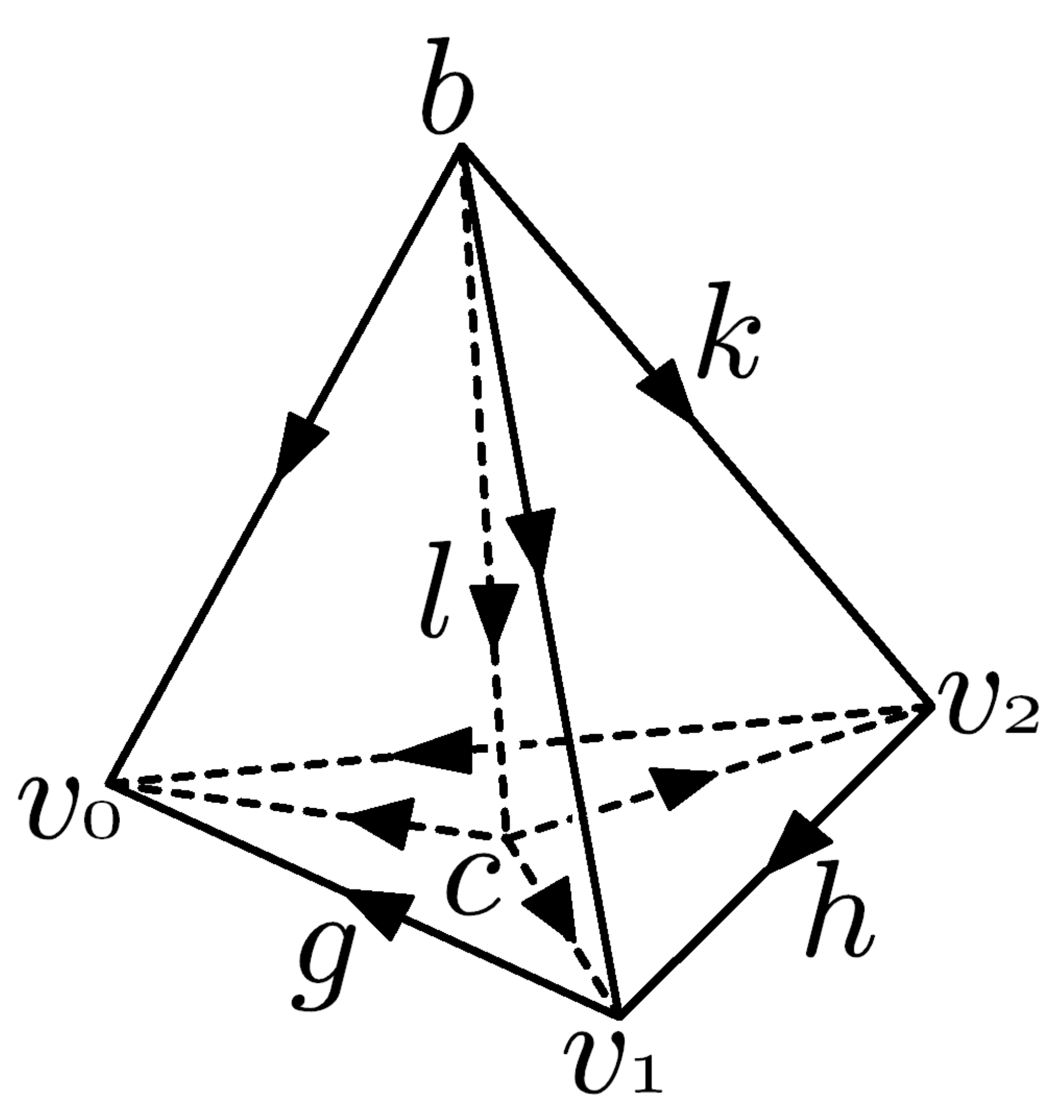}
\caption{The division in Step 2.}
\end{figure}

\noindent As $W(\sigma, \varphi) = \alpha (g,h,k)^{\epsilon_{\sigma}}$, 
the equality does not hold as that in Step 1 does because of the extra facor $\alpha (g,h,kl^{-1})^{-1}$.  
$\sigma\rq$ denotes the other tetrahedron which shares the face $\langle v_0 v_1 v_2\rangle$ and the other vertex of $\sigma\rq$ is denoted by $b\rq$. 
($b\rq$ is the center of some tetrahedron of $K$ and $b\rq$ may coincide with $b$.)  
The three tetrahedra obtained by the division of $\sigma\rq$ are denoted by $\sigma\rq_0, \sigma\rq_1$ and $\sigma\rq_2$.  
Consider the symbol of $\sigma\rq$ and the product of the symbols of $\sigma\rq_0, \sigma\rq_1$ and $\sigma\rq_2$, 
and then the same extra factor $\alpha (g,h,kl^{-1})^{-1}$ appears. ($kl^{-1}$ is the color of $\langle v_2 c\rangle$.)  
Since $\epsilon_{\sigma} = -\epsilon_{\sigma\rq}$, the extra factor of $\sigma$ and that of $\sigma\rq$ are exactly cancaled out.  Therefore
$$W(\sigma, \varphi) W(\sigma\rq, \varphi) = \frac{1}{|G|}\sum_{l \in G} \prod_{i=0}^{2} W(\sigma_i, \varphi_l) W(\sigma\rq_i, \varphi_l).$$
As this equality holds for each pair of tetrahedra $\sigma$ and $\sigma\rq$ of $K\rq$ which shares a face of $K$, 
$Z(K\rq) = Z(K^{\prime \prime})$ holds.

\begin{figure}
\centering
\includegraphics[width=3.5cm]{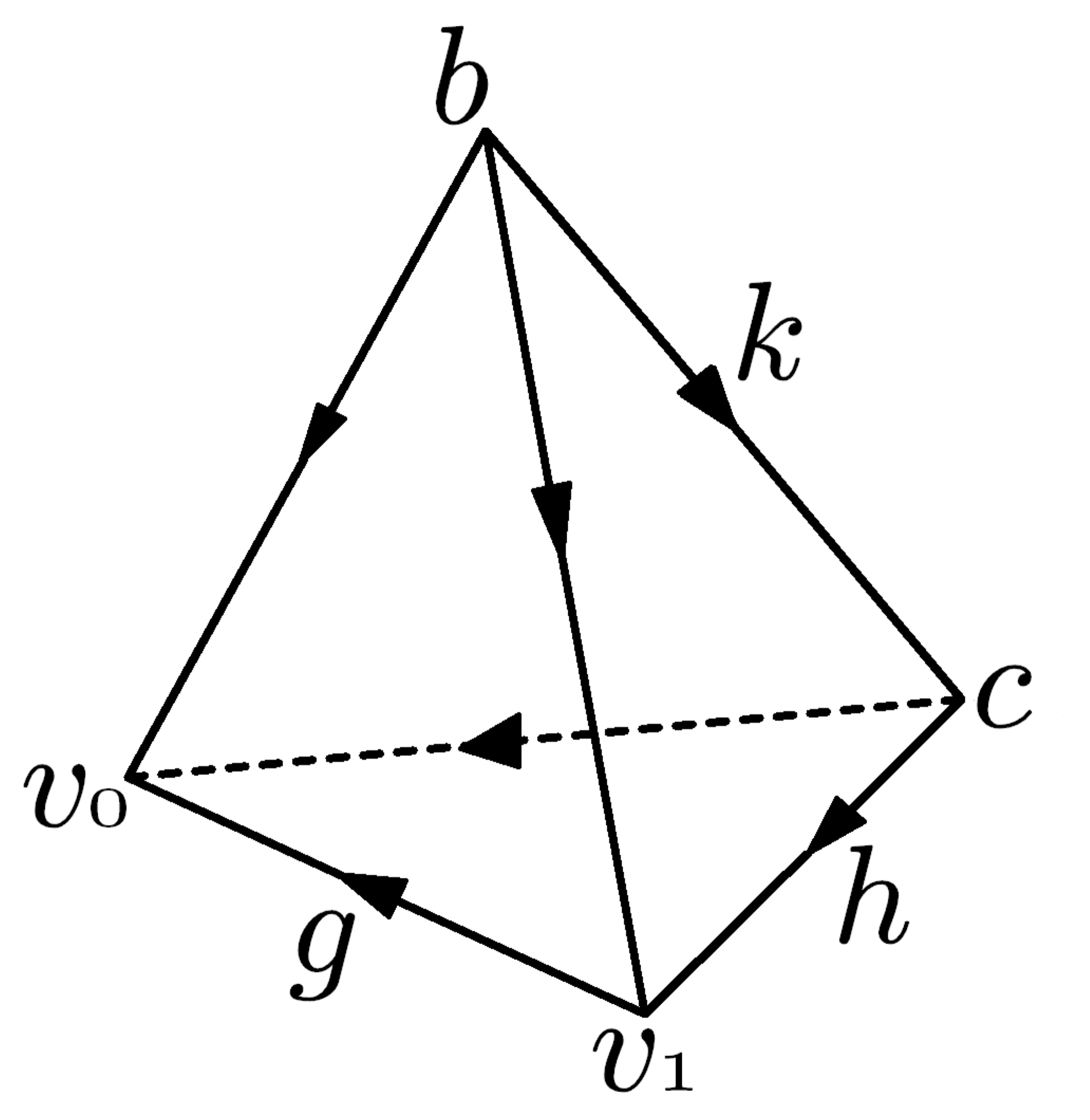} 
\includegraphics[width=1.2cm]{arrow1a3.pdf}
\includegraphics[width=3.5cm]{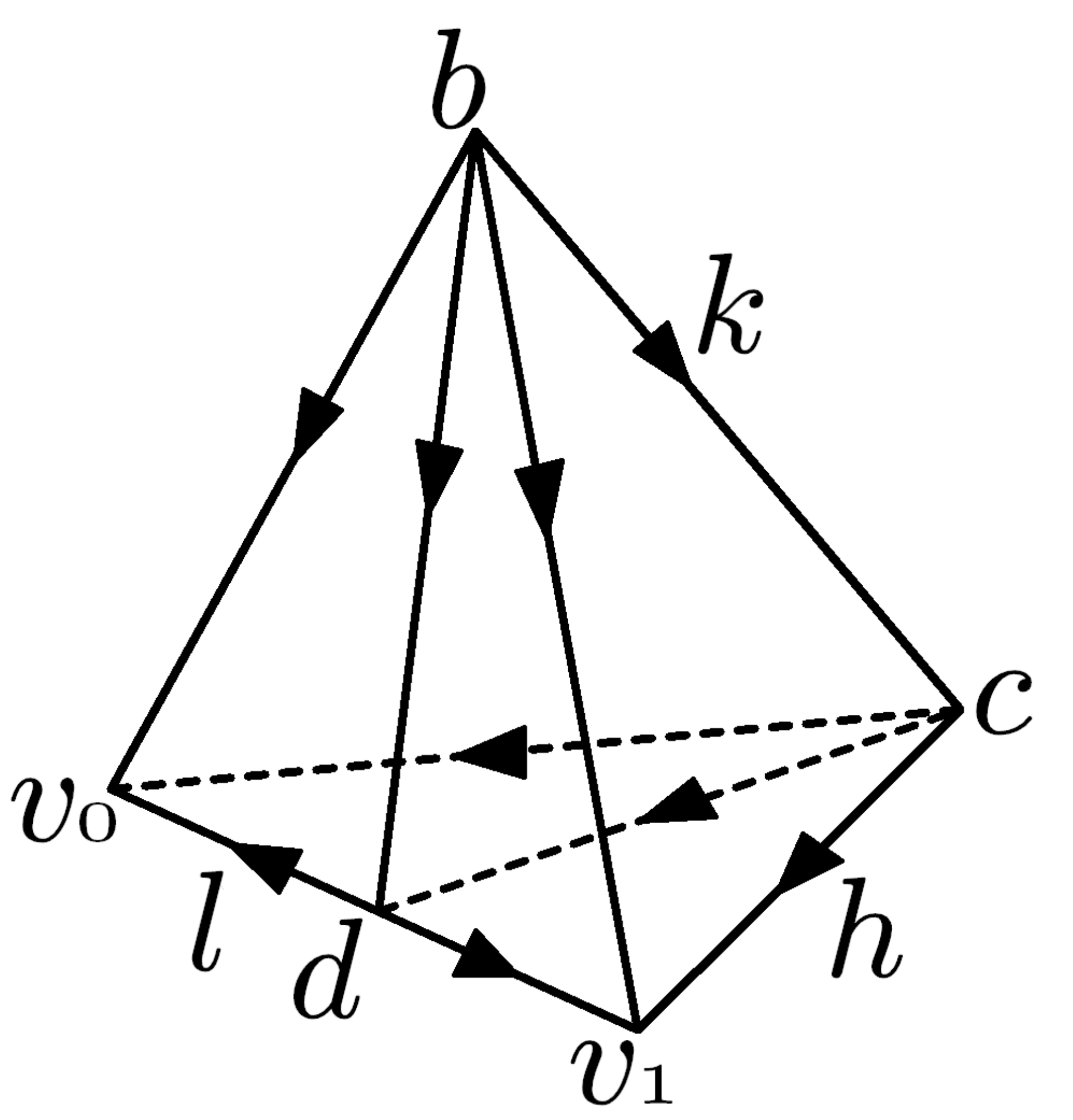}
\caption{The division in Step 3.}
\end{figure}

Step 3: Set $\sigma = \langle v_0 v_1 c b\rangle$, $\sigma_0 = \langle v_0 d c b\rangle$, $\sigma_1 = \langle v_1 d c b\rangle$.  
$\epsilon_{\sigma} = \epsilon_0 = -\epsilon_1$.  
$${\rm Col}(\sigma) \times G \ni (\varphi, l) \mapsto \varphi_l \in {\rm Col}(\sigma_0 \cup \sigma_1)$$
is a bijection, where $\varphi_l$ is the coloring determined by the following:
\[
 \varphi_l (E) = 
  \begin{cases}
   \varphi (E) & E {\rm \:is \:an \:edge \:of}\:K^{\prime \prime}\\
   l & E = \langle v_0 d\rangle.
  \end{cases}
\]

\noindent For any $\varphi \in {\rm Col}(\sigma)$, we compare 
$$W(\sigma, \varphi) \quad{\rm and}\quad \frac{1}{|G|}\sum_{l \in G} \prod_{i=0}^{1} W(\sigma_i, \varphi_l).$$

\noindent Set $\varphi(\langle v_0 v_1\rangle) = g$, $\varphi(\langle v_1 c\rangle) = h$, $\varphi(\langle c b\rangle) = k$. The right hand side equals to
$$\frac{1}{|G|}\sum_{l \in G} \alpha(l,l^{-1}gh,k)^{\epsilon_{\sigma}} \alpha(g^{-1}l,l^{-1}gh,k)^{-\epsilon_{\sigma}}.$$

\noindent By the cocycle condition for ($l,l^{-1}g,h,k$) and ($g^{-1}l,l^{-1}g,h,k$),
$$\alpha (l^{-1}g,h,k) \alpha (l,l^{-1}gh,k) = \alpha (g,h,k) \alpha (l,l^{-1}g,hk) \alpha (l,l^{-1}g,h)^{-1},$$
$$\alpha(l^{-1}g,h,k) \alpha(g^{-1}l,l^{-1}gh,k) = \alpha(g^{-1}l,l^{-1}g,hk) \alpha(g^{-1}l,l^{-1}g,h)^{-1}.$$
Hence, 
$$\alpha(l,l^{-1}gh,k) \alpha(g^{-1}l,l^{-1}gh,k)^{-1} =$$ 
$$\alpha (g,h,k) \alpha (l,l^{-1}g,hk) \alpha(g^{-1}l,l^{-1}g,hk)^{-1} \alpha (l,l^{-1}g,h)^{-1} \alpha(g^{-1}l,l^{-1}g,h).$$

\noindent Since $W(\sigma, \varphi) = \alpha (g,h,k)^{\epsilon_{\sigma}}$, the extra factor is 
$$\alpha (l,l^{-1}g,hk)\alpha(g^{-1}l,l^{-1}g,hk)^{-1}\alpha (l,l^{-1}g,h)^{-1}\alpha(g^{-1}l,l^{-1}g,h).$$ 
$\alpha (l,l^{-1}g,hk) \alpha(g^{-1}l,l^{-1}g,hk)^{-1}$ is regarded as 
the contribution of the face $\langle v_0 v_1 b\rangle$ since $hk$ is the color of $\langle v_1 b\rangle$.  
Namely, the extra factor of the pair ($\sigma$, $\langle v_0 v_1 b\rangle$) is $\alpha (l,l^{-1}g,hk) \alpha(g^{-1}l,l^{-1}g,hk)^{-1}$.  
Let $\sigma\rq$ be the other tetrahedron which shares the face $\langle v_0 v_1 b\rangle$ and the other vertex of $\sigma\rq$ is denoted by $b\rq$.  
As $b\rq$ is the center of some tetrahedron of $K$, the extra factor of ($\sigma\rq$, $\langle v_0 v_1 b\rangle$) 
is also $\alpha (l,l^{-1}g,hk) \alpha(g^{-1}l,l^{-1}g,hk)^{-1}$. 
Since $\epsilon_{\sigma} = -\epsilon_{\sigma\rq}$, the extra factors of ($\sigma$, $\langle v_0 v_1 b\rangle$) and 
($\sigma\rq$, $\langle v_0 v_1 b\rangle$) are canceled out.  
Similarly, $\alpha (l,l^{-1}g,h)^{-1} \alpha(g^{-1}l,l^{-1}g,h)$ is the extra factor of ($\sigma$, $\langle v_0 v_1 c\rangle$). 
Let $\sigma^{\prime \prime}$ be the other tetrahedron which shares the face $\langle v_0 v_1 c\rangle$ 
and the other vertex of $\sigma^{\prime \prime}$ is denoted by $c^{\prime \prime}$.   
As $c^{\prime \prime}$ is the center of some face of $K$, 
the extra factor of ($\sigma^{\prime \prime}$, $\langle v_0 v_1 c\rangle$) is $\alpha (l,l^{-1}g,h)^{-1} \alpha(g^{-1}l,l^{-1}g,h)$. 
Due to $\epsilon_{\sigma} = -\epsilon_{\sigma^{\prime \prime}}$, 
the extra factors of ($\sigma$, $\langle v_0 v_1 c\rangle$) and ($\sigma^{\prime \prime}$, $\langle v_0 v_1 c\rangle$) are canceled out.  
Let $\sigma^1, \ldots ,\sigma^m$ be the tetrahedra of $K^{\prime \prime}$ which shares the edge $\langle v_0 v_1\rangle$.  
The two tetrahedra obtained by the division of $\sigma^i$ are denoted by $\sigma^i_0$ and $\sigma^i_1$.  
Then,
$$\prod_{i=1}^{m} W(\sigma^i, \varphi) = \frac{1}{|G|}\sum_{l \in G} \prod_{i=1}^{m} \prod_{j=0}^{1} W(\sigma^i_j, \varphi_l).$$
As this equality holds for each set of tetrahedra of $K^{\prime \prime}$ which shares a edge of $K$, 
$Z(K^{\prime \prime}) = Z(K^b)$ holds.
\end{proof}

Next we prove that $Z(M)$ is independent of the choice of a generalized ideal triangulation $K$ of $M$.  
In order to show that, we make use of the following theorem by Pachner.
  
\begin{thm}[Pachner]
Any two triangulations of a 3-manifold $M$ can be transformed one to another by a finite sequence of the following two types of transformations shown in Figure 7. 
\end{thm}

\begin{figure}[h] 
\centering
(1,4)-Pachner move \hspace{3.2cm} (2,3)-Pachner move

\includegraphics[width=2.3cm]{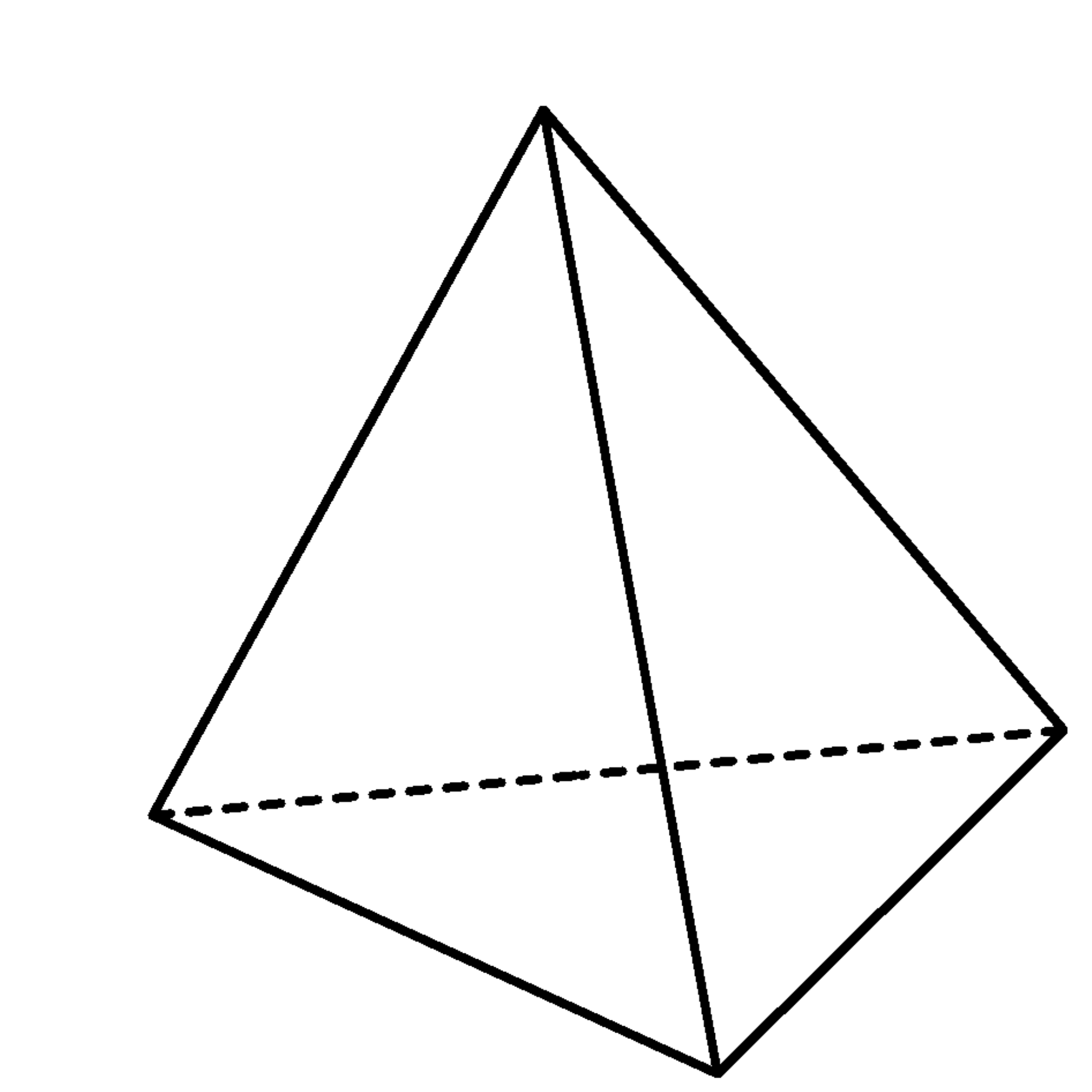}
\includegraphics[width=0.8cm]{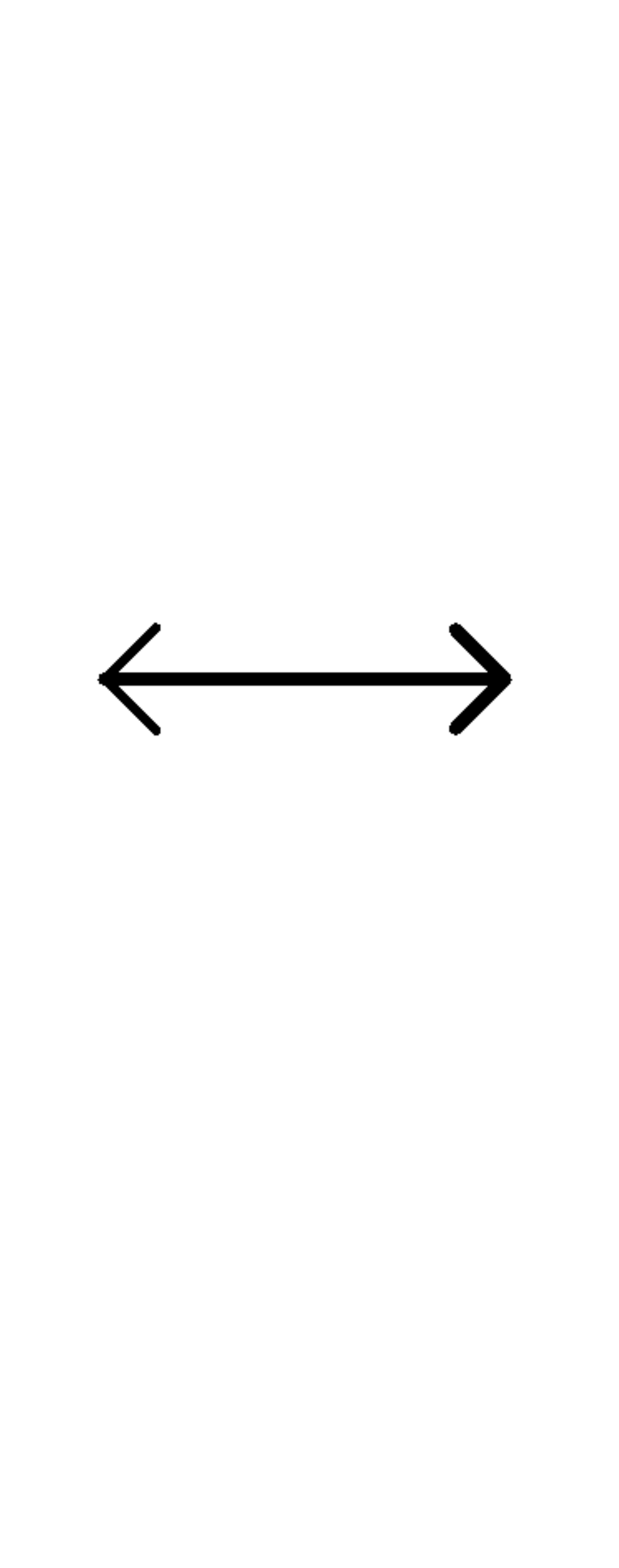}
\includegraphics[width=2.3cm]{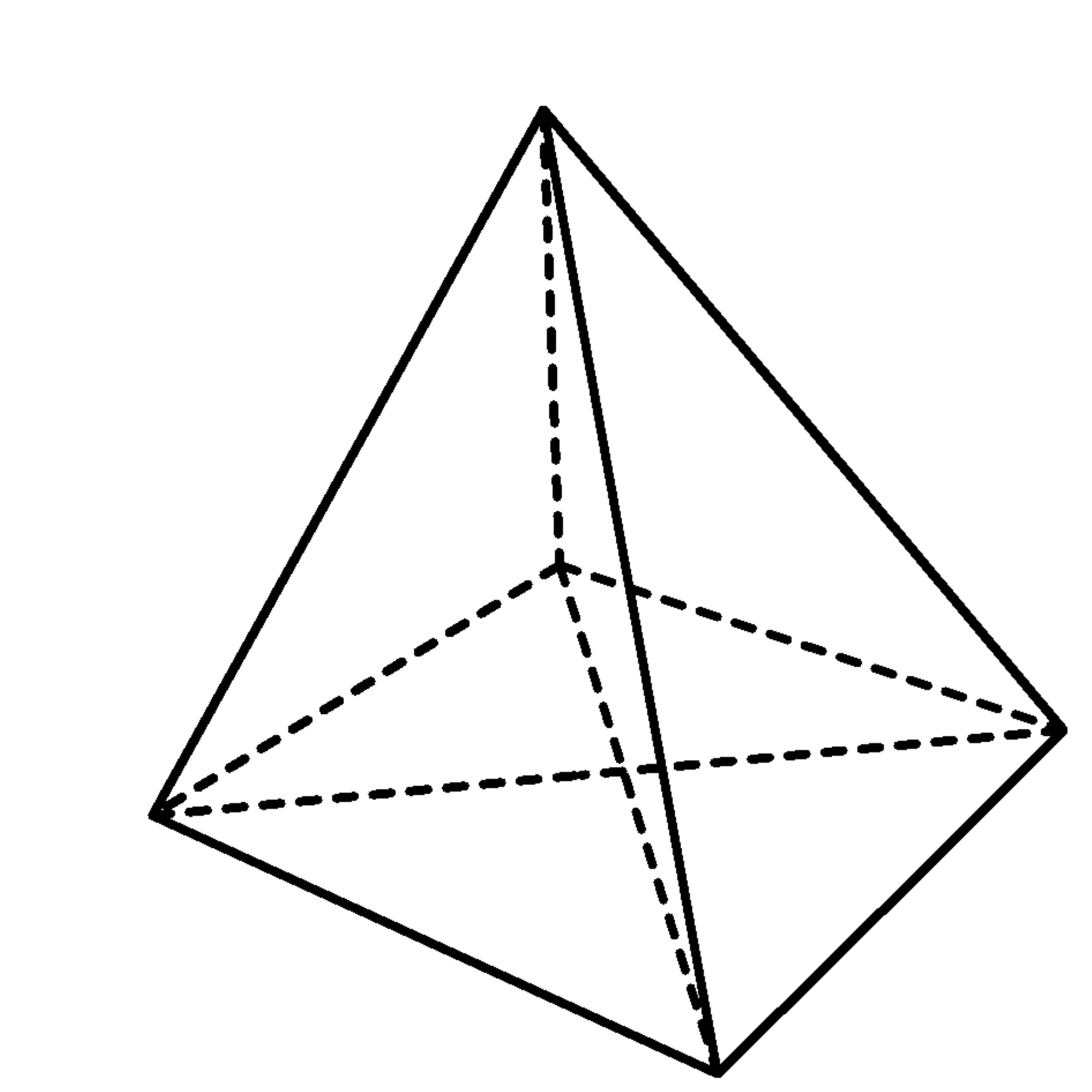}\hspace{0.75cm}
\includegraphics[width=2.3cm]{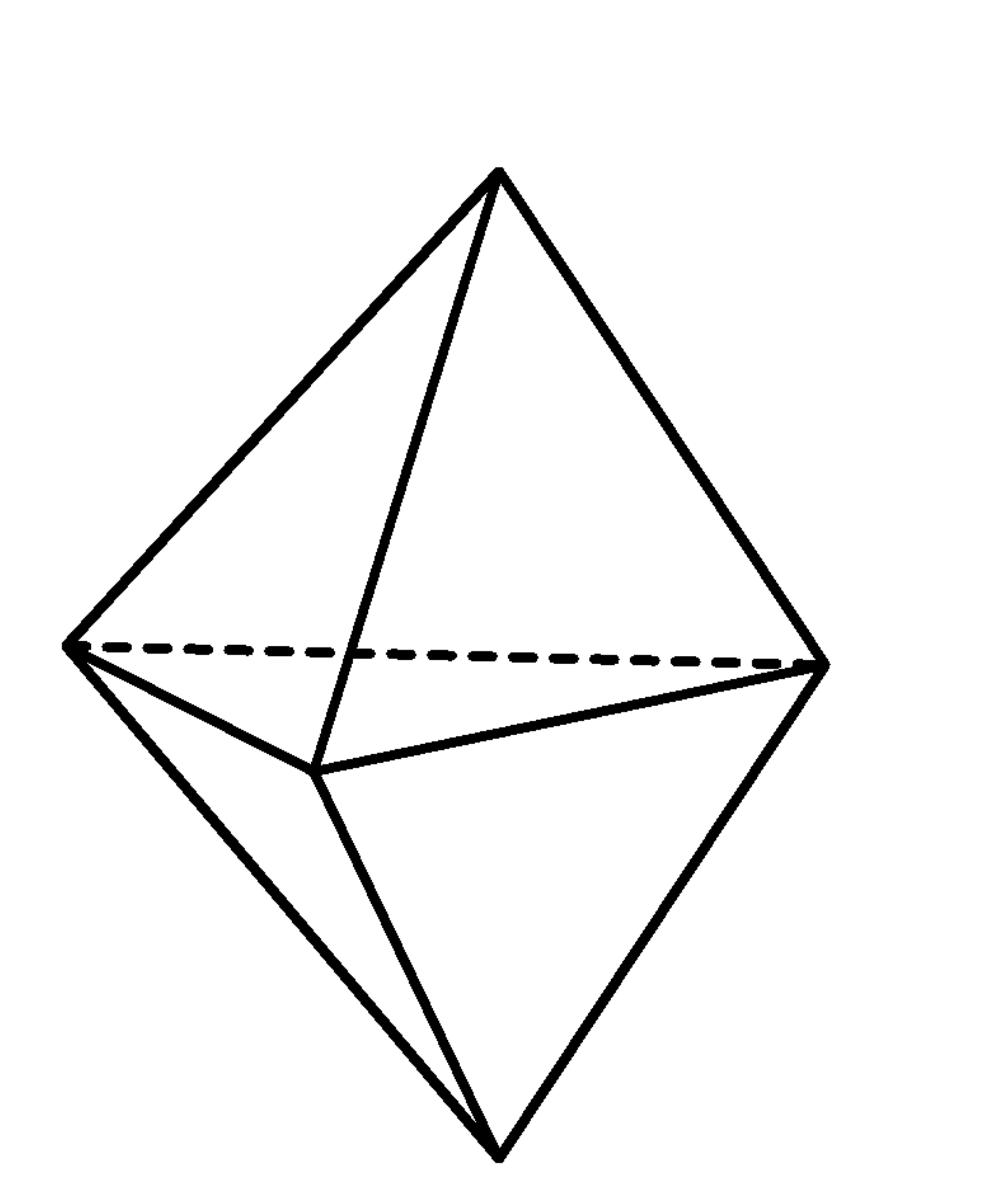}
\includegraphics[width=0.8cm]{arrow2a1.pdf}\hspace{0.15cm}
\includegraphics[width=2.3cm]{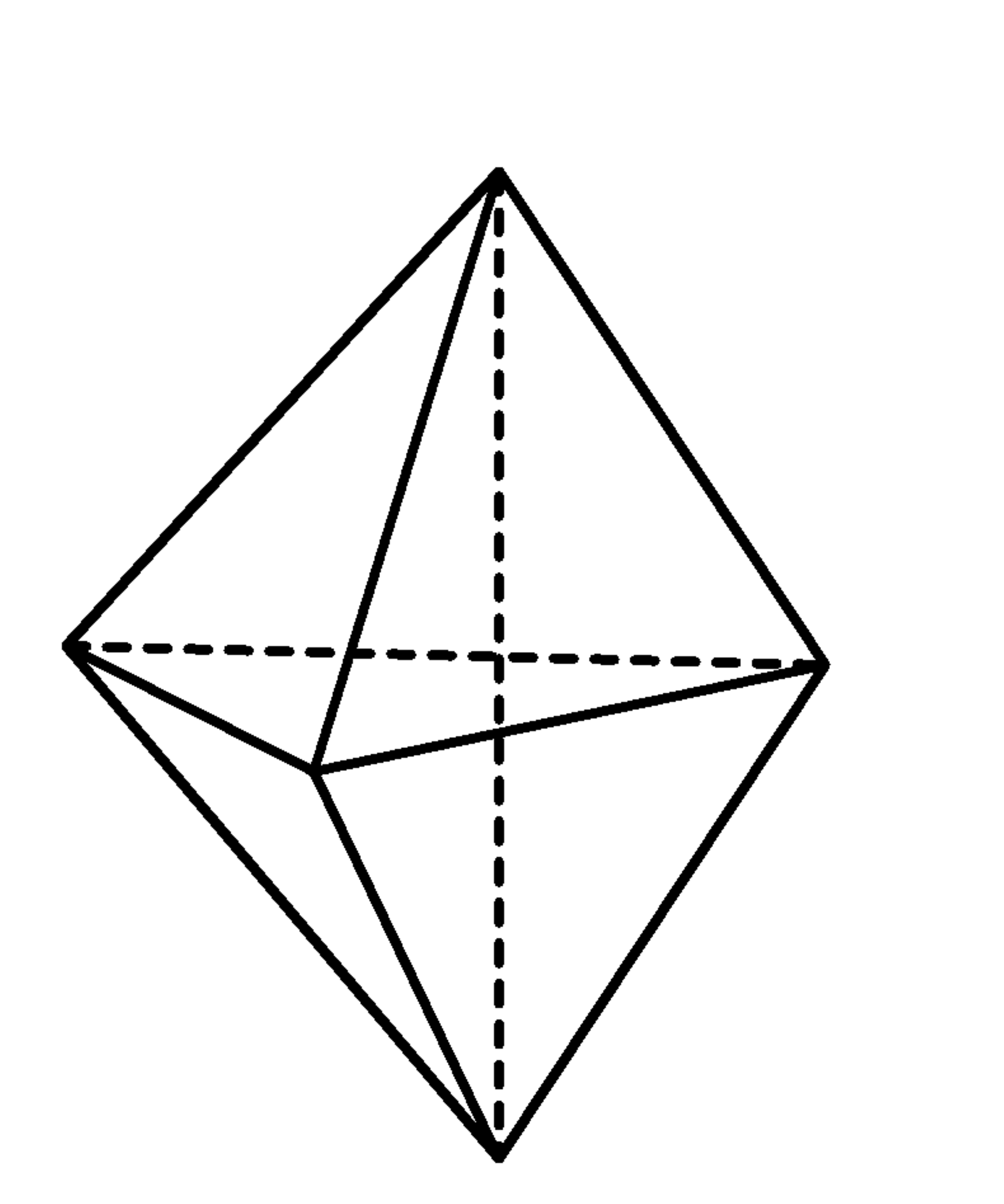}
\caption{The Pachner moves.}
\end{figure}

Let $K$ and $L$ be any two generalized ideal triangulations of $M$.  
Owing to Lemma 3.1, $Z(K) = Z(L)$ implies Theorem 2.2.  
By Theorem 3.2, there exists a finite sequence of generalized ideal triangulations of $M$, 
$K = K_0 \rightarrow K_1 \rightarrow \cdots \rightarrow K_n = L$, 
such that $K_i$ is transformed to $K_{i+1}$ by one of Pachner moves once.  
However, for some $i$, $K_i$ may not admit a local order.  

In order to avoid this problm, we take a barycentric subdivision again.  
Let $K^b$ and $K^{bb}$ be generalized ideal triangulations of $M$ obtained by applying the 
barycentric subdivision to each tetrahedron of $K$ once and twice respectively.  
Even though $K$ does not admit a local order, $K^b$ always admits a local order, for example, the barycentric local order.  
Furthermore $K^{bb}$ can be dealed in the same way as a simplicial triangulation of a closed 3-manifold.  
We call such a generalized ideal triangulation {\it a generalized ideal simplicial triangulation}.
For sufficiently large positive integers $p$ and $q$, there exists a finite sequence of generalized ideal simplicial triangulations of $M$, 
$K^p = K\rq_0 \rightarrow K\rq_1 \rightarrow \cdots \rightarrow K\rq_m = L^q$, 
where $K^p$ and $L^q$ are generalized ideal simplicial triangulations of $M$ 
obtained by applying the barycentric subdivision to each tetrahedron of $K$ $p$ times and of $L$ $q$ times respectively, 
such that $K\rq_i$ is transformed to $K\rq_{i+1}$ by one of Pachner moves once.  

Therefore it suffices to prove that for any generalized ideal simplicial triangulation $K$ of $M$, 
$Z(K)$ (with any fixed local order of $K$) equals to $Z(K\rq)$ (with any fixed local order of $K\rq$), 
where $K\rq$ is the generalized ideal simplicial triangulation of $M$ obtained by one of Pachner moves once from $K$.  
Furthermore we can take a local order determined by a total order on the set of the vertices in the proof of the following lemmas.  
The following two lemmas are given in \cite{Wakui2}.    


\begin{lem}
$Z(M)$ is invariant under a {\rm(1,4)}-Pachner move.
\end{lem}

\begin{proof}
We have already proved in Step 1 of Lemma 3.1.  
\end{proof}

\begin{lem}
$Z(M)$ is invariant under a {\rm(2,3)}-Pachner move.
\end{lem}

\begin{figure}[h]
\centering
\includegraphics[width=3.5cm]{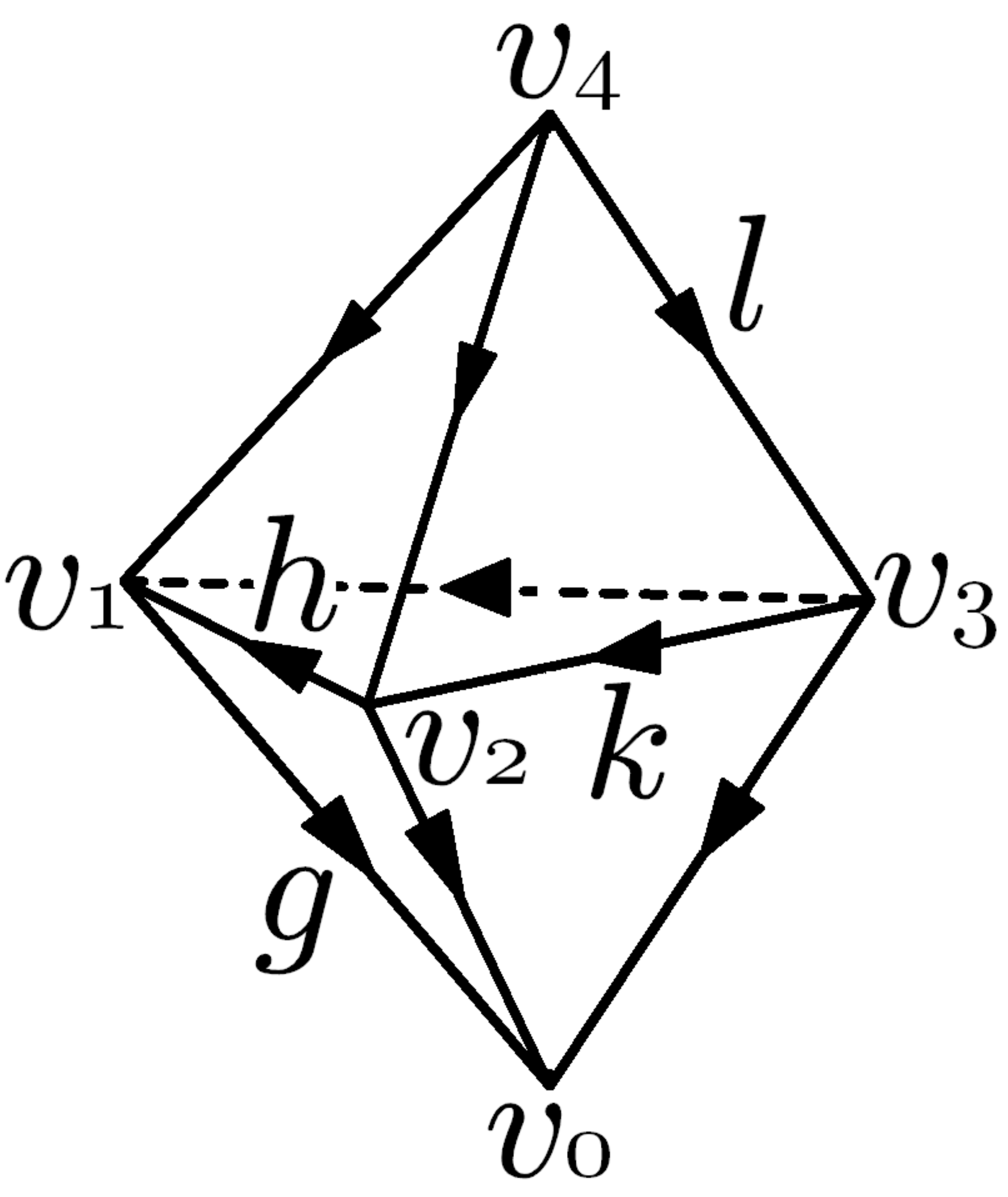} 
\includegraphics[width=1.2cm]{arrow1a3.pdf}
\includegraphics[width=3.5cm]{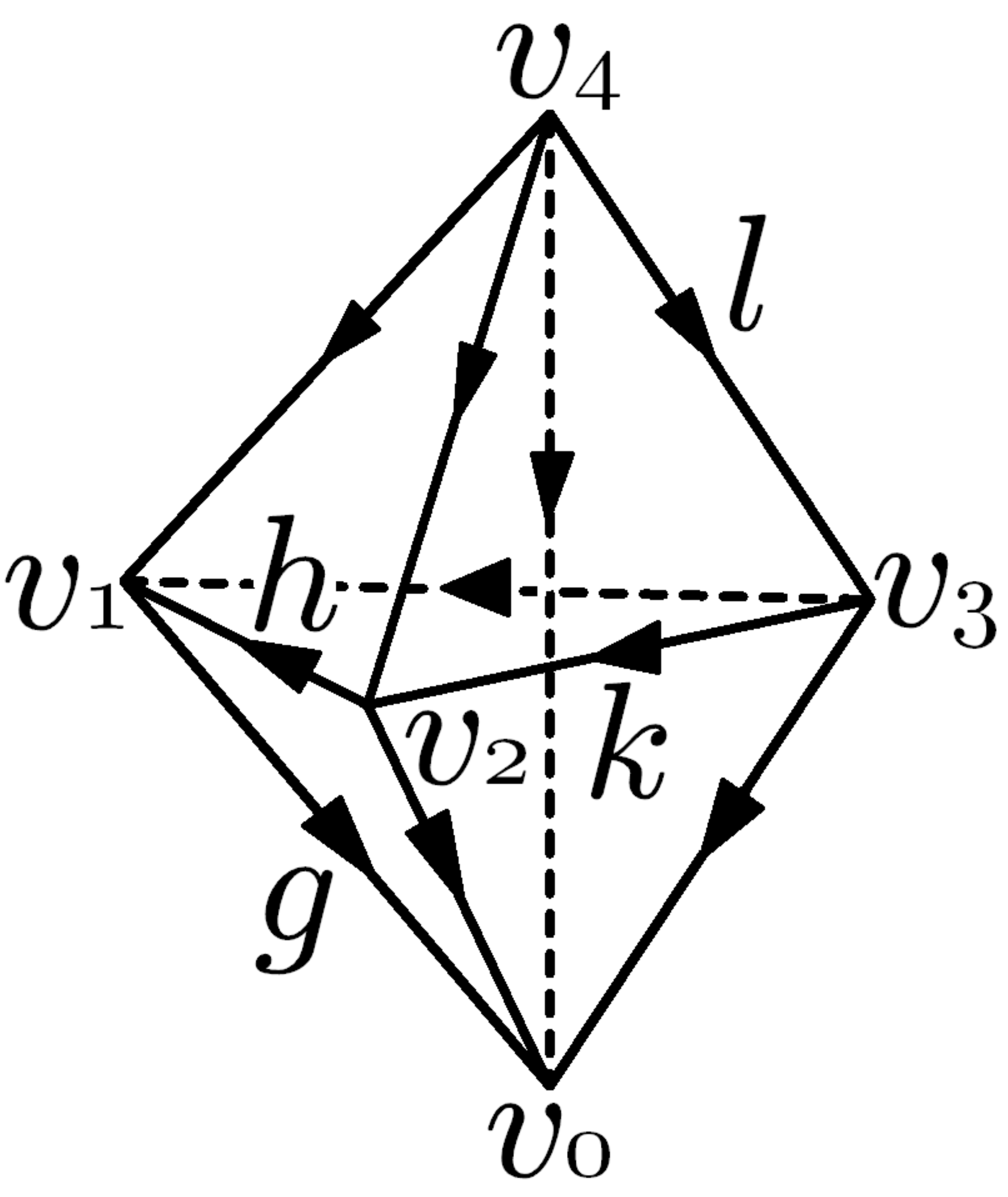}
\caption{The local order for a (2,3)-Pachner move.}
\end{figure}

\begin{proof}
Set $\sigma = \langle v_0 v_1 v_2 v_3\rangle$, $\sigma\rq = \langle v_1 v_2 v_3 v_4\rangle$, $\sigma_1 = \langle v_0 v_2 v_3 v_4\rangle$, 
$\sigma_2 = \langle v_0 v_1 v_3 v_4\rangle$, $\sigma_3 = \langle v_0 v_1 v_2 v_4\rangle$ and 
$K\rq =(K\setminus \{\sigma, \sigma\rq\}) \cup \{\sigma_1, \sigma_2, \sigma_3\}$.  
Suppose that $v_0 < v_1 < v_2 < v_3 < v_4$.  Then $\epsilon_{\sigma} = \epsilon_{\sigma\rq} = \epsilon_1 = -\epsilon_2 = \epsilon_3$ and 
$${\rm Col}(K) \ni \varphi \mapsto \varphi\rq \in {\rm Col}(K\rq)$$ 
is a bijection, where $\varphi\rq$ is the coloring such that
\[
 \varphi\rq (E) = 
  \begin{cases}
   \varphi (E) & E {\rm \:is \:an \:edge \:of}\:K\\
   \varphi(\langle v_0 v_1\rangle)\varphi(\langle v_1 v_2\rangle)\varphi(\langle v_2 v_3\rangle)\varphi(\langle v_3 v_4\rangle) & E = \langle v_0 v_4\rangle.
  \end{cases}
\]

\noindent It suffices to prove that for any $\varphi \in {\rm Col}(K)$, 
$$W(\sigma, \varphi) W(\sigma\rq, \varphi) = W(\sigma_1, \varphi\rq) W(\sigma_2, \varphi\rq) W(\sigma_3, \varphi\rq).$$
Set $\varphi(\langle v_0 v_1\rangle) = g$, $\varphi(\langle v_1 v_2\rangle) = h$, $\varphi(\langle v_2 v_3\rangle) = k$, $\varphi(\langle v_3 v_4\rangle) = l$ 
and then this equation follows from the cocycle condition for ($g,h,k,l$).
\end{proof}

We complete the proof of Theorem 2.2.

We present simple properties of the generalized DW invariant which are known for the original DW invariant in \cite{Wakui1}.  
The following proposition can be proved in the same way as the original DW case in \cite{Wakui1}.
 
\begin{prop}
Let $M$ be a compact or cusped oriented 3-manifold, $G$ a finite group and $\alpha\in Z^3(G,U(1))$.  Then the following holds.

(1) $Z(M)$ only depends on the cohomology class of $\alpha$.

(2) $Z(-M) = \overline{Z(M)}$, where $-M$ is the oriented 3-manifold with the opposite orientation to $M$.
\end{prop}

Although we introduce a generalized ideal triangulation in the definition of the generalized DW invariant, 
it suffices to consider ideal triangulations of $M$ by the following two theorems. 

\begin{thm}[{\cite[Theorem 1.2.27]{Matveev}}]
Any two  ideal triangulations of a 3-manifold $M$ can be transformed one to another by a finite sequence of the following two types of transformations 
shown in Figure 9.      
\end{thm}

\begin{figure}[h]
\centering
(0,2)-Pachner move \hspace{3.2cm} (2,3)-Pachner move

\includegraphics[width=2cm]{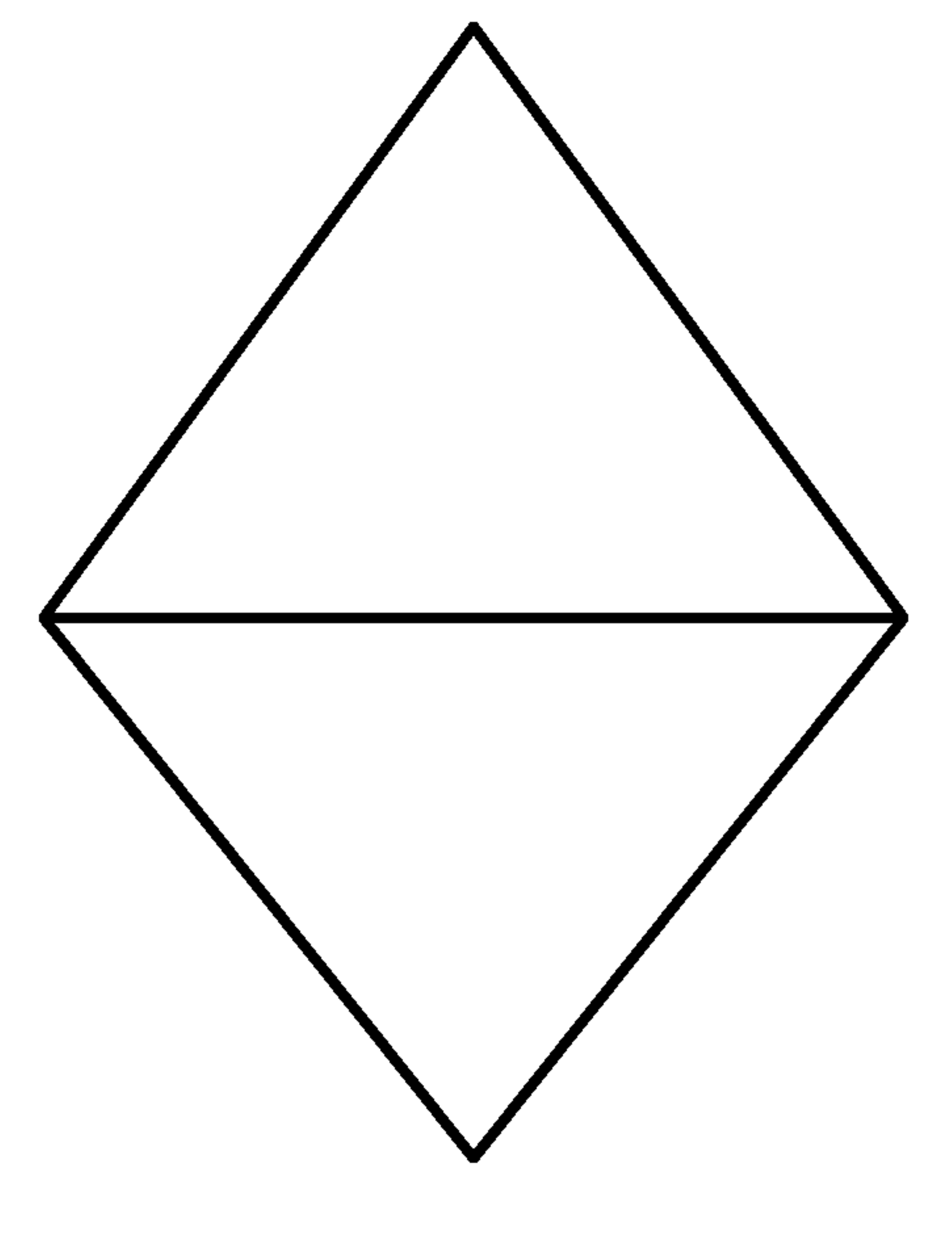}\hspace{0.1cm}
\includegraphics[width=0.8cm]{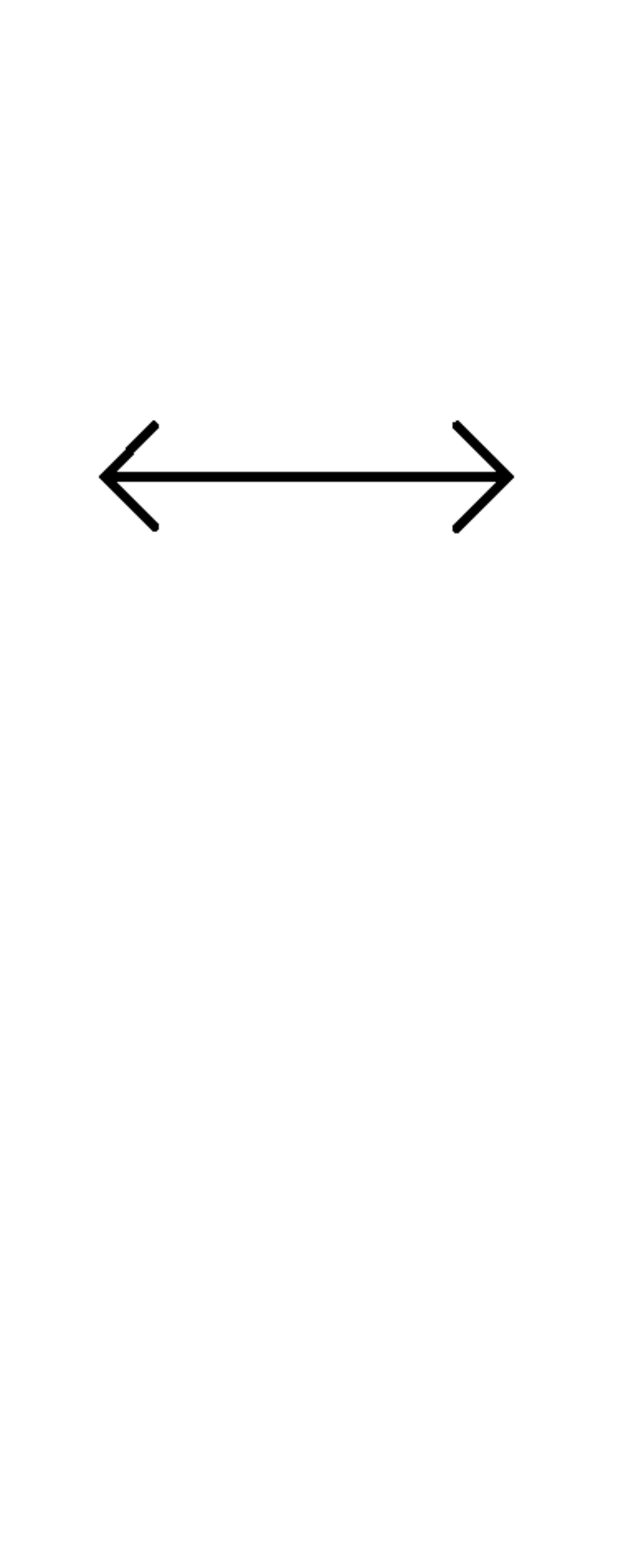}\hspace{0.1cm}
\includegraphics[width=2cm]{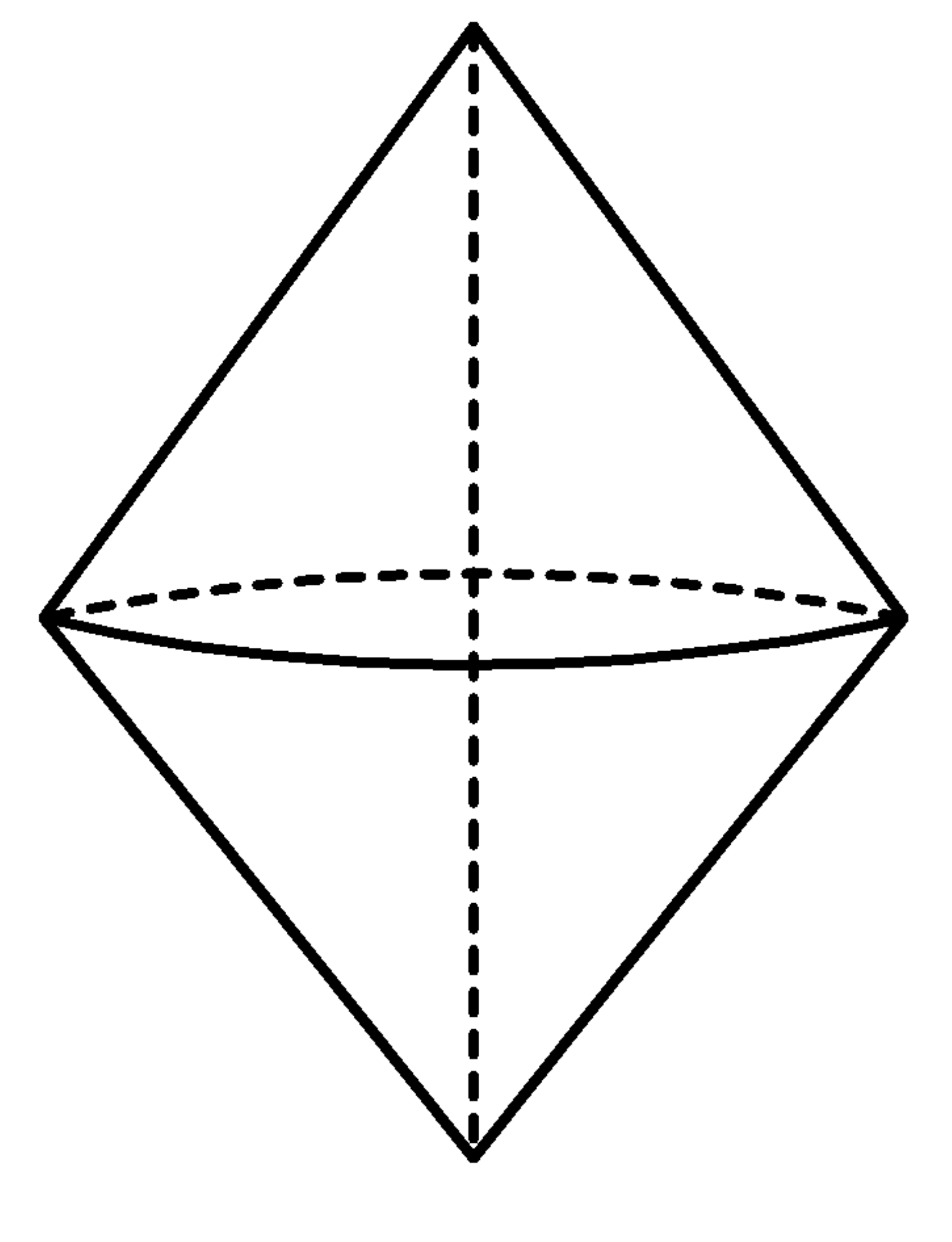}\hspace{0.9cm} 
\includegraphics[width=2.3cm]{2-3pach3.pdf}
\includegraphics[width=0.8cm]{arrow2a.pdf}\hspace{0.15cm}
\includegraphics[width=2.3cm]{2-3pach4.pdf}
\caption{The Pachner moves for ideal triangulations.}
\end{figure}

We call a (2,3)-Pachner move that increases the number of the ideal tetrahedra {\it a positive (2,3)-Pachner move} in this paper.  
In general, a given ideal triangulation of $M$ may not admit a local order.  
However Benedetti and Petronio proved the existence of an ideal triangulation with a local order \cite[Theorem 3.4.9]{Benedetti2}.

\begin{thm}[Benedetti-Petronio]
Let $M$ be a compact oriented 3-manifold with boundary and $K$ an ideal triangulation of $M$.  
Then there exists a finite sequence of ideal triangulations of $M$, 
$K = K_0 \rightarrow K_1 \rightarrow \cdots \rightarrow K_n$, 
such that $K_i$ is transformed to $K_{i+1}$ by a positive (2,3)-Pachner move and $K_n$ admits a local order.
\end{thm}

\begin{cor}
For any cusped or compact 3-manifold $M$ with boundary, there exists an ideal triangulation $K$ of $M$ with a local order.  
Since $K$ does not have interior vertices, the generalized Dijkgraaf-Witten invariant $Z(M)$ is described by the following form: 
$$
Z(M) = \sum_{\varphi \in {\rm Col}(K)} \prod_{i=1}^{n} W(\sigma_i, \varphi).
$$
\end{cor}
 

\section{Examples of cusped hyperbolic 3-manifolds}

In this section, we calculate the generalized DW invariants of some cusped orientable hyperbolic 3-manifolds by using Theorem 3.7 and Corollary 3.8.  
We show that the generalized DW invariants distinguish some pairs of cusped hyperbolic 3-manifolds 
with the same hyperbolic volumes and with the same Turaev-Viro invariants.
We also present an example of a pair of cusped hyperbolic 3-manifolds with the same hyperbolic volumes and with the same homology groups,
whereas with distinct generalized DW invariants.
 
For a positive integer $m$, it is known that $H^3(\mathbb{Z}_m, U(1))$ is isomorphic to $\mathbb{Z}_m$ 
and a generator $\alpha$ of $H^3(\mathbb{Z}_m, U(1)) \cong  \mathbb{Z}_m$ is described by the following formula \cite{Altschuler}:
$$\alpha(g_1,g_2,g_3) = \exp{(\frac{2\pi i}{m^2}\overline{g_1}(\overline{g_2}+\overline{g_3}-\overline{g_2+g_3}))},$$

\noindent where $\overline{g_i} \in \{0, \cdots , m-1\}$ is a representative of $g_i \in \mathbb{Z}_m$.  

(1) $m003$ and $m004$

According to Regina \cite{Burton} and SnapPy \cite{Culler}, 
$m003$ and $m004$ are cusped orientable 3-manifolds with the minimal ideal triangulations shown in 

\begin{figure}
\centering
$m003$ \hspace{3.6cm} $m004$ ( = $S^3\setminus4_1$)

\includegraphics[width=2.8cm]{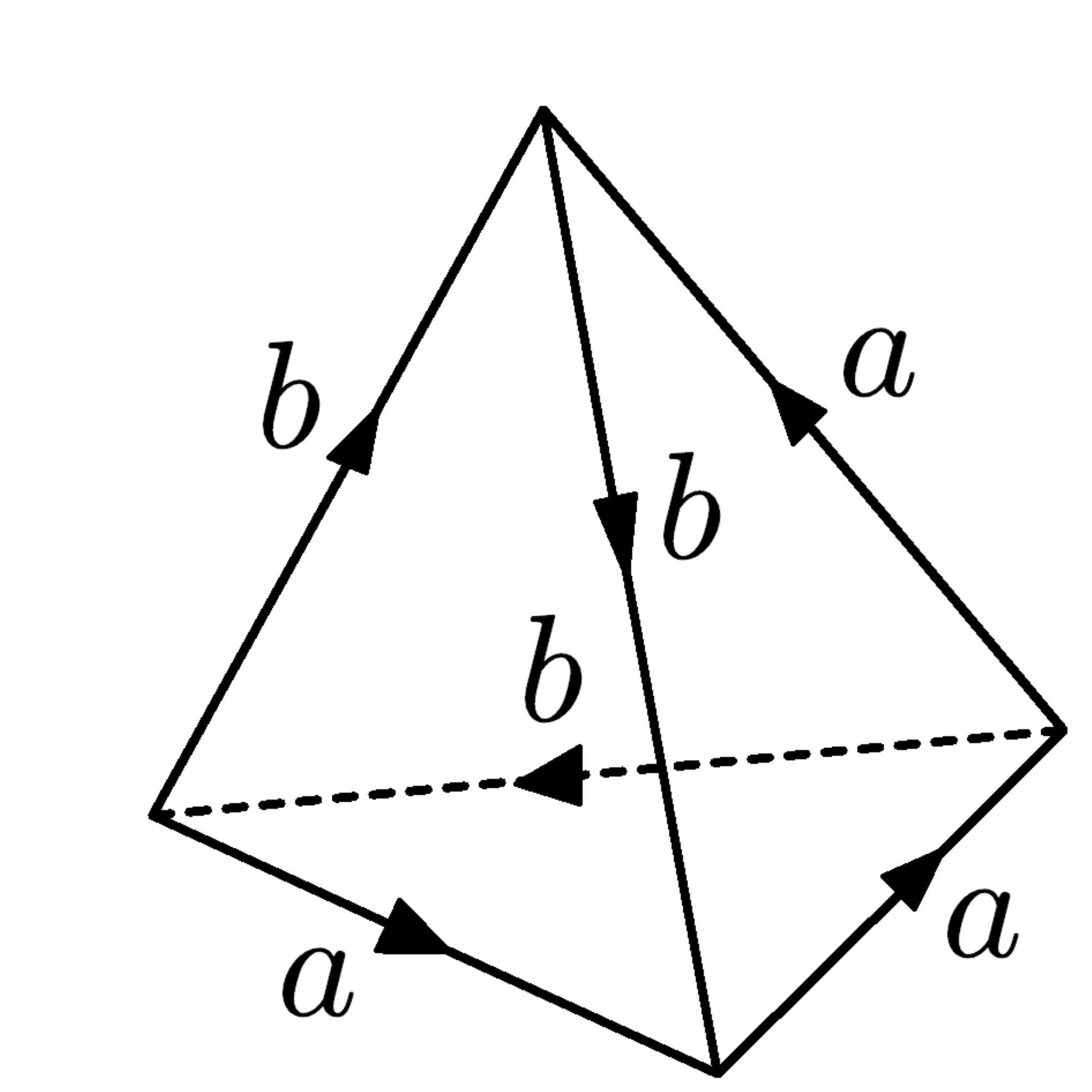}
\includegraphics[width=2.8cm]{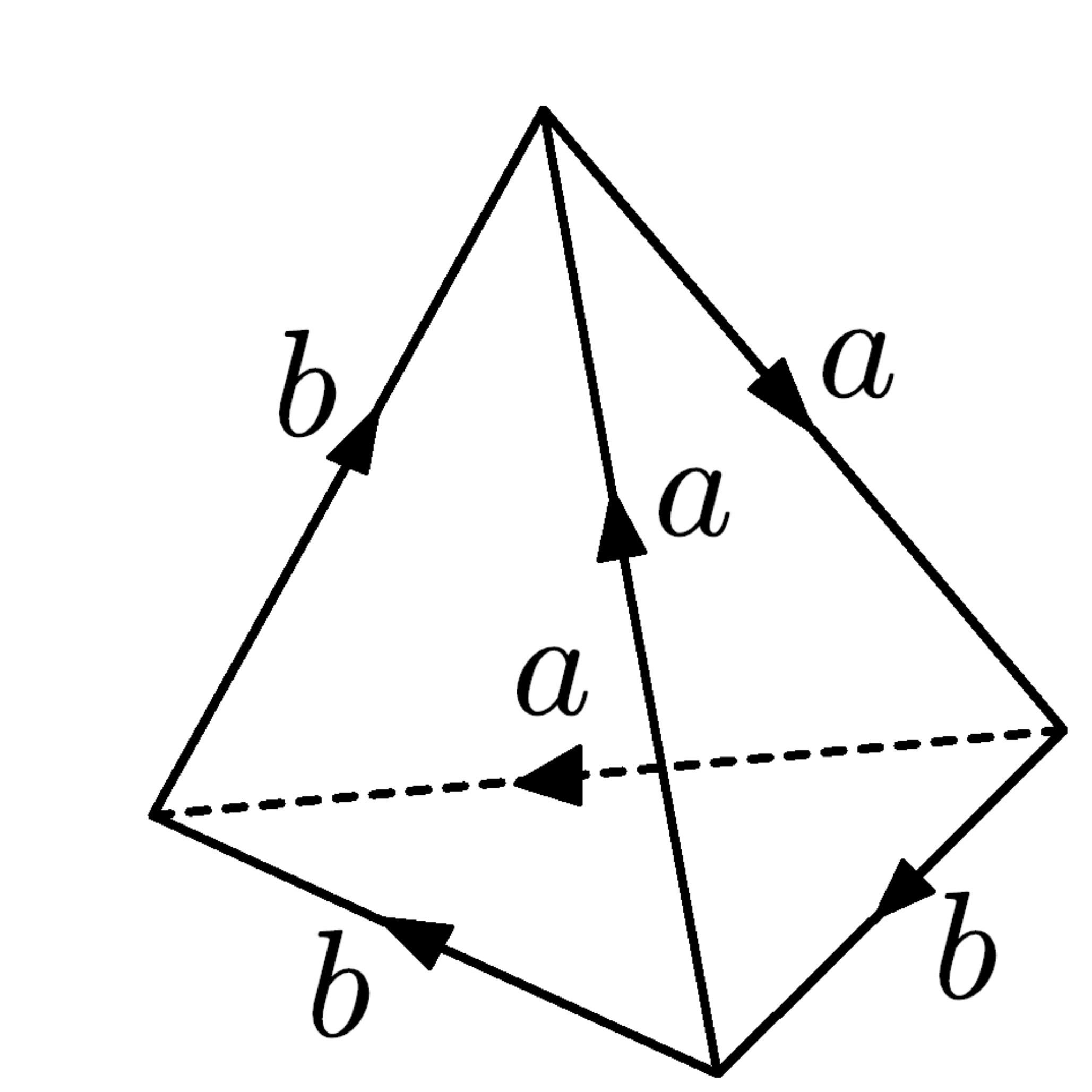} \hspace{0.57cm}
\includegraphics[width=2.8cm]{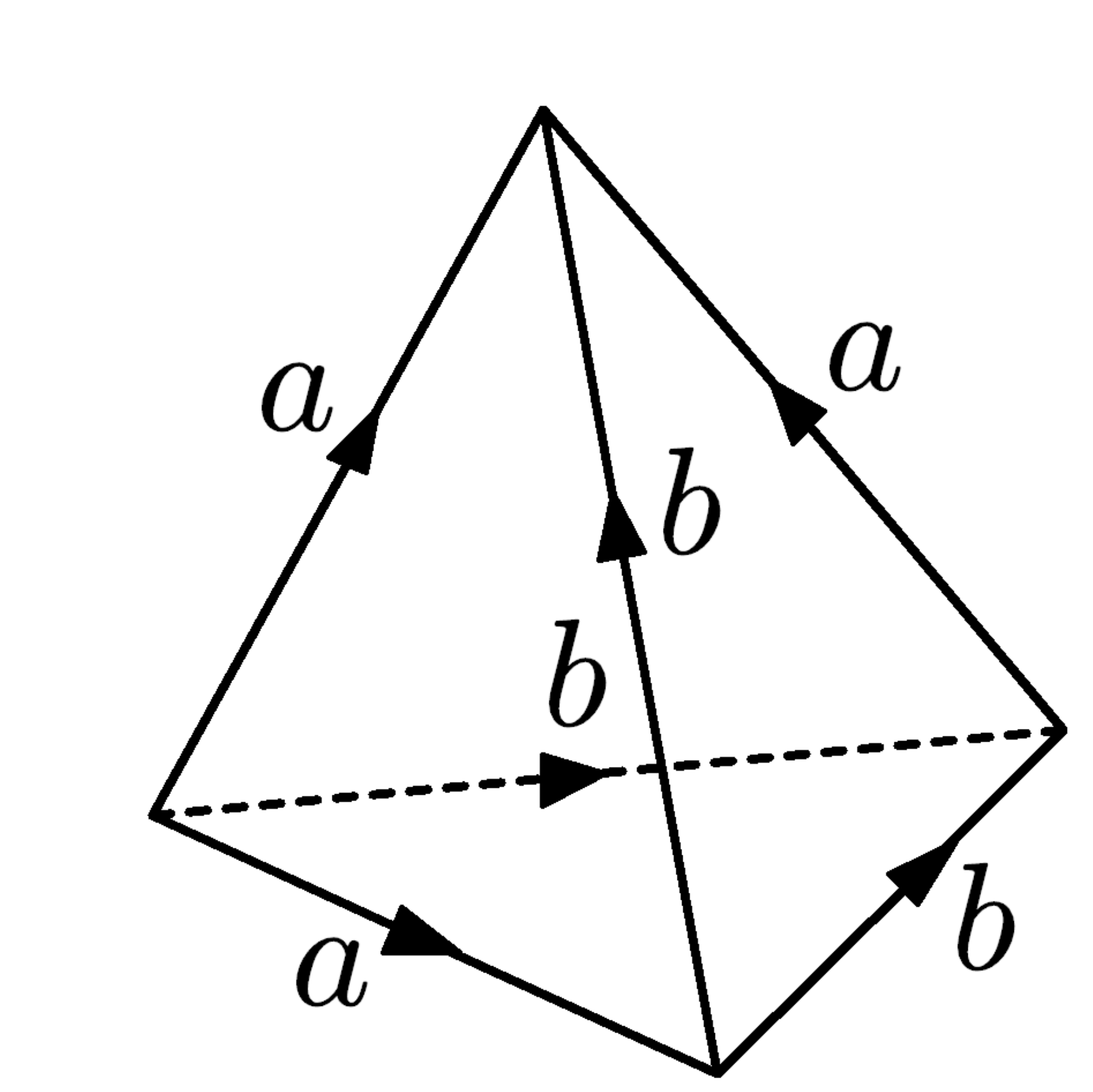}
\includegraphics[width=2.8cm]{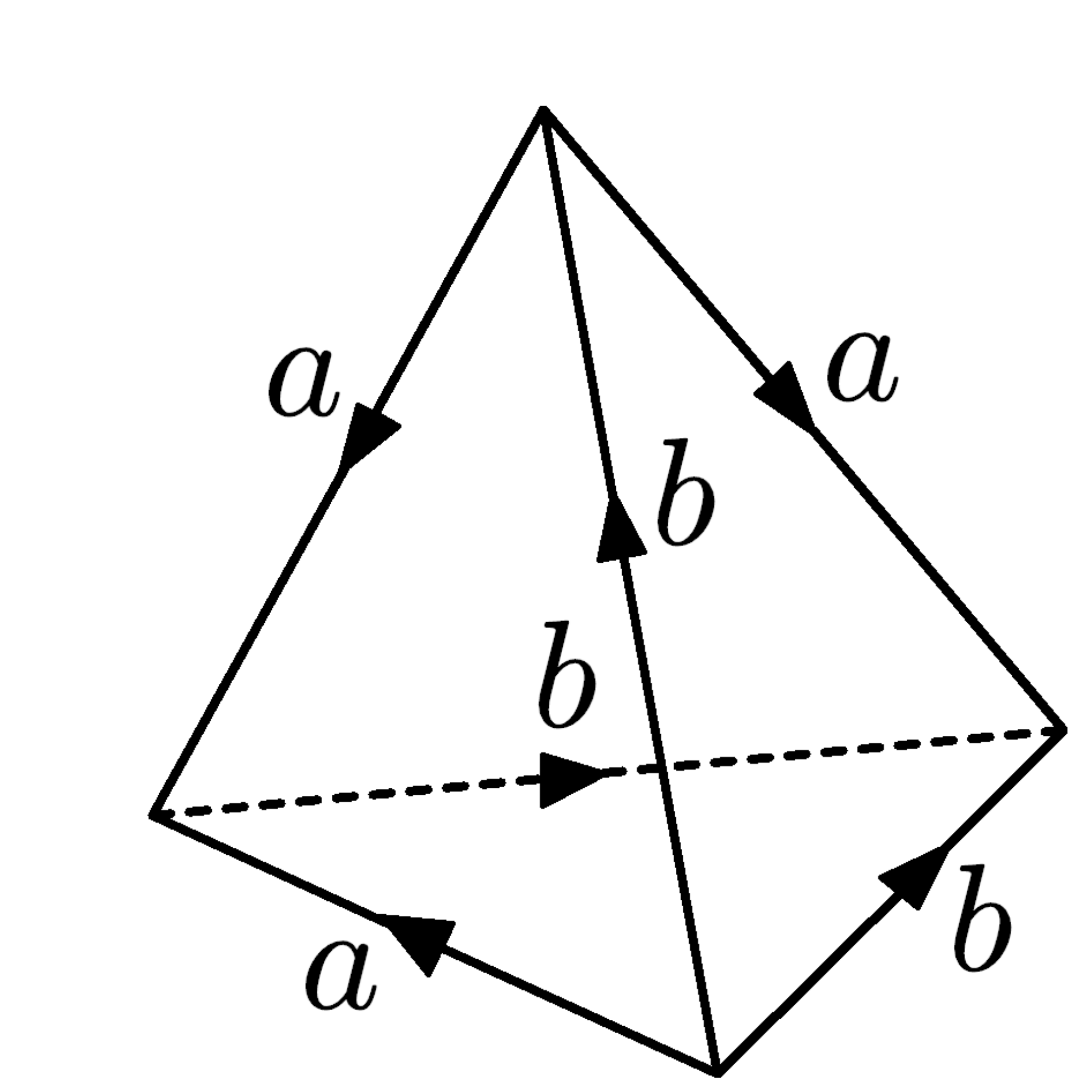} 
\caption{Minimal ideal triangulations of $m003$ and $m004$.}
\end{figure}

\noindent Figure 10.  
The 3-manifold $m004$ is the figure eight knot complement.  
Their hyperbolic volumes, Turaev-Viro invariants and homology groups are as follows:
$${\rm Vol}(m003) = {\rm Vol}(m004) \approx 2.02988,$$
$$TV(m003) = \sum_{(a,a,b),(a,b,b) \in adm} w_a w_b \begin{vmatrix} a & a & b \\ a & b & b \end{vmatrix} \begin{vmatrix} a & a & b \\ a & b & b \end{vmatrix} = TV(m004),$$
$$H_1(m003 ;\mathbb{Z}) = \mathbb{Z} \oplus \mathbb{Z}_5,\quad H_1(m004 ; \mathbb{Z}) = \mathbb{Z}.$$

We show that $m003$ and $m004$ have distinct generalized DW invariants.  

First we calculate the generalized DW invariant of $m004$.
The minimal ideal triangulation of $m004$ admits the local order shown in Figure 10.  
Identify the labels of edges with the colors of edges.  
By the left front face of the left ideal tetrahedron of $m004$ shown in Figure 10, $a=ba$.  
By the right front face of the left ideal tetrahedron of $m004$, $b=ab$.  
Hence $a=b=1 \in G$, which implies $m004$ has only a trivial coloring.
Therefore, for any finite group $G$ and its any normalized 3-cocycle $\alpha$, 
$$Z(m004) = 1.$$

On the other hand, the minimal ideal triangulation of $m003$ shown in Figure 10 does not admit a local order.  
Then we apply Theorem 3.8 to the ideal triangulation of $m003$.  
In order to assign a local order, transform the ideal triangulation of $m003$ by positive (2,3)-Pachner moves.

\hspace{0.8cm}
\includegraphics[width=2.2cm]{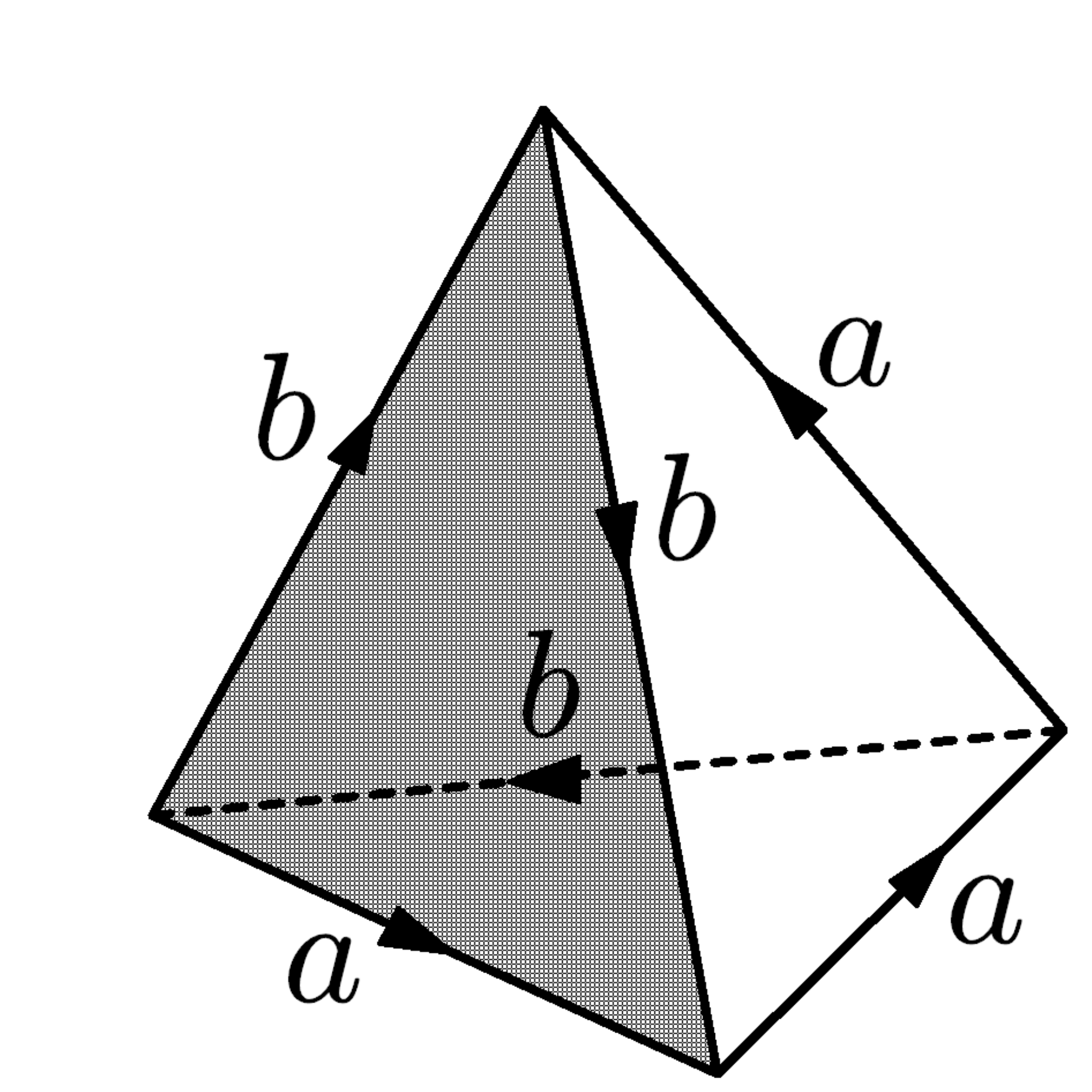}
\includegraphics[width=2.2cm]{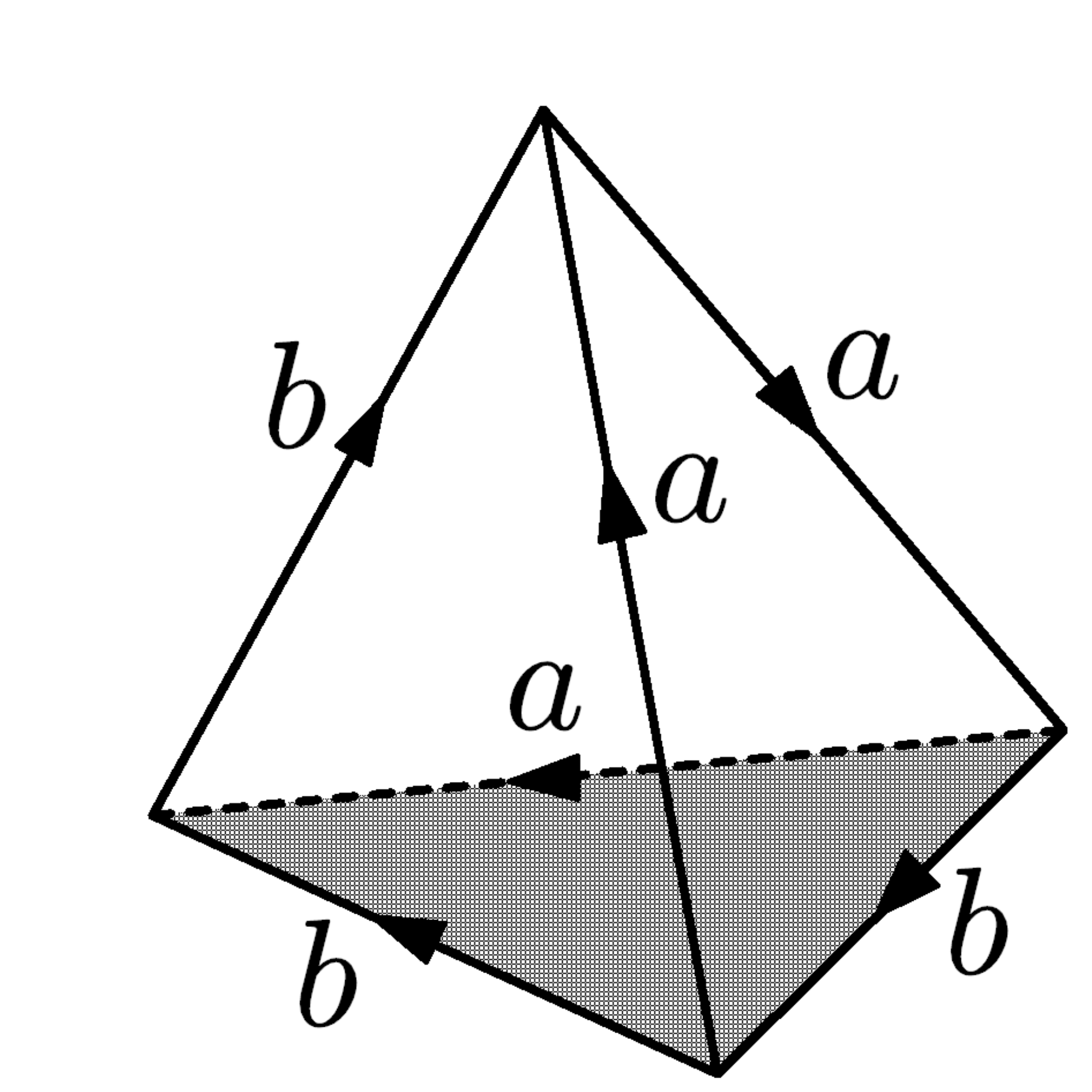}

\vspace{0.2cm}

\includegraphics[width=0.8cm]{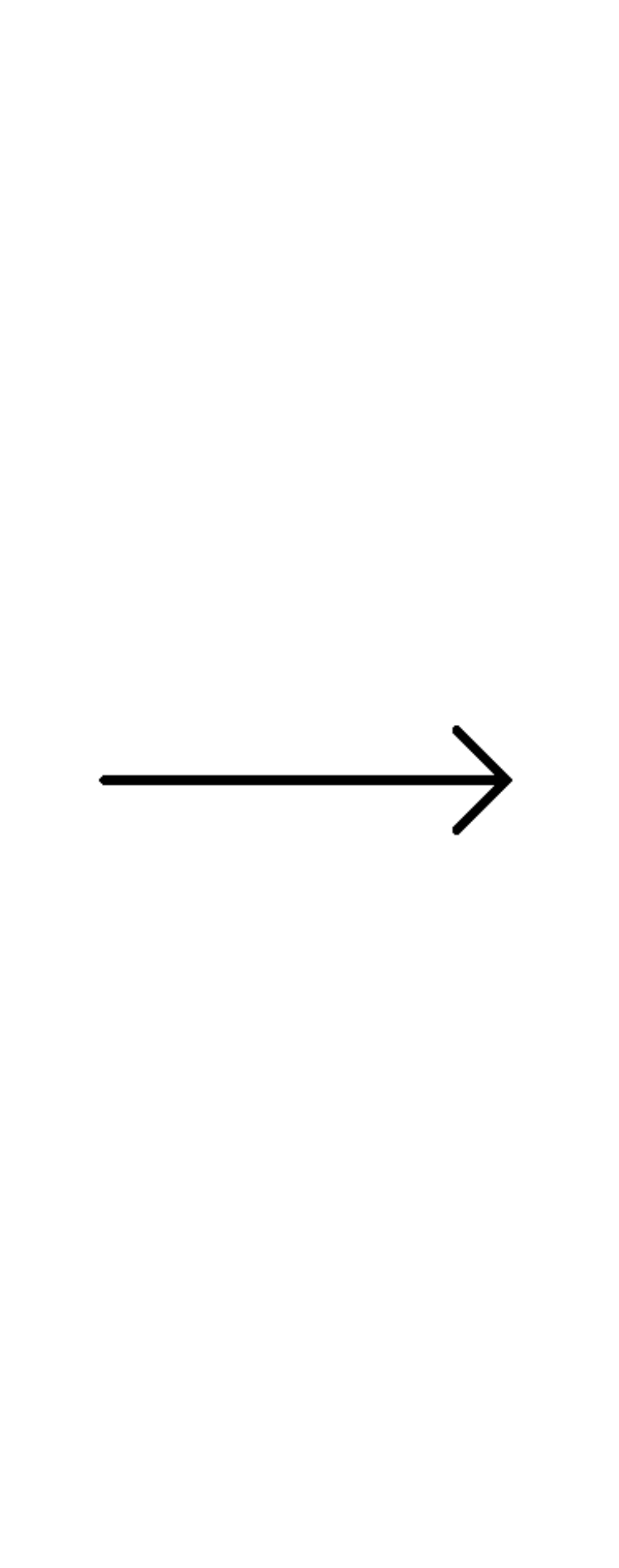}
\includegraphics[width=2.2cm]{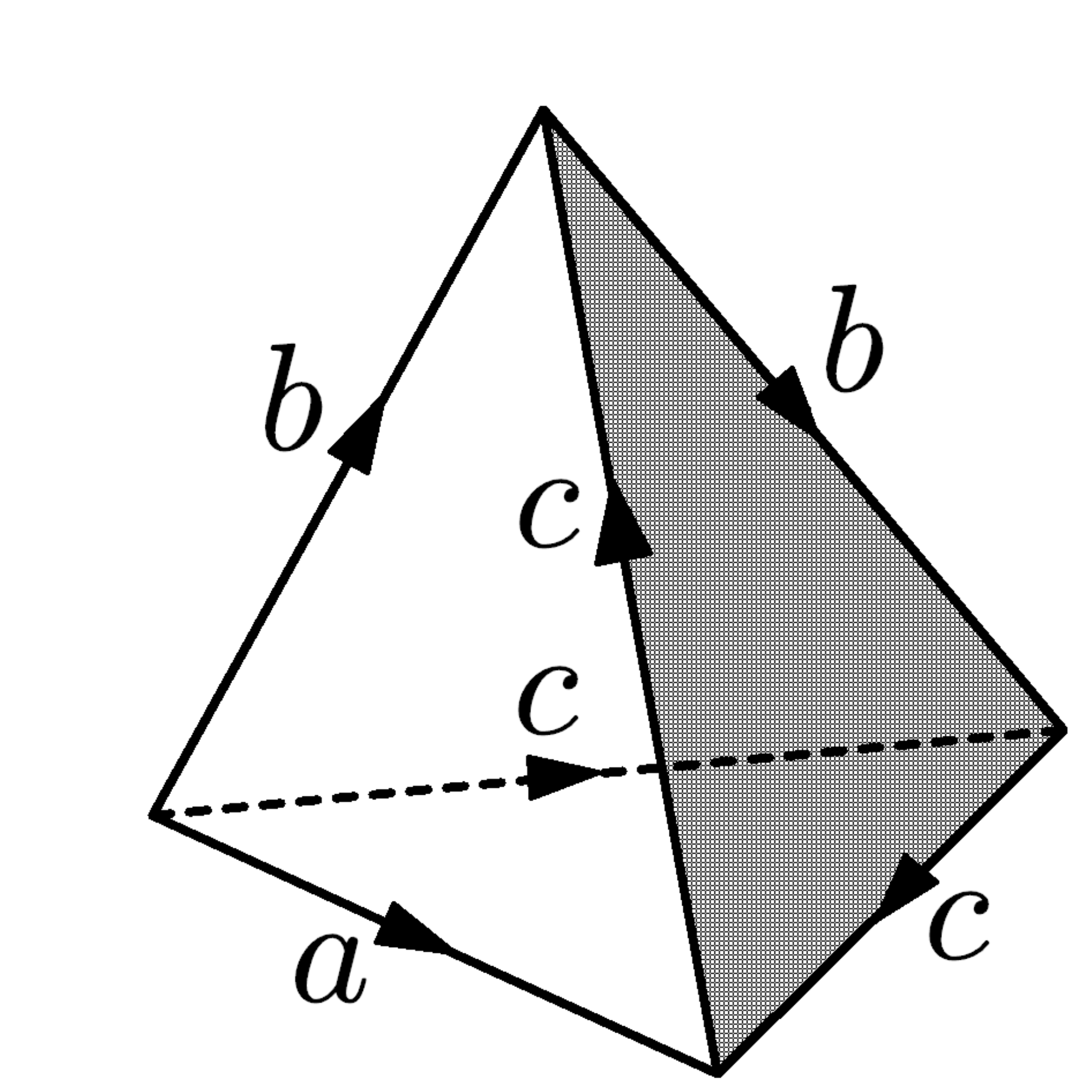}
\includegraphics[width=2.2cm]{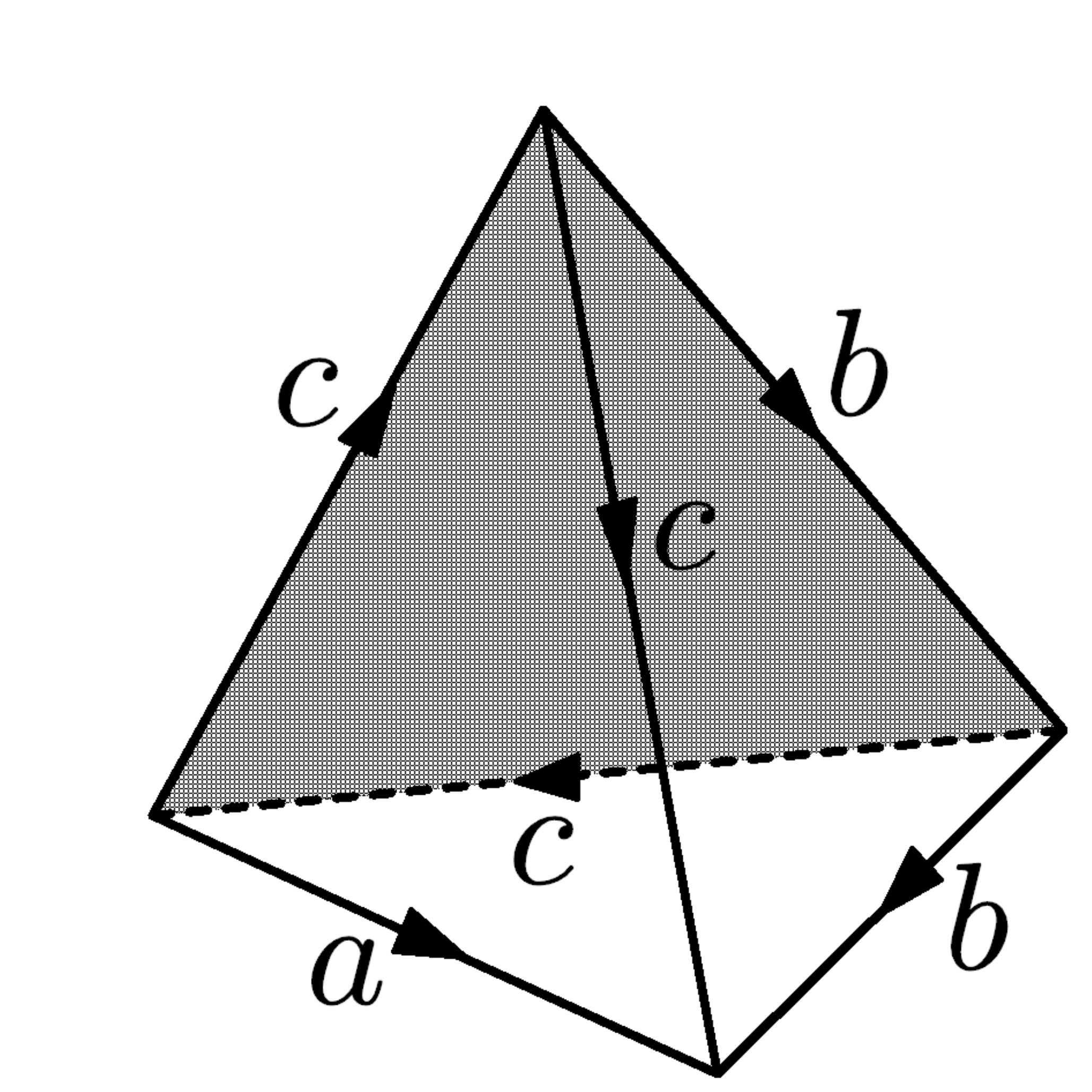}
\includegraphics[width=2.2cm]{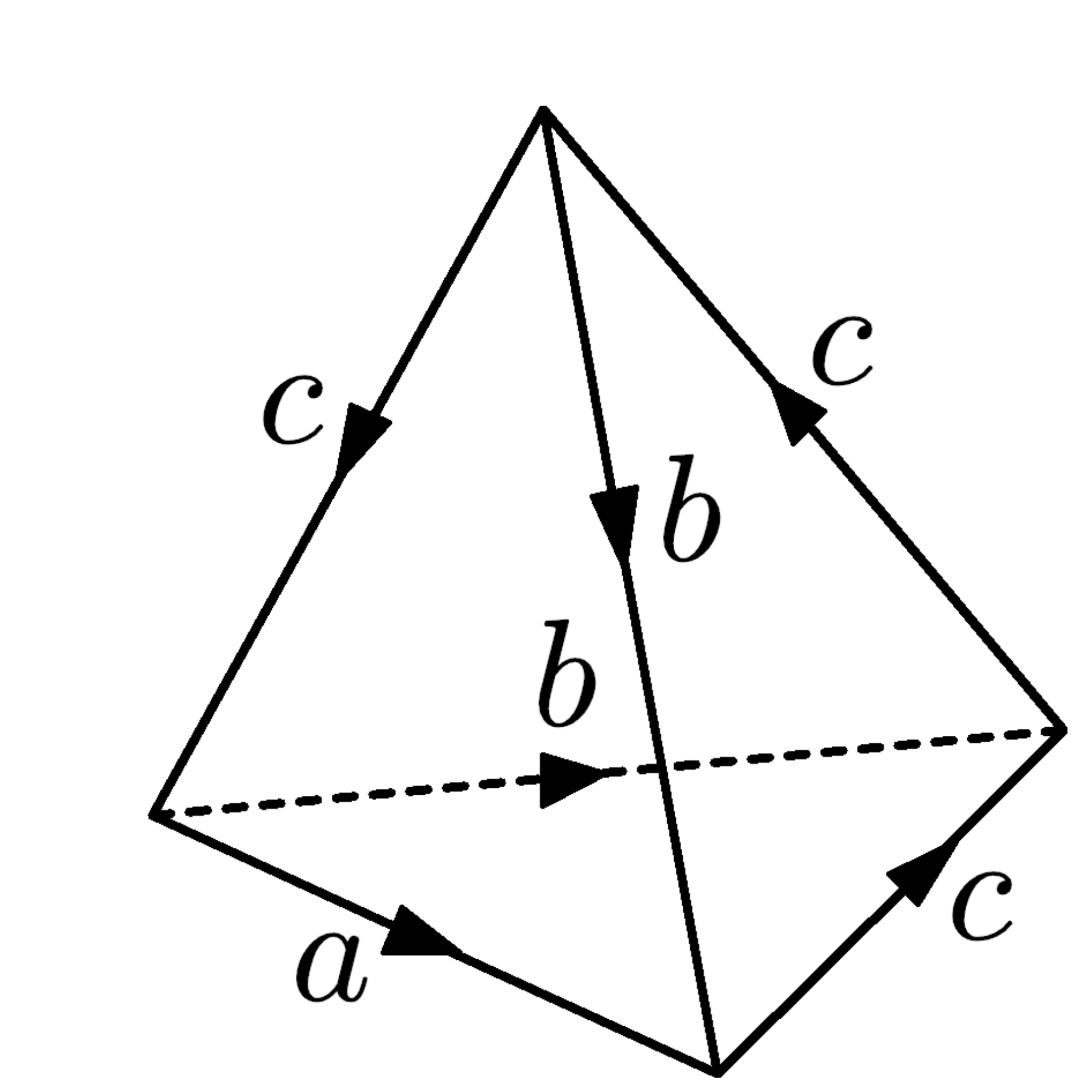}

\vspace{0.2cm}

\includegraphics[width=0.8cm]{arrow1a2.pdf}
\includegraphics[width=2.2cm]{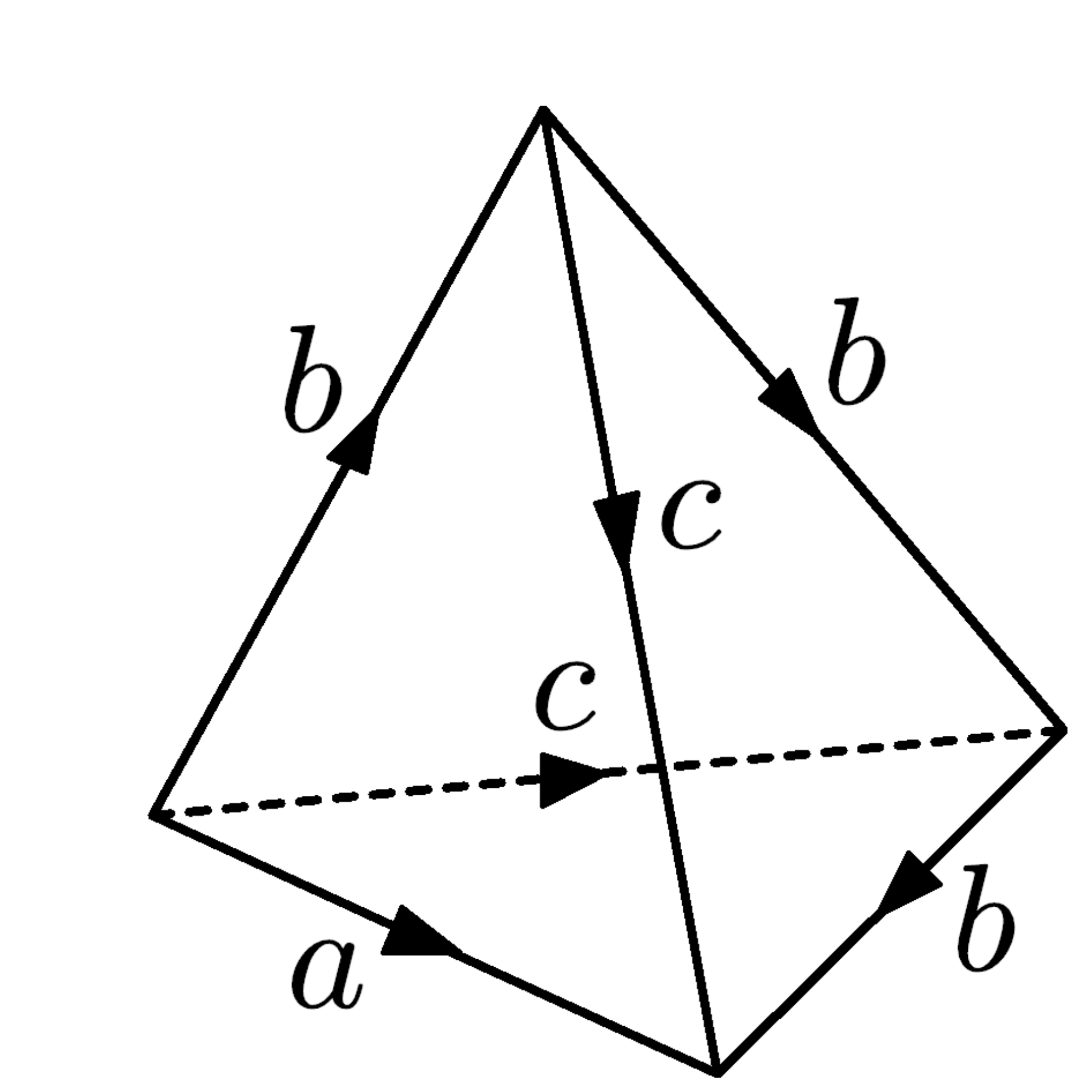}
\includegraphics[width=2.2cm]{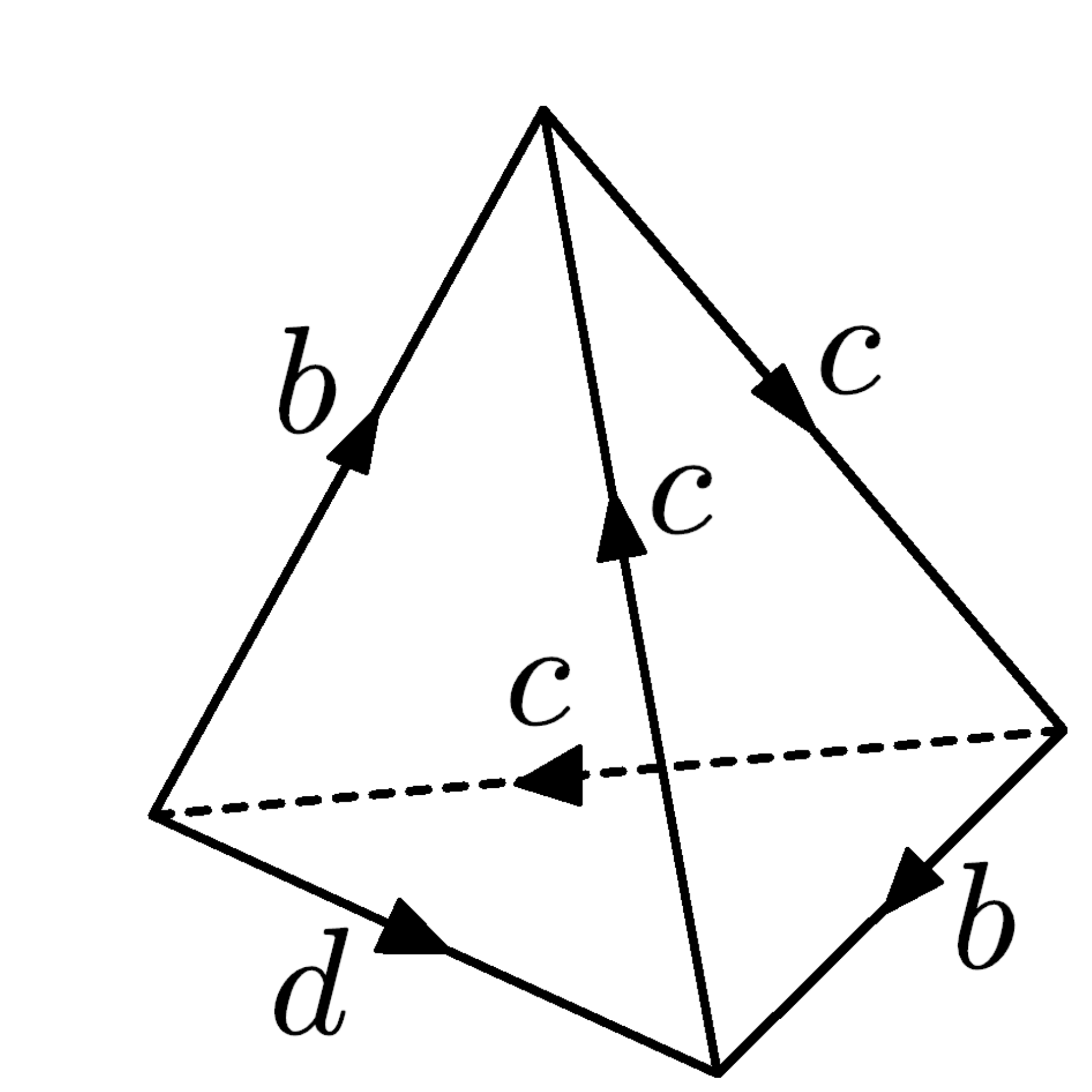}
\includegraphics[width=2.2cm]{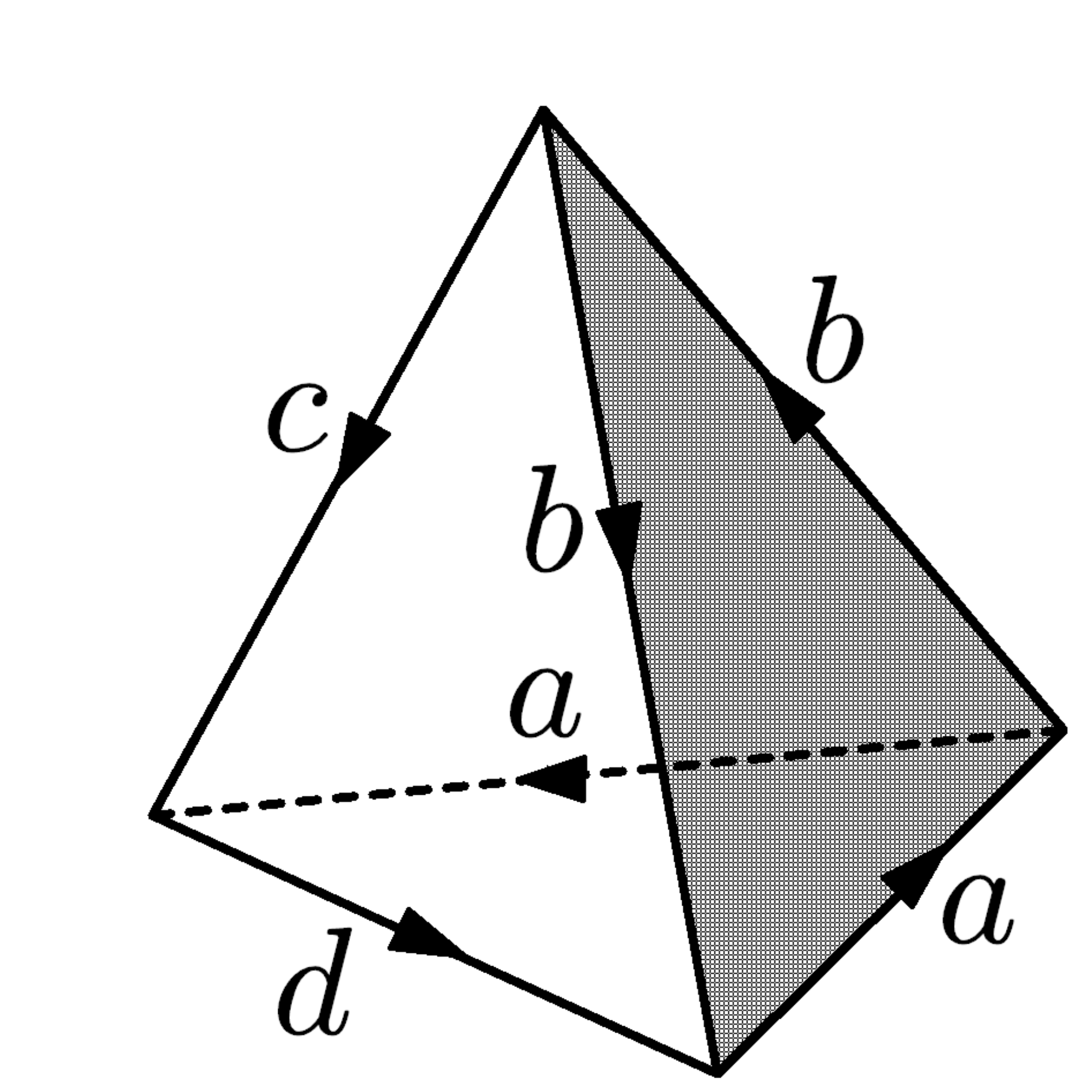}
\includegraphics[width=2.2cm]{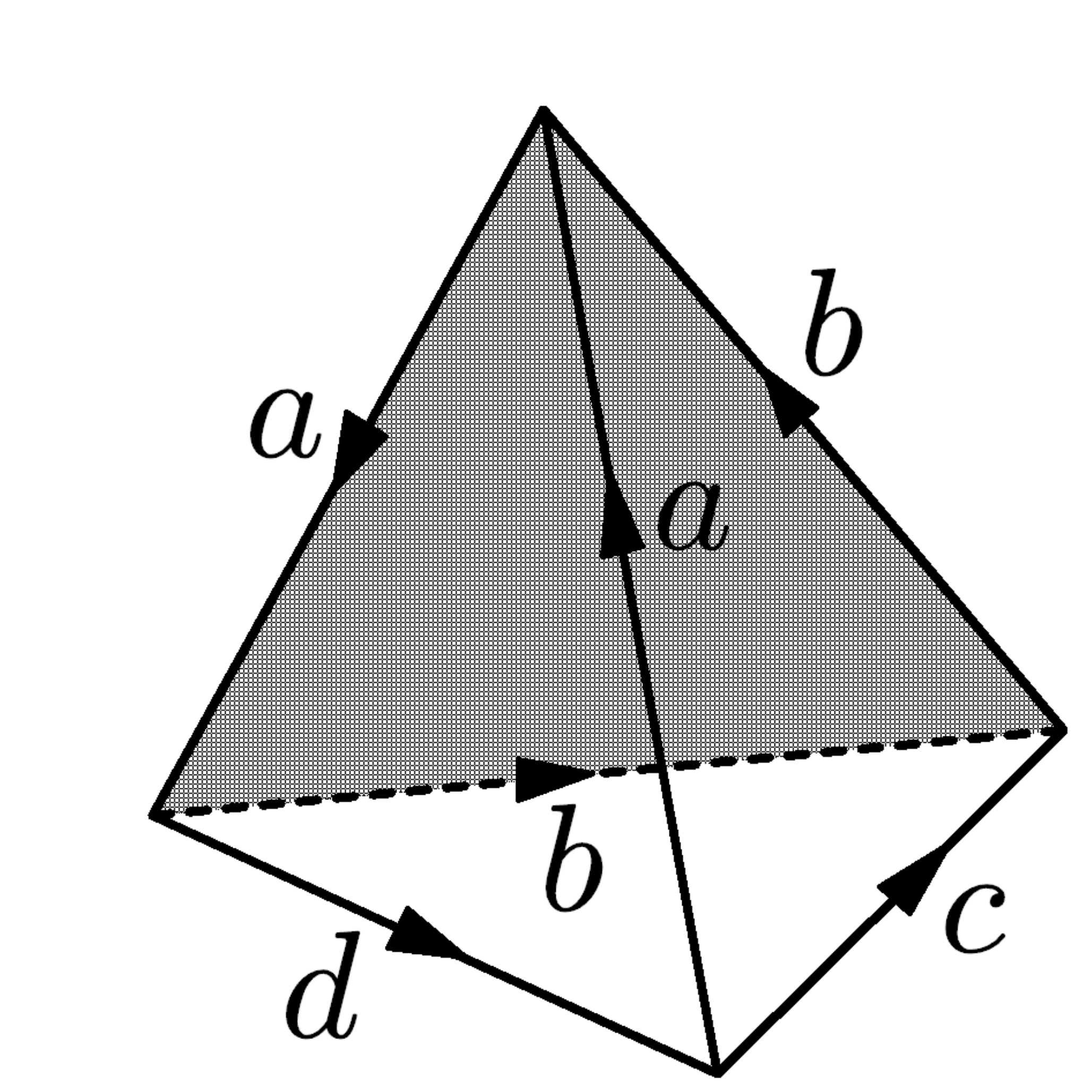}

\vspace{0.2cm}

\includegraphics[width=0.8cm]{arrow1a2.pdf}
\includegraphics[width=2.2cm]{tetra010.pdf}
\includegraphics[width=2.2cm]{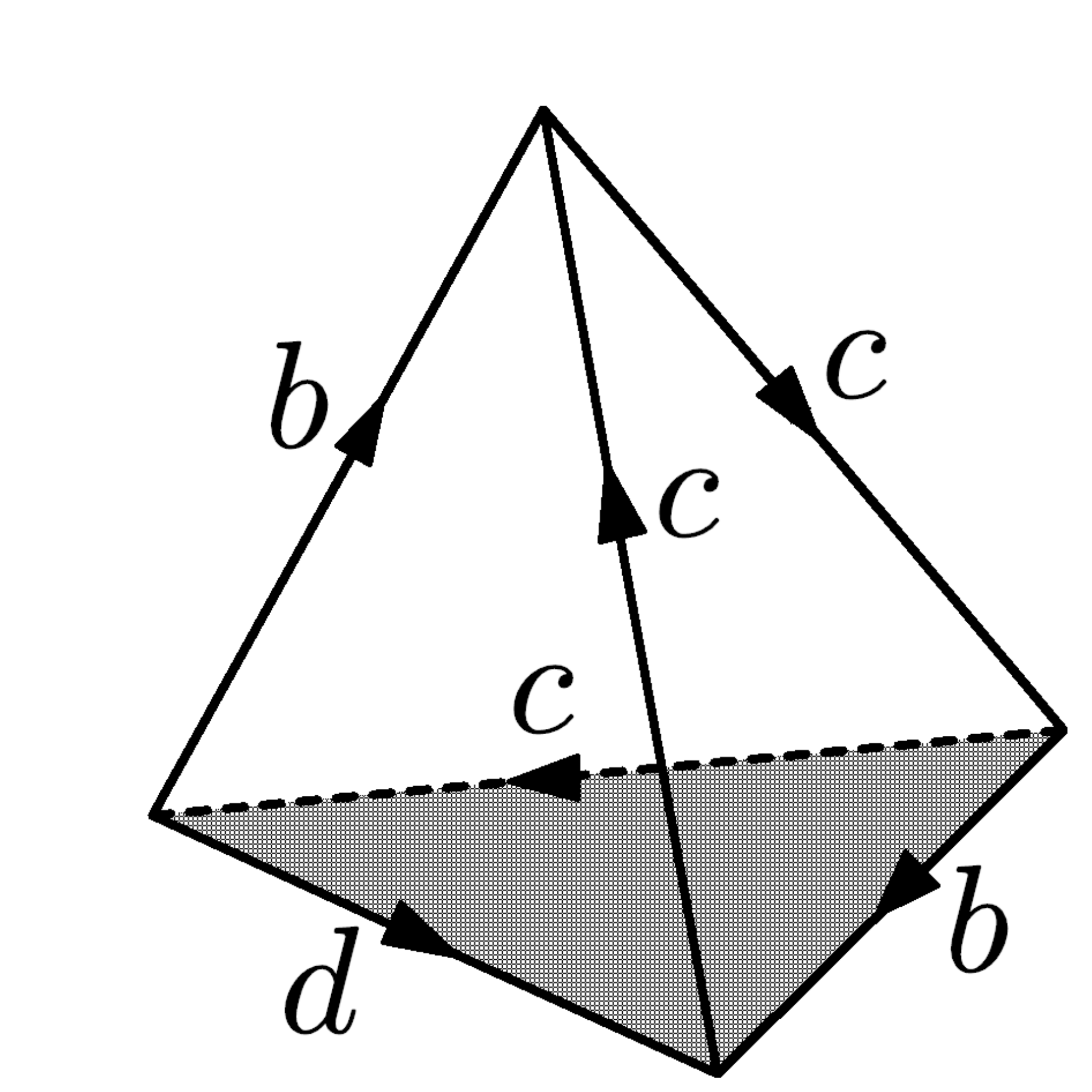}
\includegraphics[width=2.2cm]{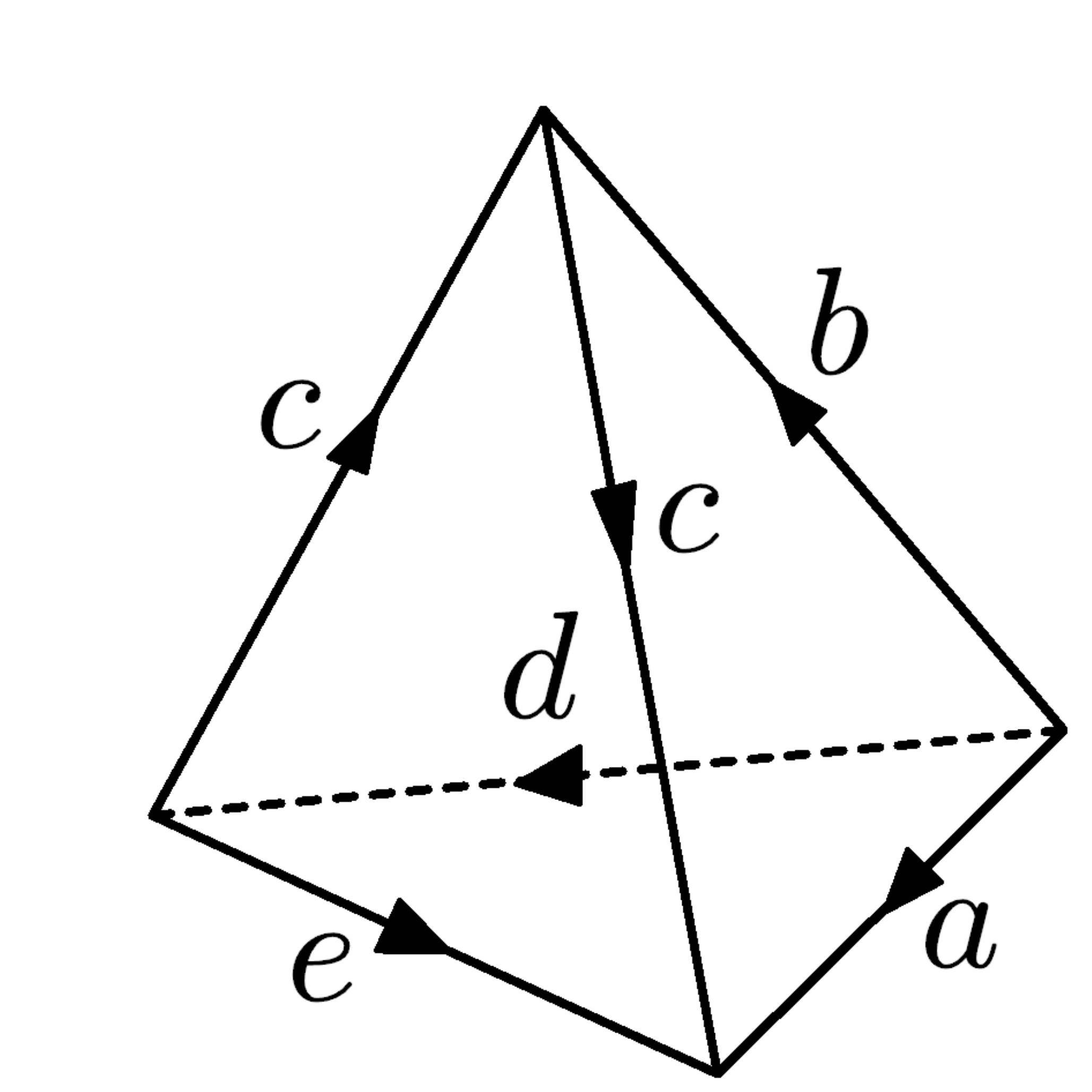}
\includegraphics[width=2.2cm]{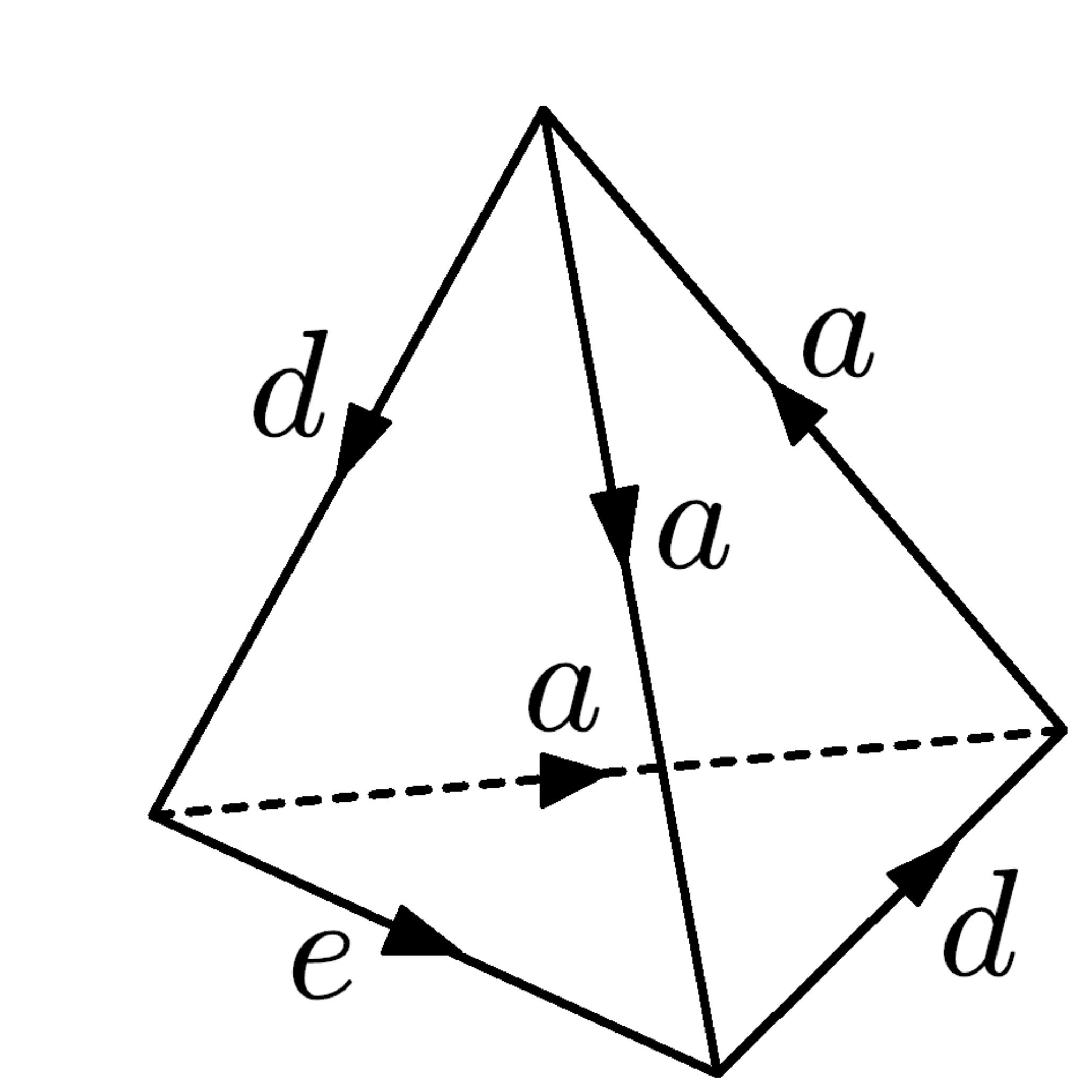}
\includegraphics[width=2.2cm]{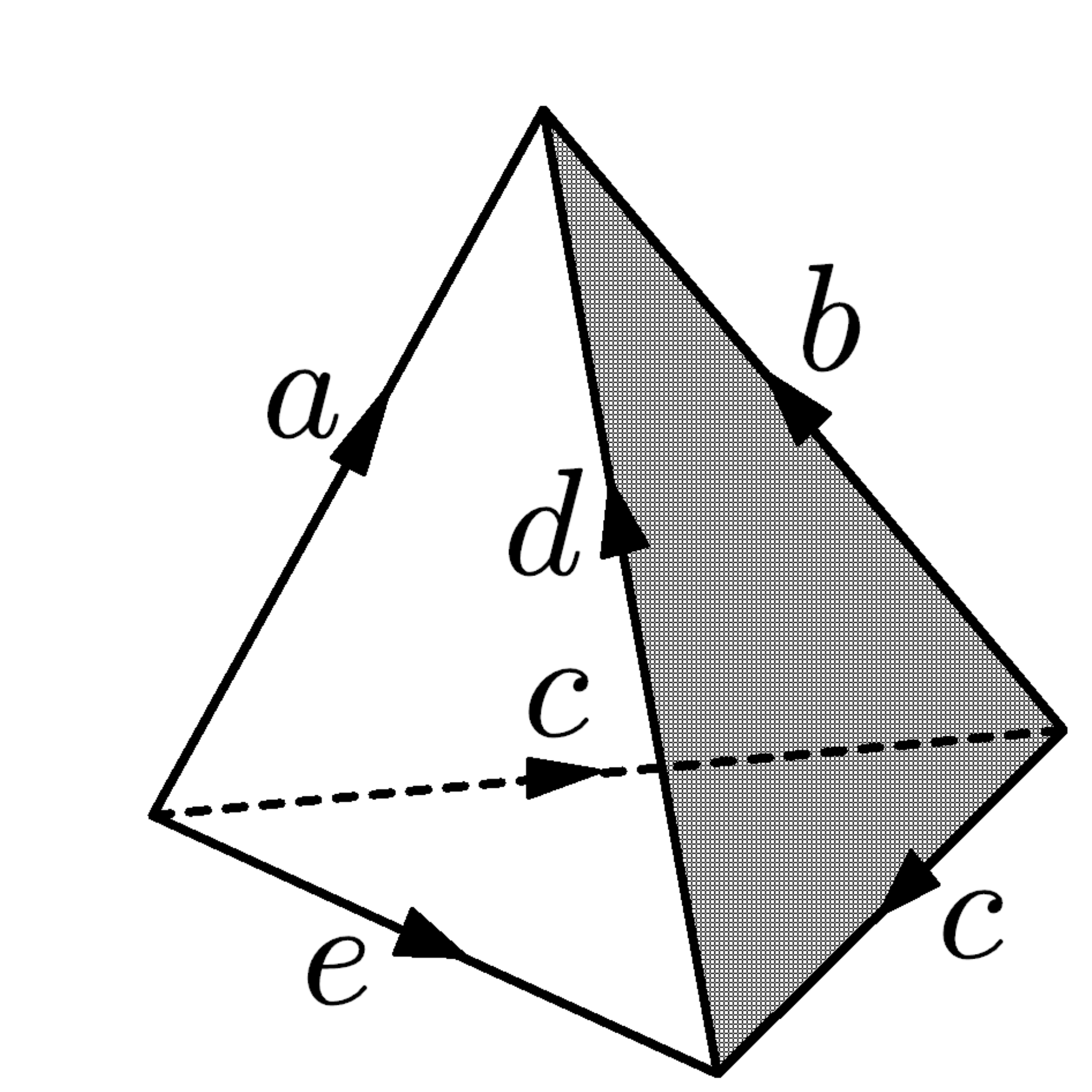}

\vspace{0.2cm}

\includegraphics[width=0.8cm]{arrow1a2.pdf}
\includegraphics[width=2.2cm]{tetra010.pdf}
\includegraphics[width=2.2cm]{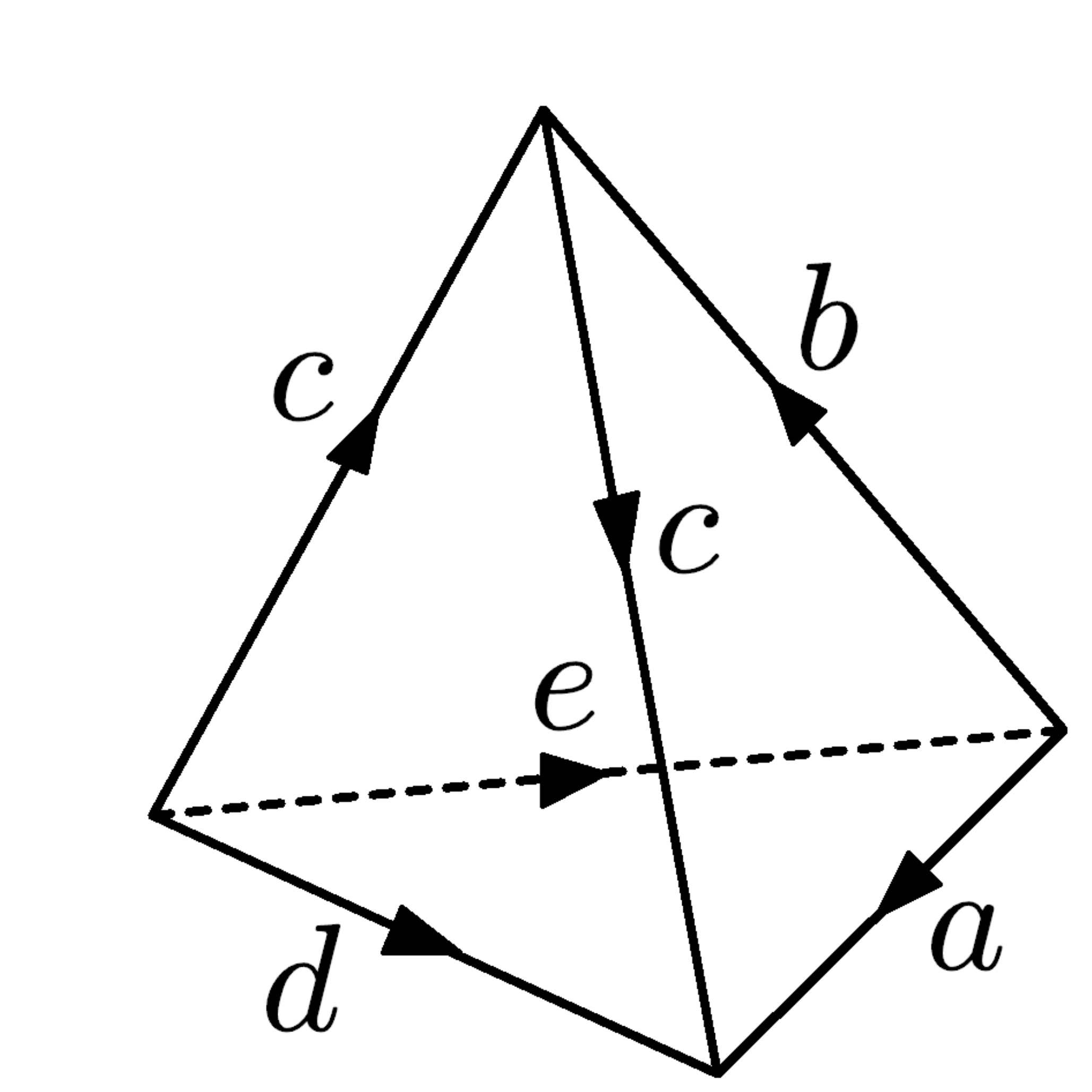}
\includegraphics[width=2.2cm]{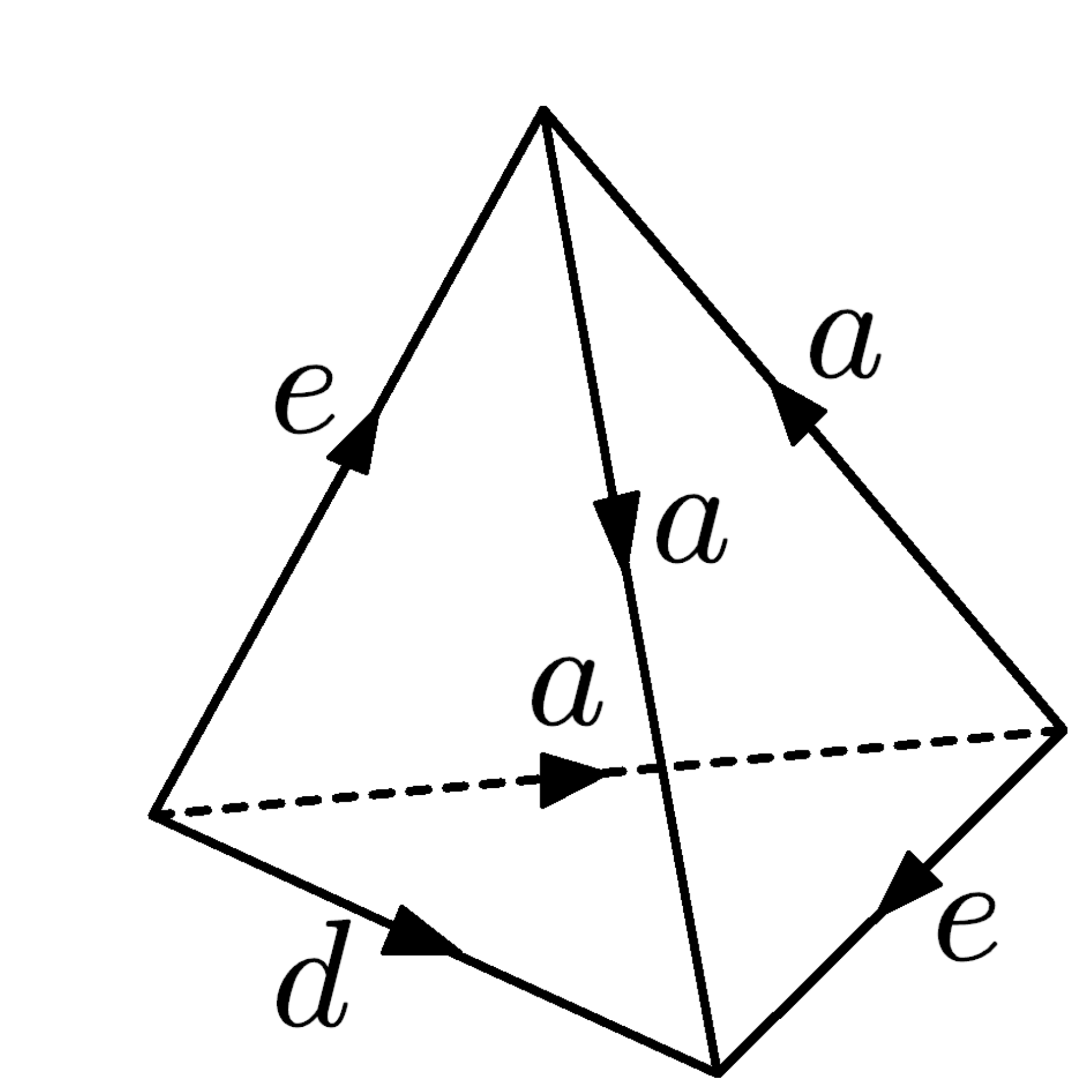}

\vspace{0.2cm}

\hspace{0.8cm}
\includegraphics[width=2.2cm]{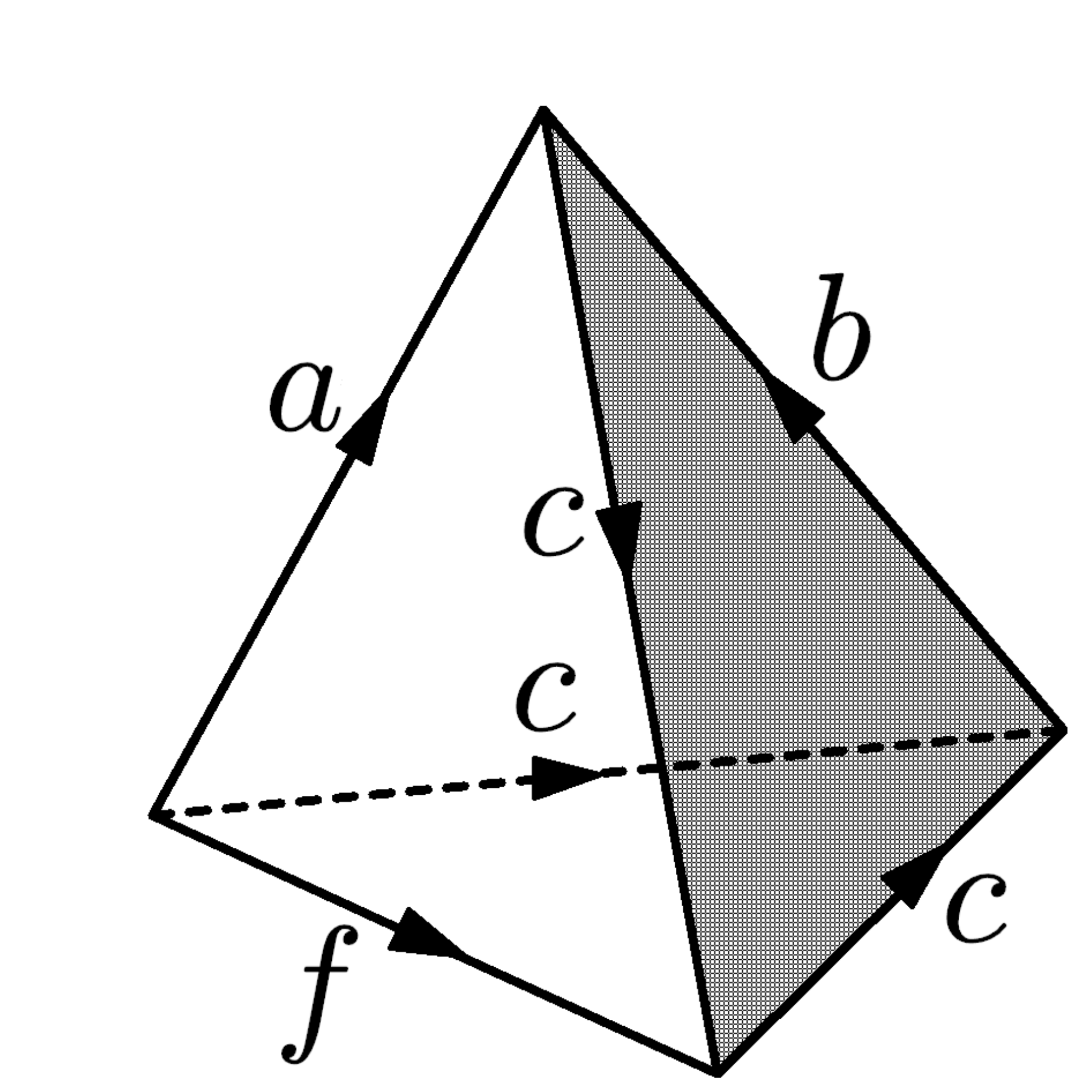}
\includegraphics[width=2.2cm]{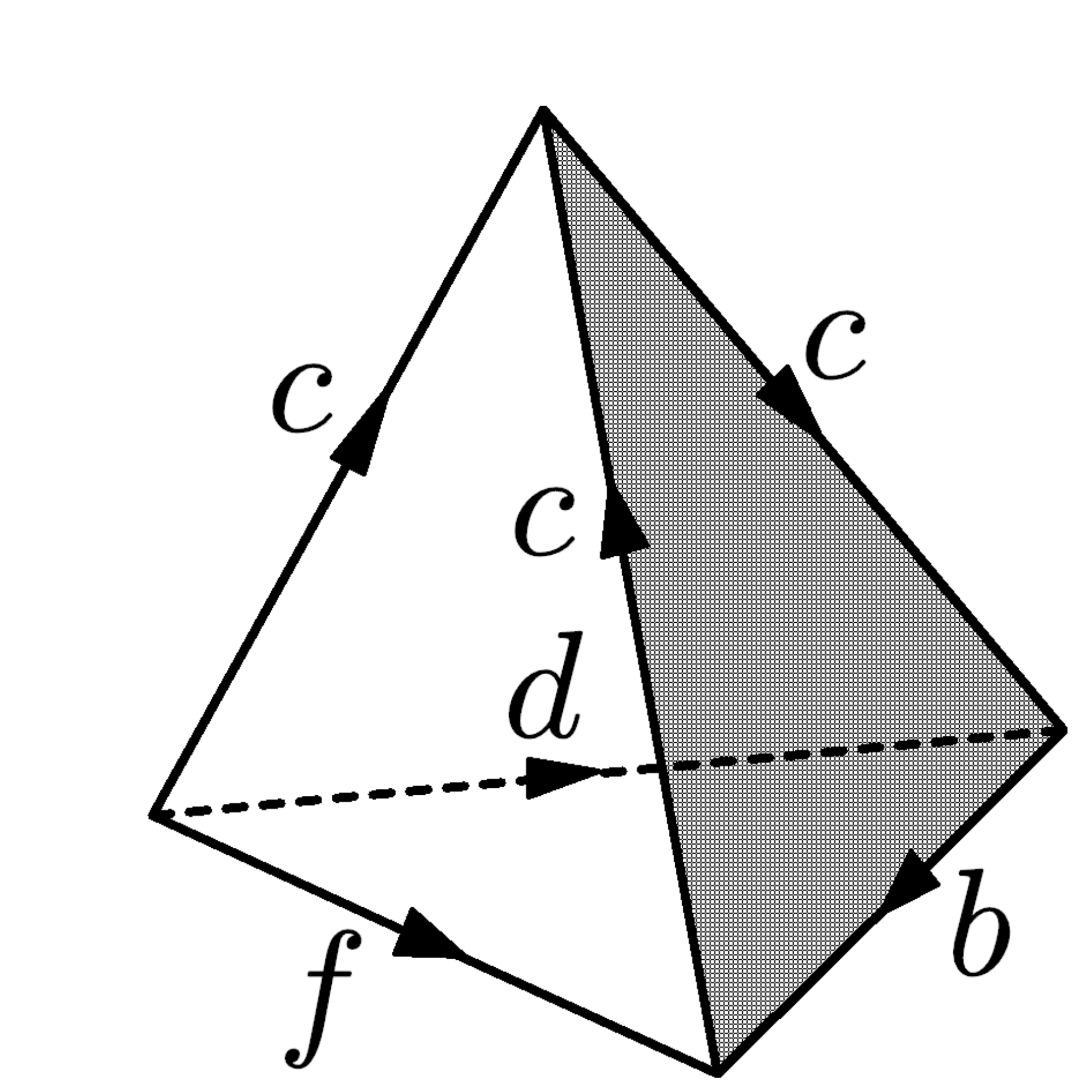}
\includegraphics[width=2.2cm]{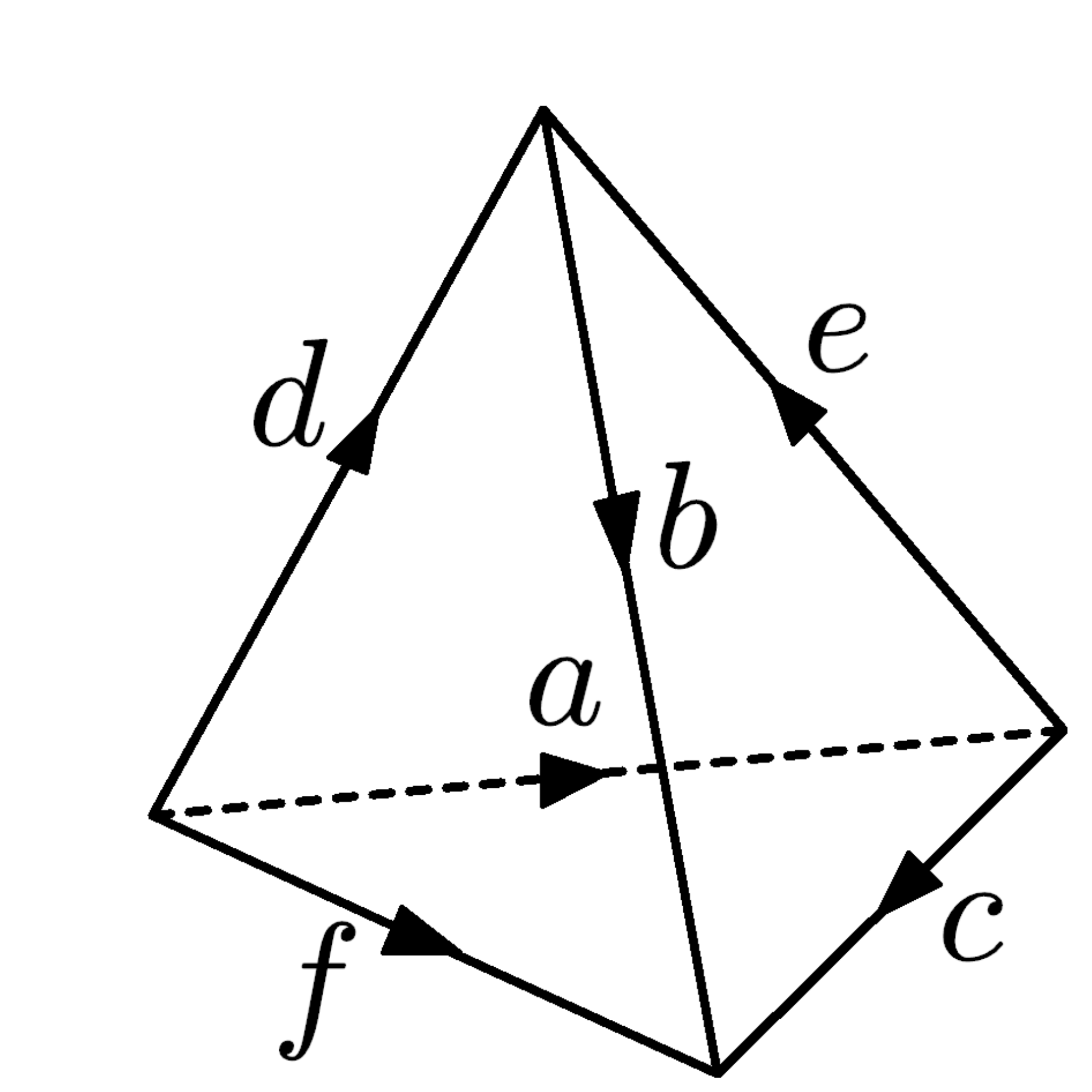}

\vspace{0.2cm}

\includegraphics[width=1cm]{arrow1a2.pdf}
\includegraphics[width=2.2cm]{tetra010.pdf}
\includegraphics[width=2.2cm]{tetra018.pdf}
\includegraphics[width=2.2cm]{tetra019.pdf}
\includegraphics[width=2.2cm]{tetra022.pdf} 

\vspace{0.2cm}

\hspace{2.3cm}
\includegraphics[width=2.2cm]{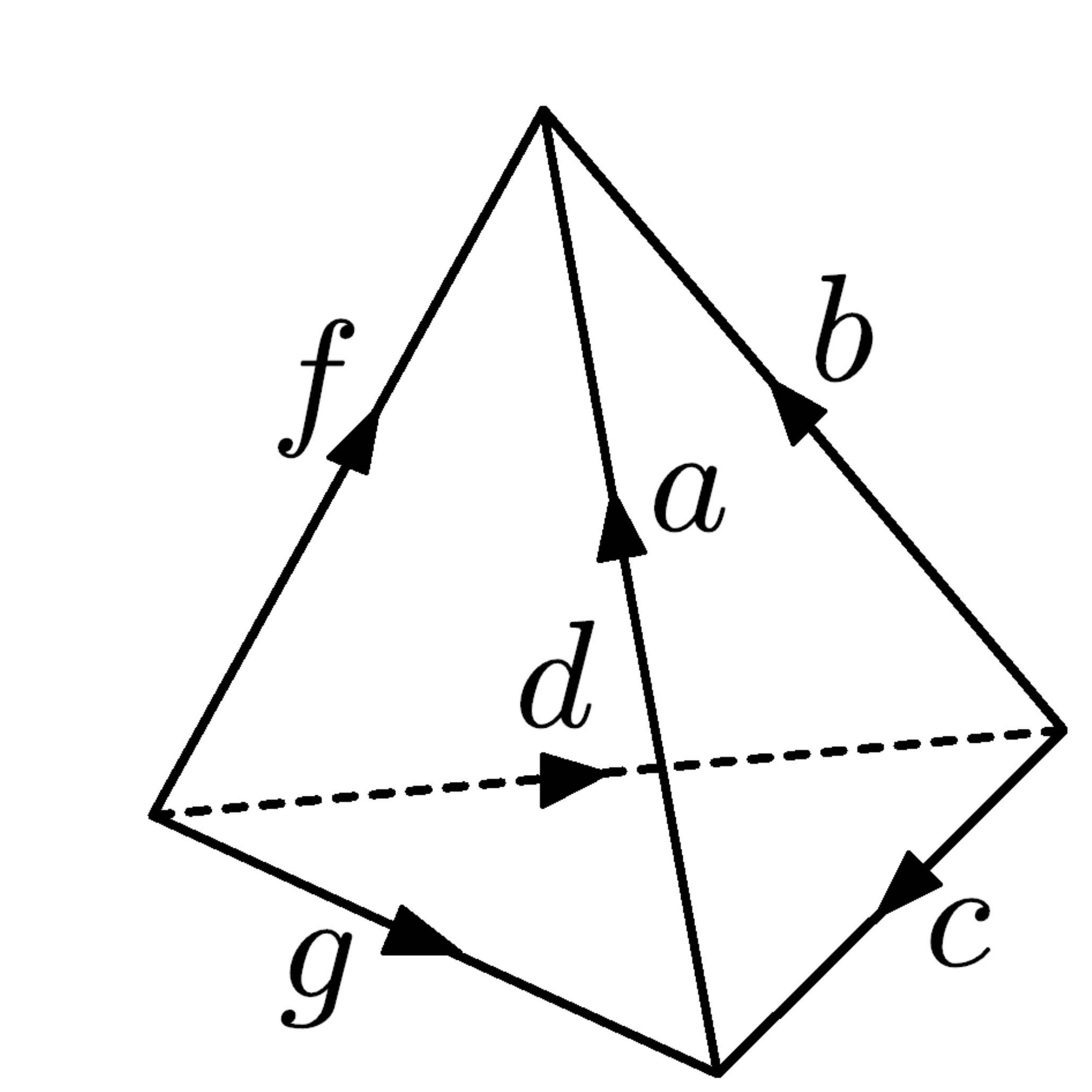}
\includegraphics[width=2.2cm]{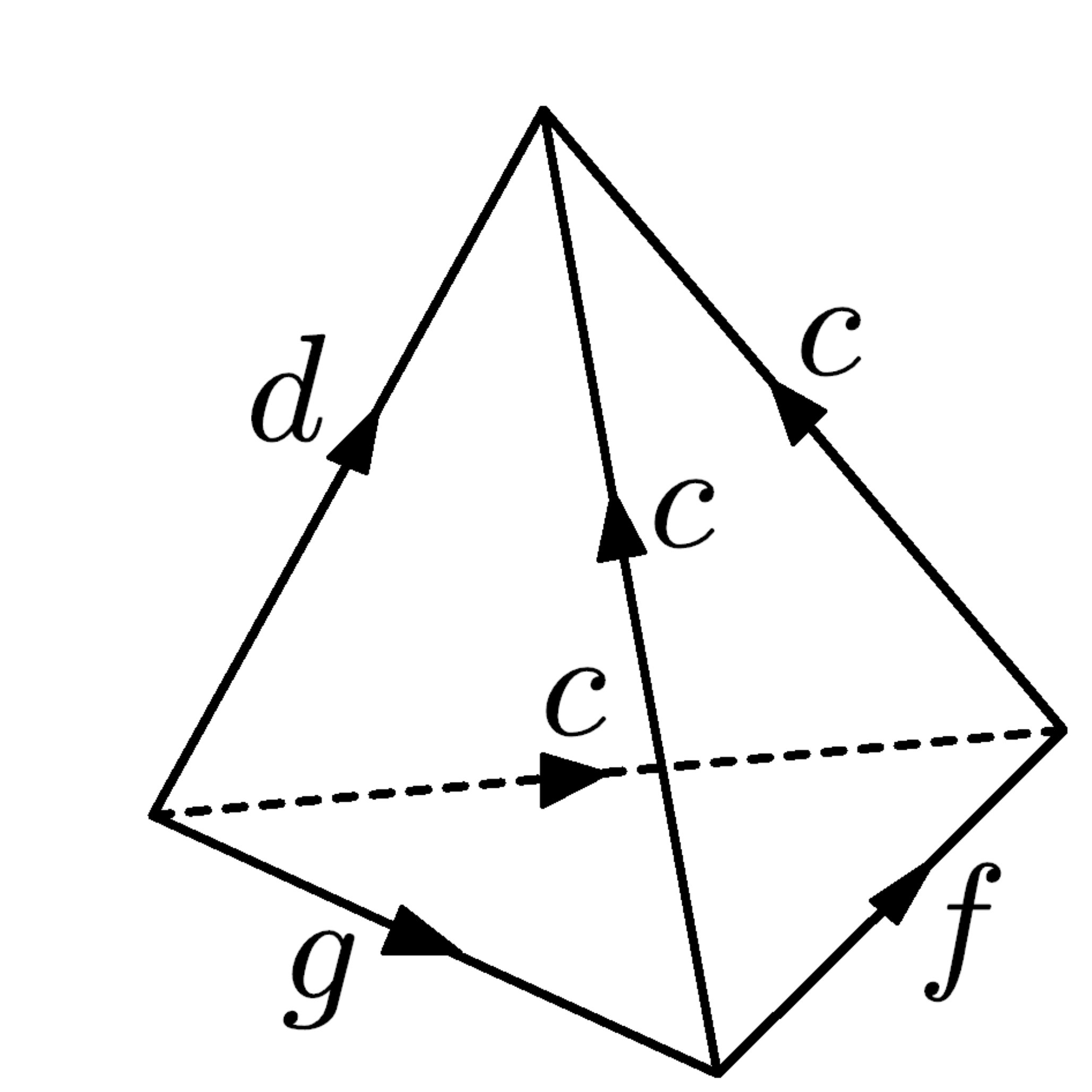}
\includegraphics[width=2.2cm]{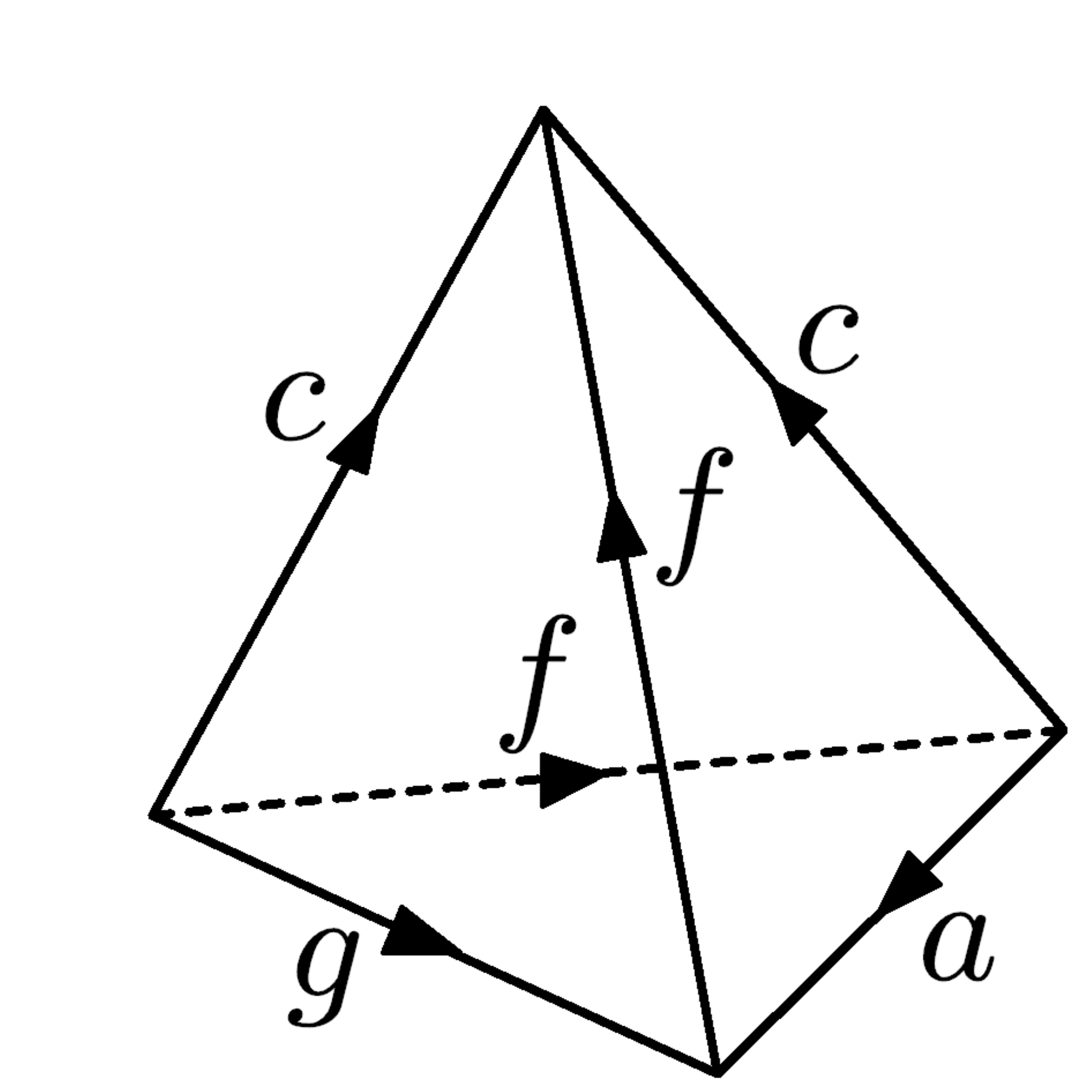}
\begin{figure}[h]
\caption{A sequence of (2,3)-Pachner moves for $m003$ to obtain a locally ordered ideal triangulation.}
\end{figure}

After positive (2,3)-Pachner moves five times, the ideal triangulation of $m003$ which consists of seven ideal tetrahedra admits the local order shown in Figure 11.  
The relations between the colors of edges are the following:
$$a=b^3,\; c=b^2,\; d=b^4,\; e=b,\; f=1,\; g=b^2,\; b^5=1.$$ 
$$Z(m003) = \sum_{b \in G, b^5=1} \alpha(b,b,b)^{-1} \alpha(b^2,b,b) \alpha(b^3,b^3,b^3)$$
$$\hspace{3.3cm} \times \alpha(b,b,b^3) \alpha(b,b^2,b^2) \alpha(b^2,b^3,b^2).$$

In order to confirm $Z(m003) \neq Z(m004)$, we calculate $Z(m003)$ for 
$G = \mathbb{Z}_5$ and a generator $\alpha$ of $H^3(\mathbb{Z}_5, U(1)) \cong  \mathbb{Z}_5$.  
$$Z(m003) = \frac{1}{2} (5+\sqrt{5}+i\sqrt{10+2\sqrt{5}}).$$
Hence the generalized DW invariants distinguish $m003$ and $m004$. 

\begin{figure}
\centering
$m006$ \hspace{0.6cm} 
\includegraphics[width=2.8cm]{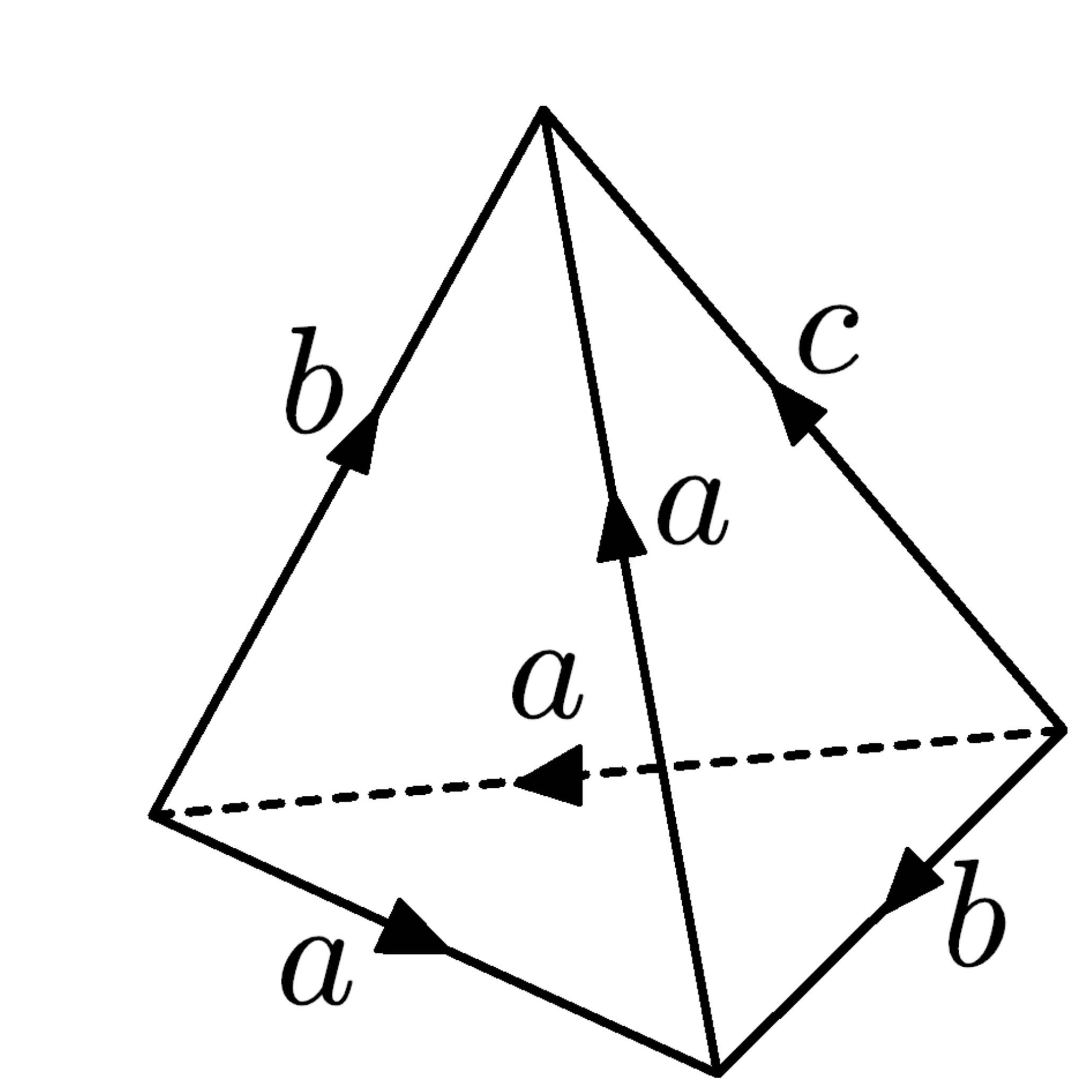}
\includegraphics[width=2.8cm]{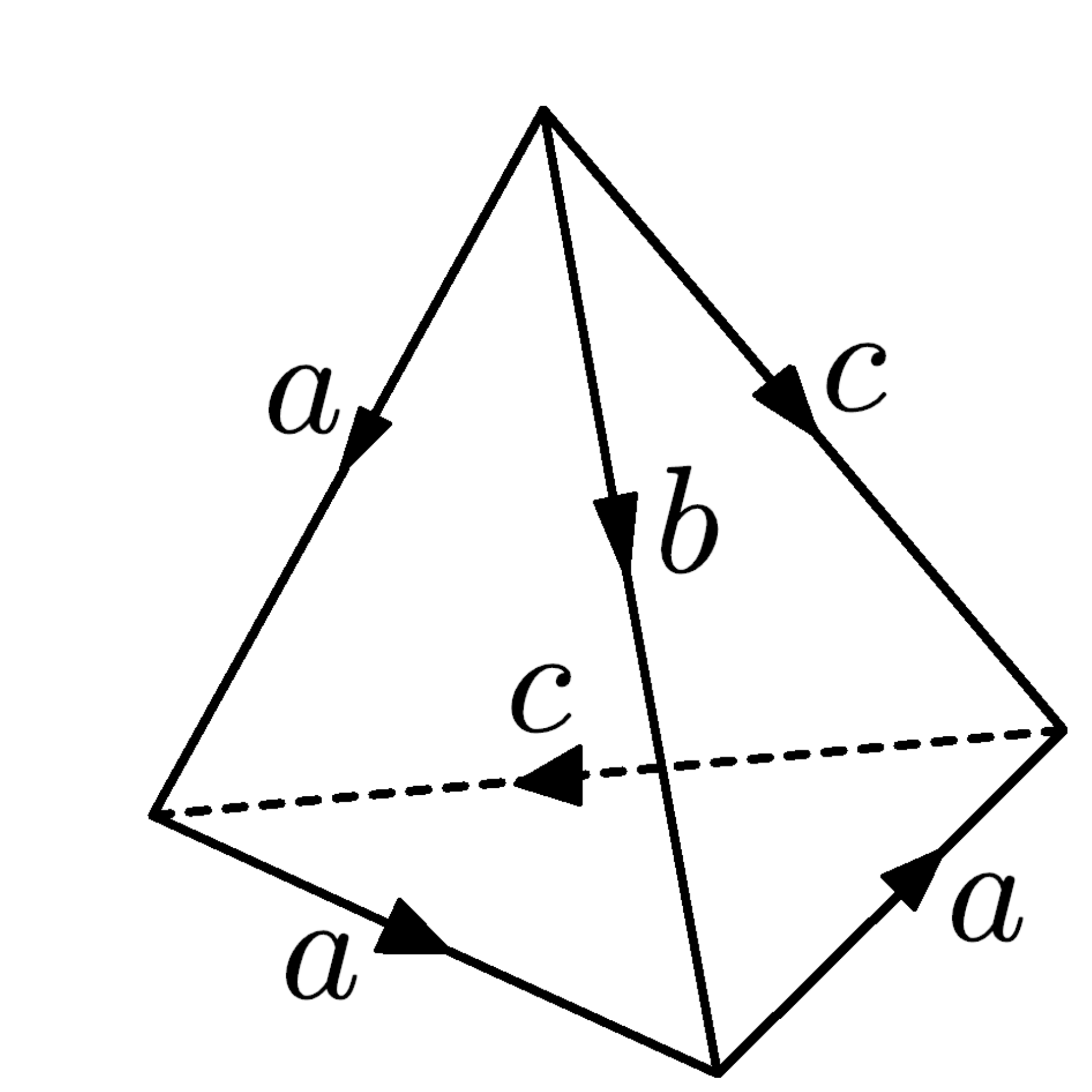} 
\includegraphics[width=2.8cm]{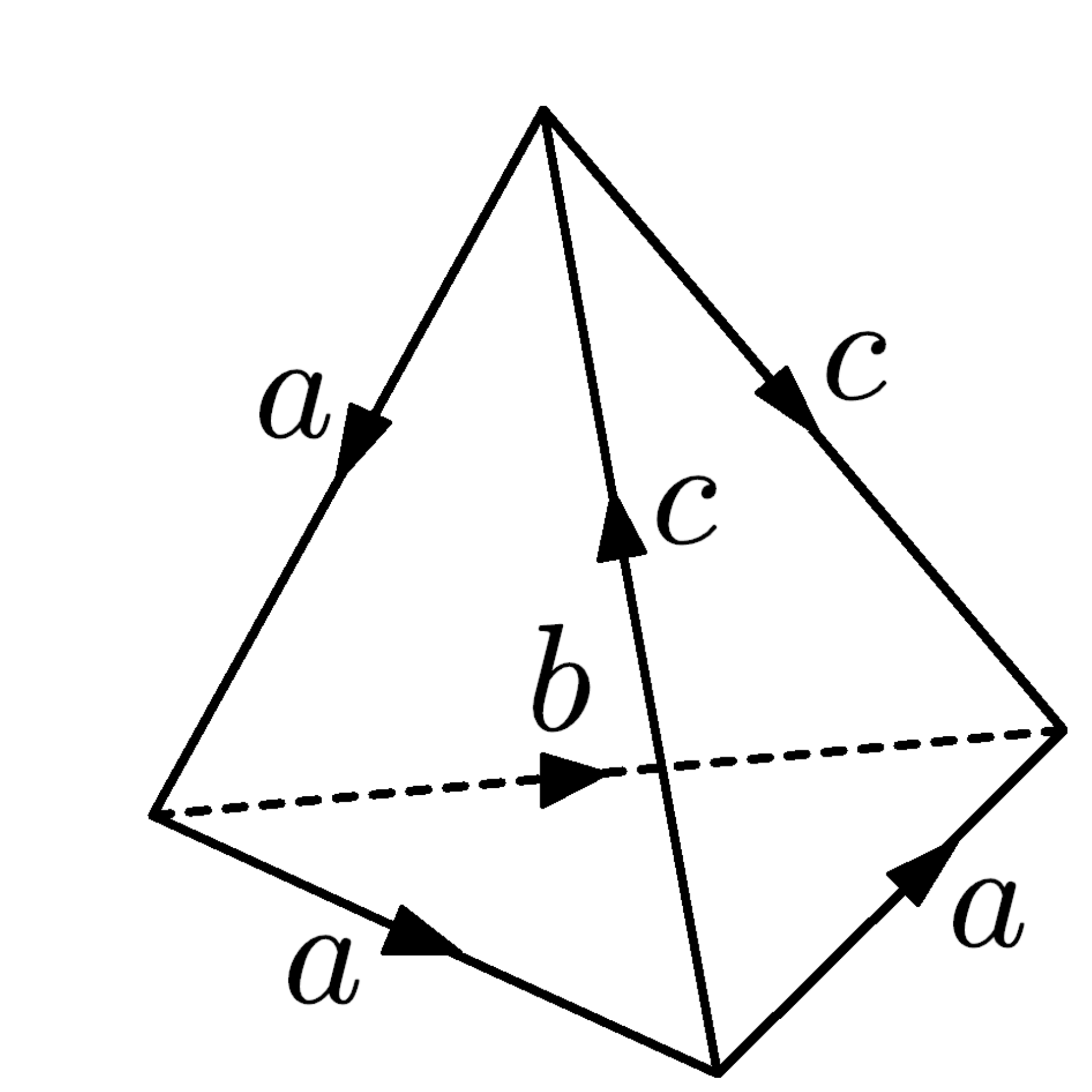}

\vspace{0.5cm}

$m007$ \hspace{0.6cm}
\includegraphics[width=2.8cm]{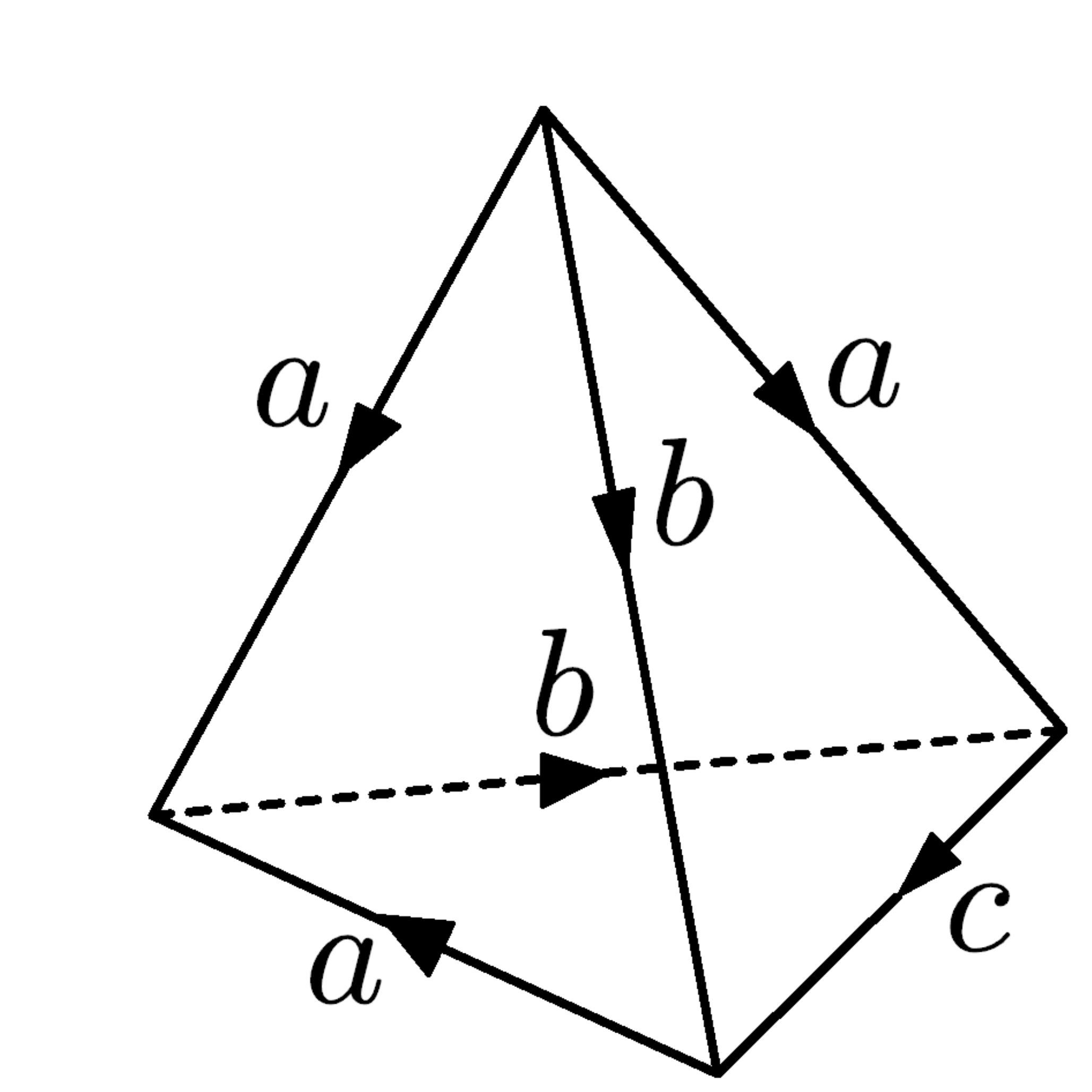} 
\includegraphics[width=2.8cm]{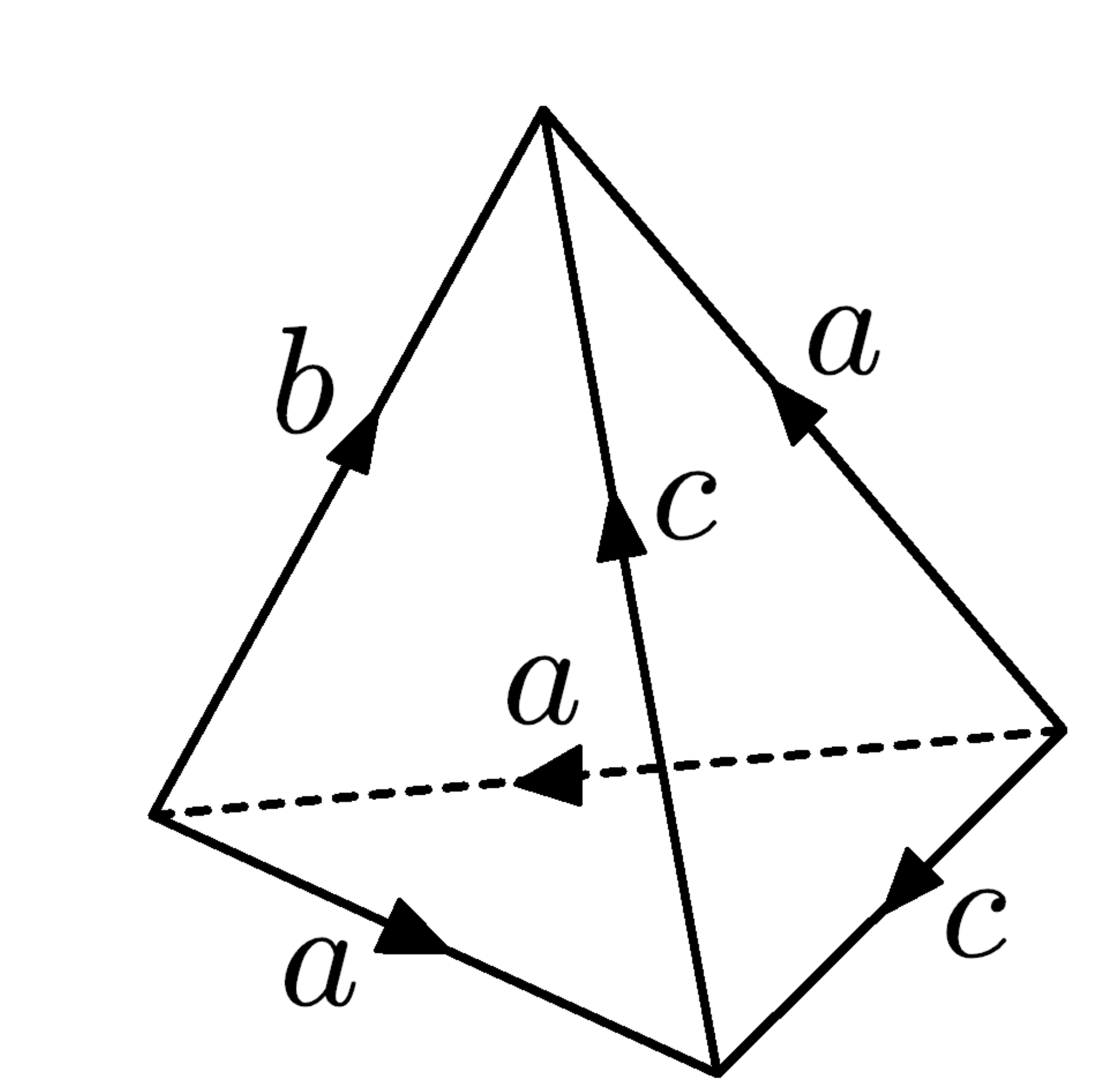}
\includegraphics[width=2.8cm]{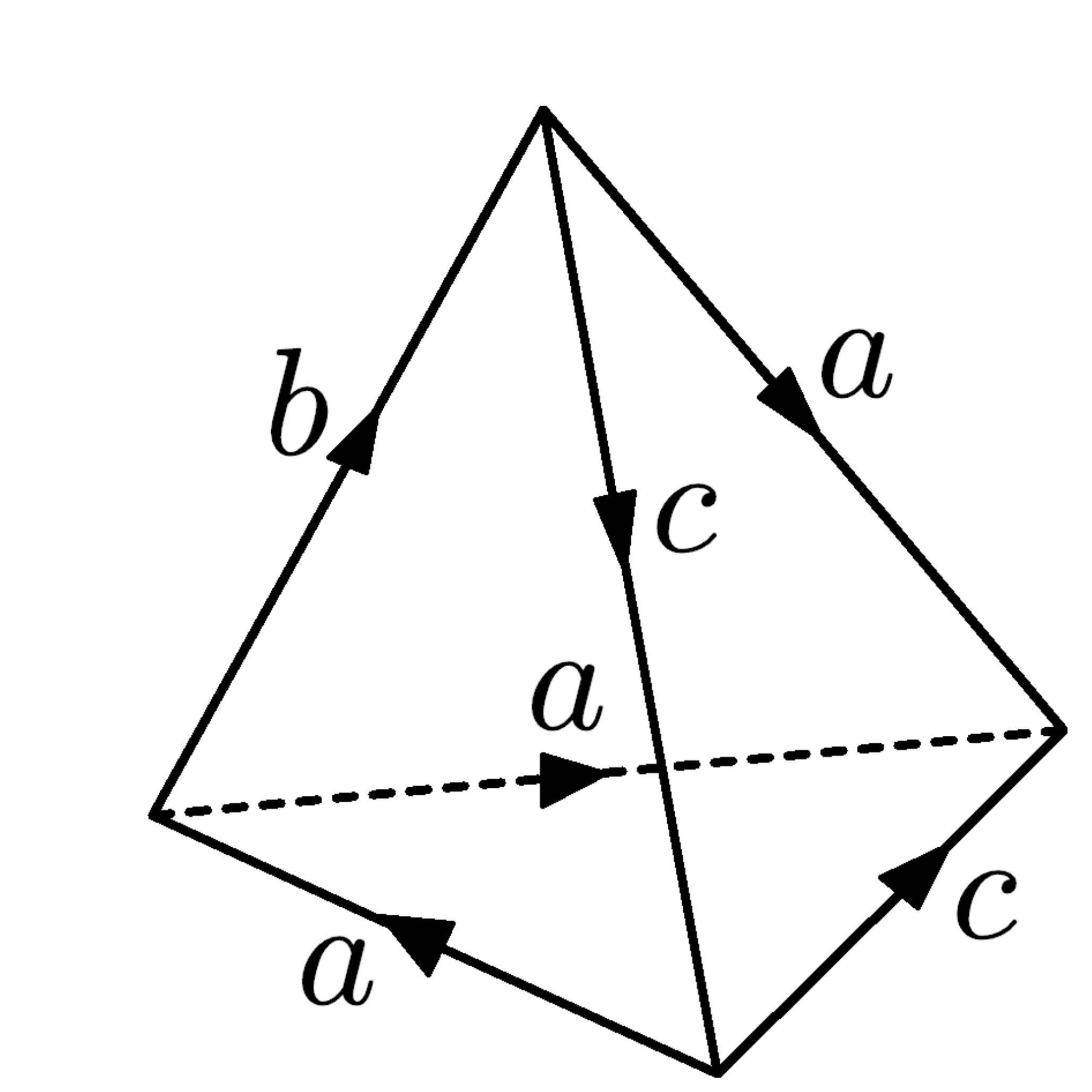} 
\caption{Minimal ideal triangulations of $m006$ and $m007$.}
\end{figure}

\newpage

(2) $m006$ and $m007$

According to Regina \cite{Burton} and SnapPy \cite{Culler}, $m006$ and $m007$ are cusped orientable 3-manifolds 
with the minimal ideal triangulations shown in Figure 12.
Their hyperbolic volumes, Turaev-Viro invariants and homology groups are as follows:
$${\rm Vol}(m006) = {\rm Vol}(m007) \approx 2.56897,$$
$$
TV(m006) = \sum w_a w_b w_c \begin{vmatrix} a & b & c \\ a & b & a \end{vmatrix} \begin{vmatrix} a & b & c \\ a & c & a \end{vmatrix}
\begin{vmatrix} a & b & c \\ a & c & a \end{vmatrix} = TV(m007),
$$
$$H_1(m006 ;\mathbb{Z}) = \mathbb{Z} \oplus \mathbb{Z}_5,\quad H_1(m007 ; \mathbb{Z}) = \mathbb{Z} \oplus \mathbb{Z}_3.$$










$$Z(m006) = \sum_{a \in G, a^5=1} \alpha(a,a,a)^3 \alpha(a,a^2,a) \alpha(a^3,a^3,a^3).$$









$$Z(m007) = \sum_{a \in G, a^3=1} \alpha(a,a,a) \alpha(a^{-1},a^{-1},a^{-1}).$$

If $G = \mathbb{Z}_5$ and $\alpha$ is a generator of $H^3(\mathbb{Z}_5, U(1)) \cong  \mathbb{Z}_5$,
$$Z(m006) = -\frac{\sqrt{5}}{2}+\frac{i}{4}(\sqrt{10+2\sqrt{5}}-\sqrt{10-2\sqrt{5}}),\quad Z(m007) = 1.$$ 
Hence the generalized DW invariants distinguish $m006$ and $m007$.

\begin{figure}
\centering
$m009$ \hspace{0.6cm} 
\includegraphics[width=2.8cm]{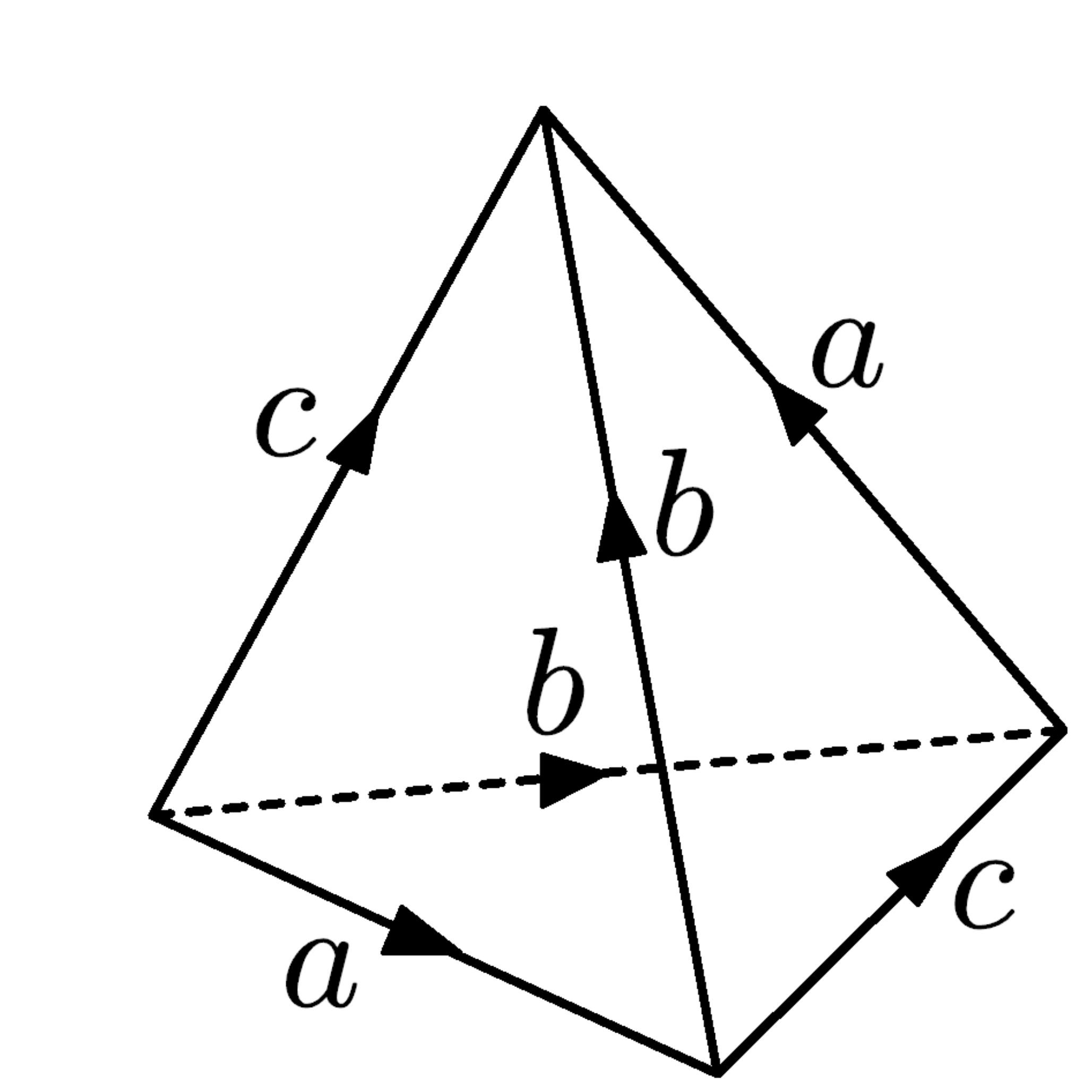}
\includegraphics[width=2.8cm]{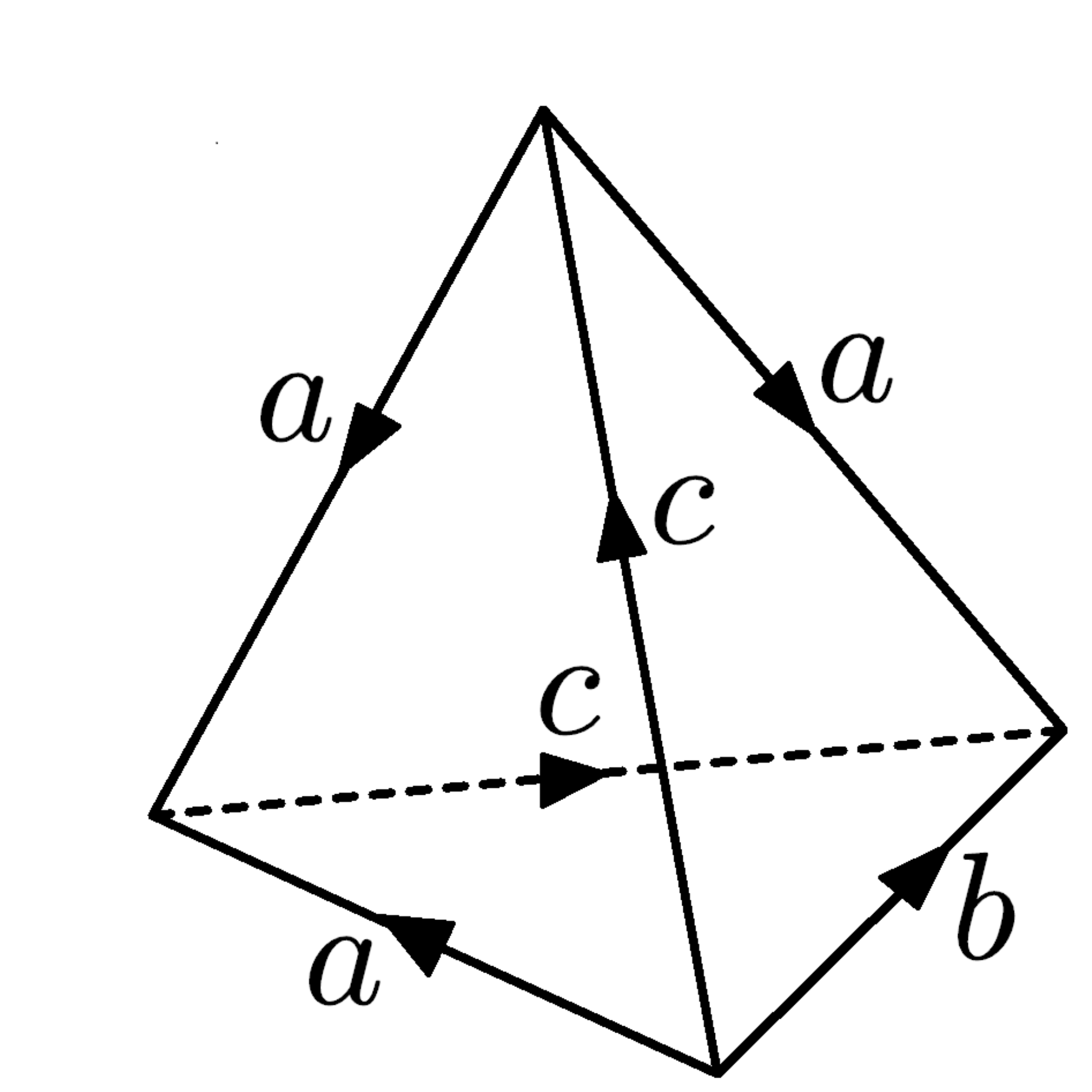} 
\includegraphics[width=2.8cm]{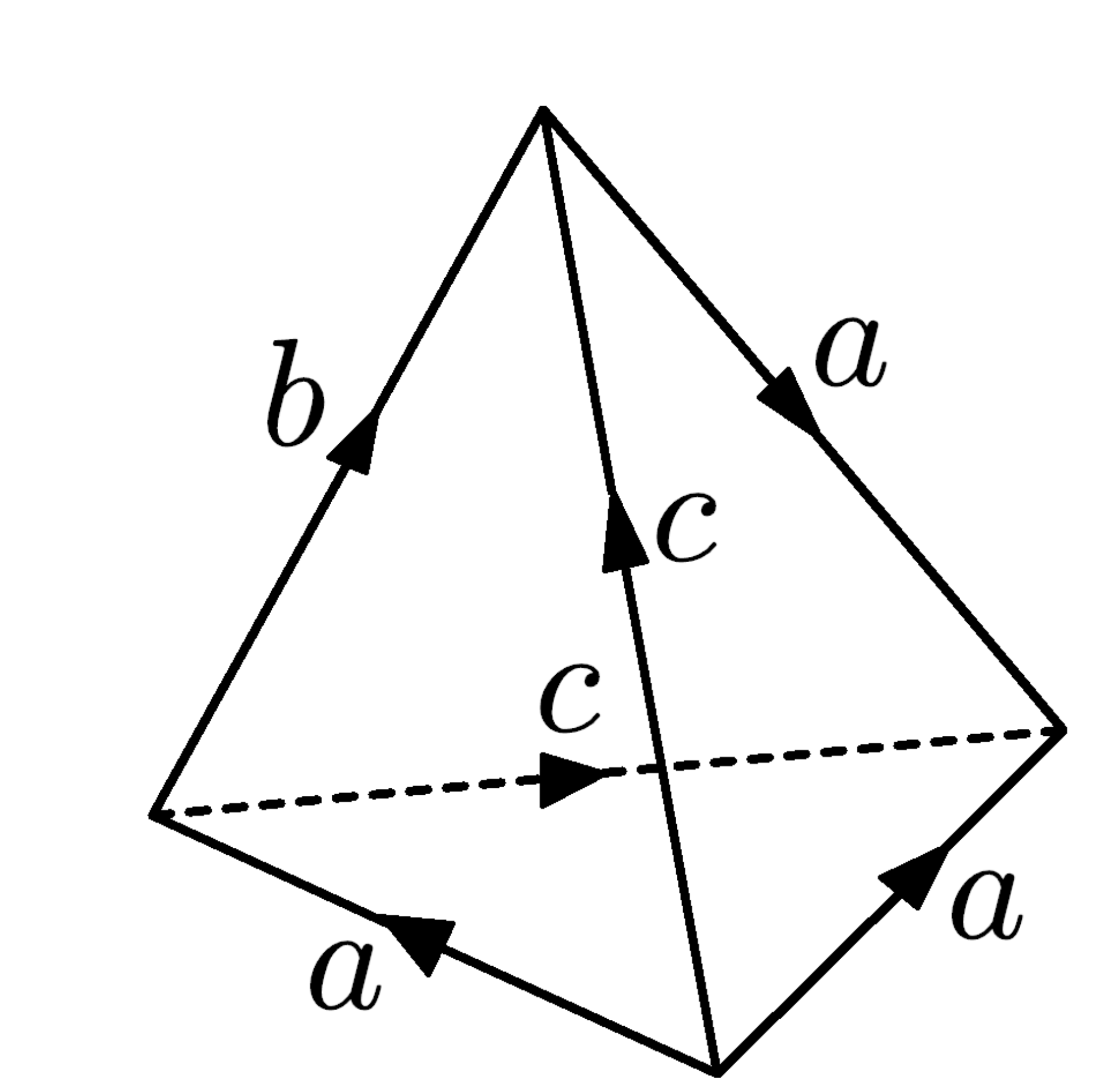}

\vspace{0.5cm}

$m010$ \hspace{0.6cm}
\includegraphics[width=2.8cm]{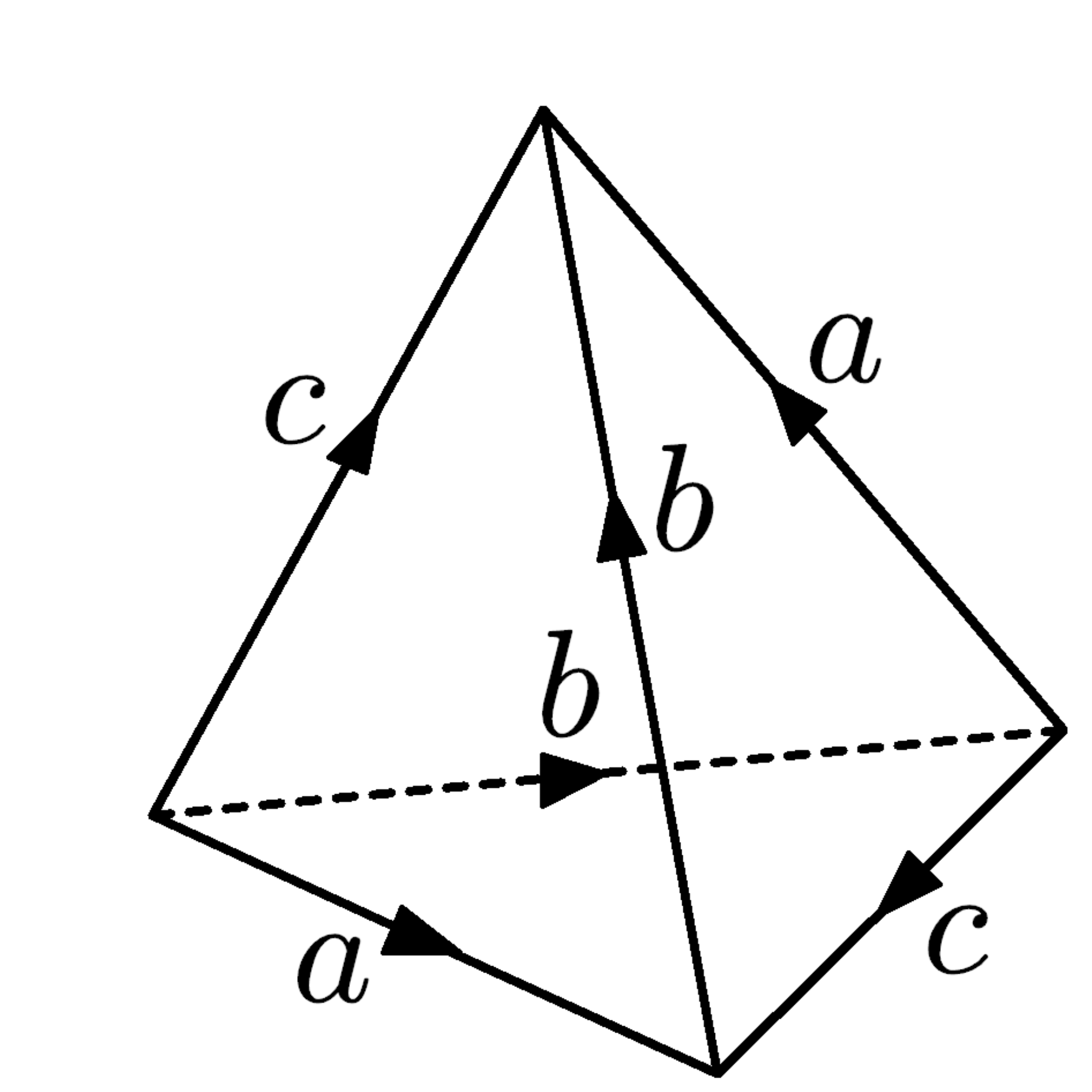} 
\includegraphics[width=2.8cm]{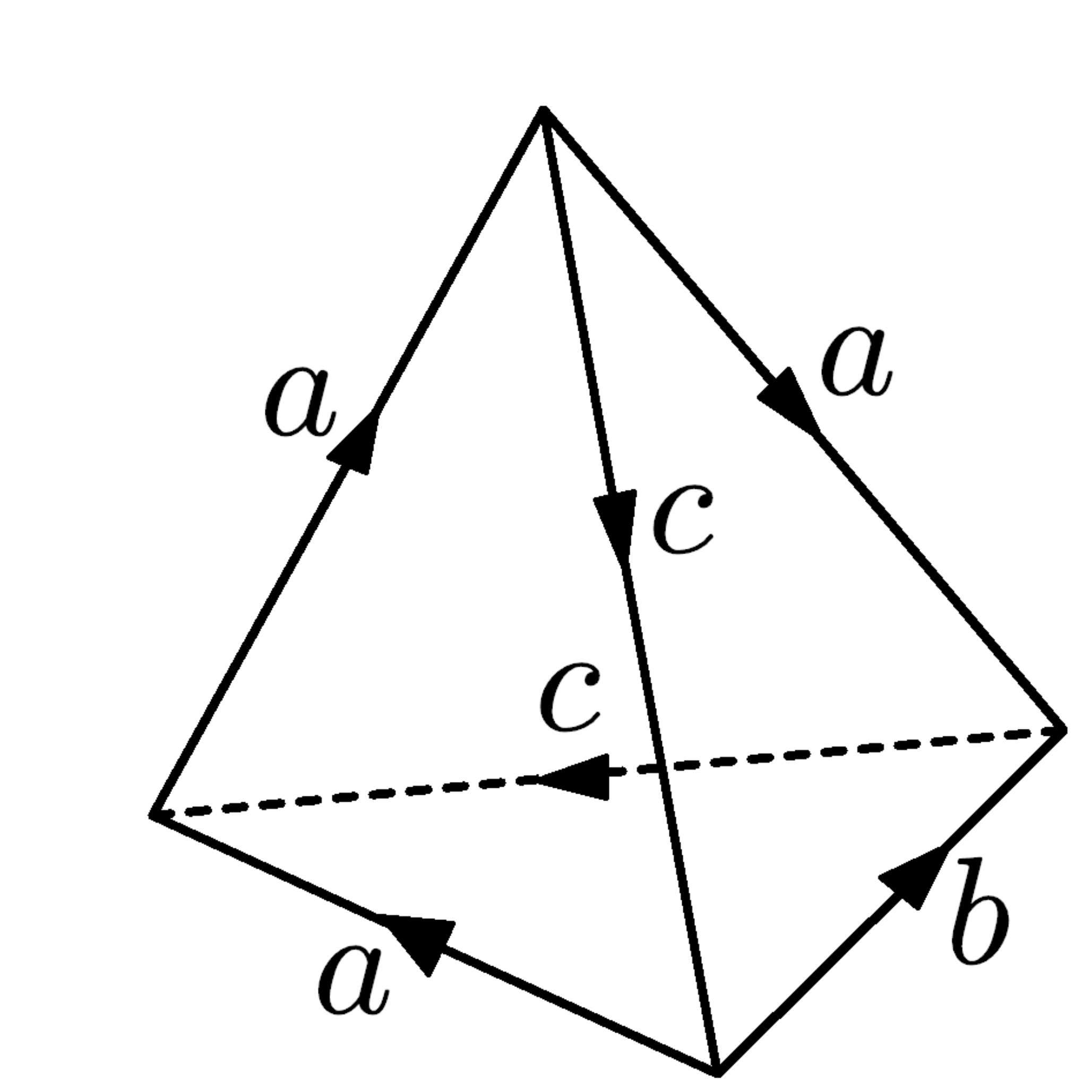}
\includegraphics[width=2.8cm]{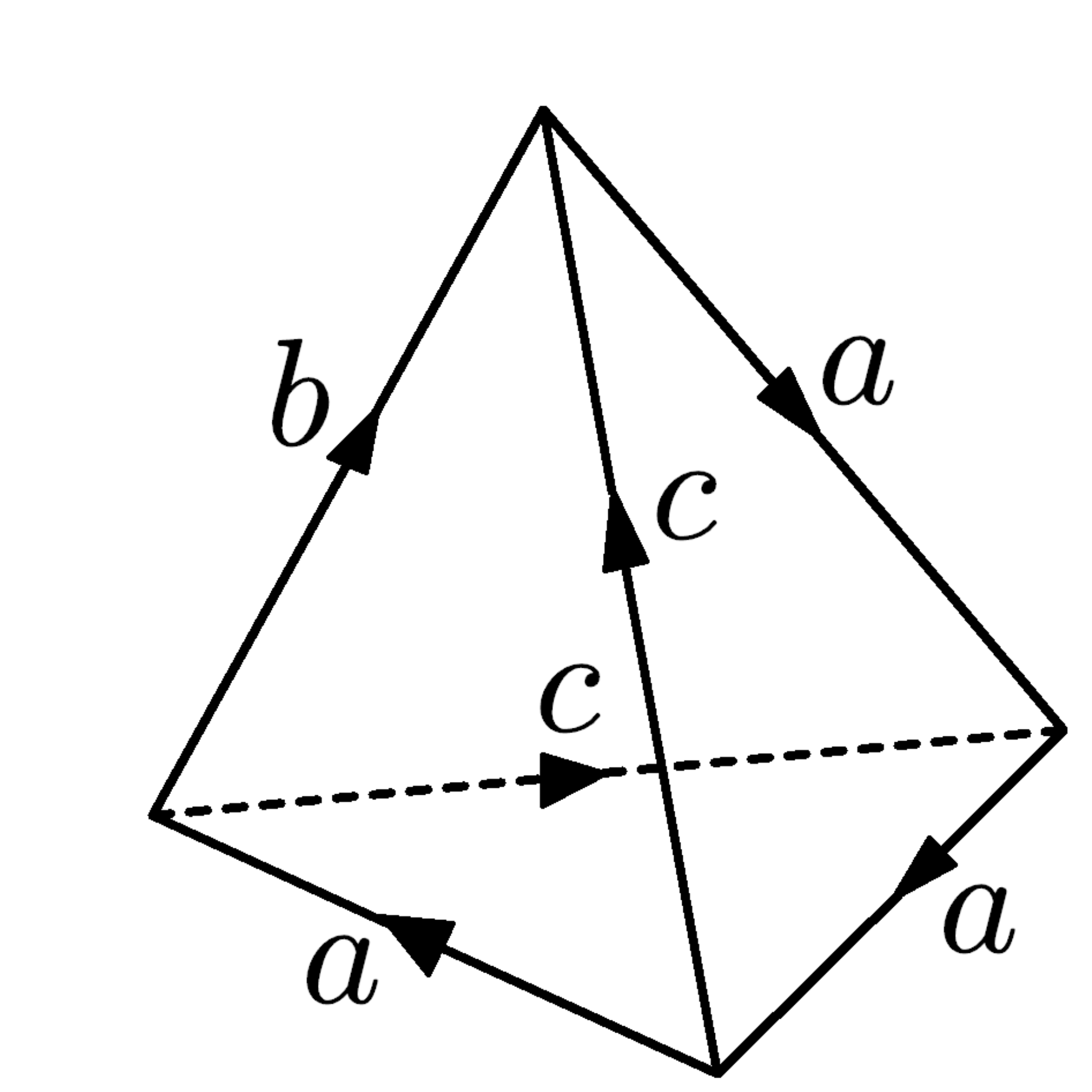} 
\caption{Minimal ideal triangulations of $m009$ and $m010$.}
\end{figure}

(3) $m009$ and $m010$

According to Regina \cite{Burton} and SnapPy \cite{Culler}, $m009$ and $m010$ are cusped orientable 3-manifolds 
with the minimal ideal triangulations shown in Figure 13.
Their hyperbolic volumes, Turaev-Viro invariants and homology groups are as follows:
$${\rm Vol}(m009) = {\rm Vol}(m010) \approx 2.66674,$$
$$
TV(m009) = \sum w_a w_b w_c \begin{vmatrix} a & b & c \\ a & b & c \end{vmatrix} \begin{vmatrix} a & b & c \\ a & a & c \end{vmatrix}
\begin{vmatrix} a & b & c \\ a & a & c \end{vmatrix} = TV(m010),
$$
$$H_1(m009 ;\mathbb{Z}) = \mathbb{Z} \oplus \mathbb{Z}_2,\quad H_1(m010 ; \mathbb{Z}) = \mathbb{Z} \oplus \mathbb{Z}_6.$$



$$Z(m009) = \sum_{a \in G, a^2=1} \alpha(a,a,a).$$

$$Z(m010) = \sum_{b,c \in G, b^3=1,c^2=1,bc=cb} \alpha(b,b,b)^{-1} \alpha(b,b^{-1},b) \alpha(b,c,c) \alpha(cb,c,c)^{-1}.$$

If $G = \mathbb{Z}_3$ and $\alpha$ is a generator of $H^3(\mathbb{Z}_3, U(1)) \cong  \mathbb{Z}_3$,

$$Z(m009) = 1, \quad Z(m010) = -\sqrt{3}i.$$ 

Hence the generalized DW invariants distinguish $m009$ and $m010$.

In fact the previous three pairs of cusped hyperbolic 3-manifolds with the same hyperbolic volumes 
and the same Turaev-Viro invariants are distinguished by their homology groups.
The following pair of cusped hyperbolic 3-manifolds with the same hyperbolic volumes and the same homology groups have distinct generalized DW invariants.

\begin{figure}
\centering
$s778$ \hspace{0.6cm} 
\includegraphics[width=2.6cm]{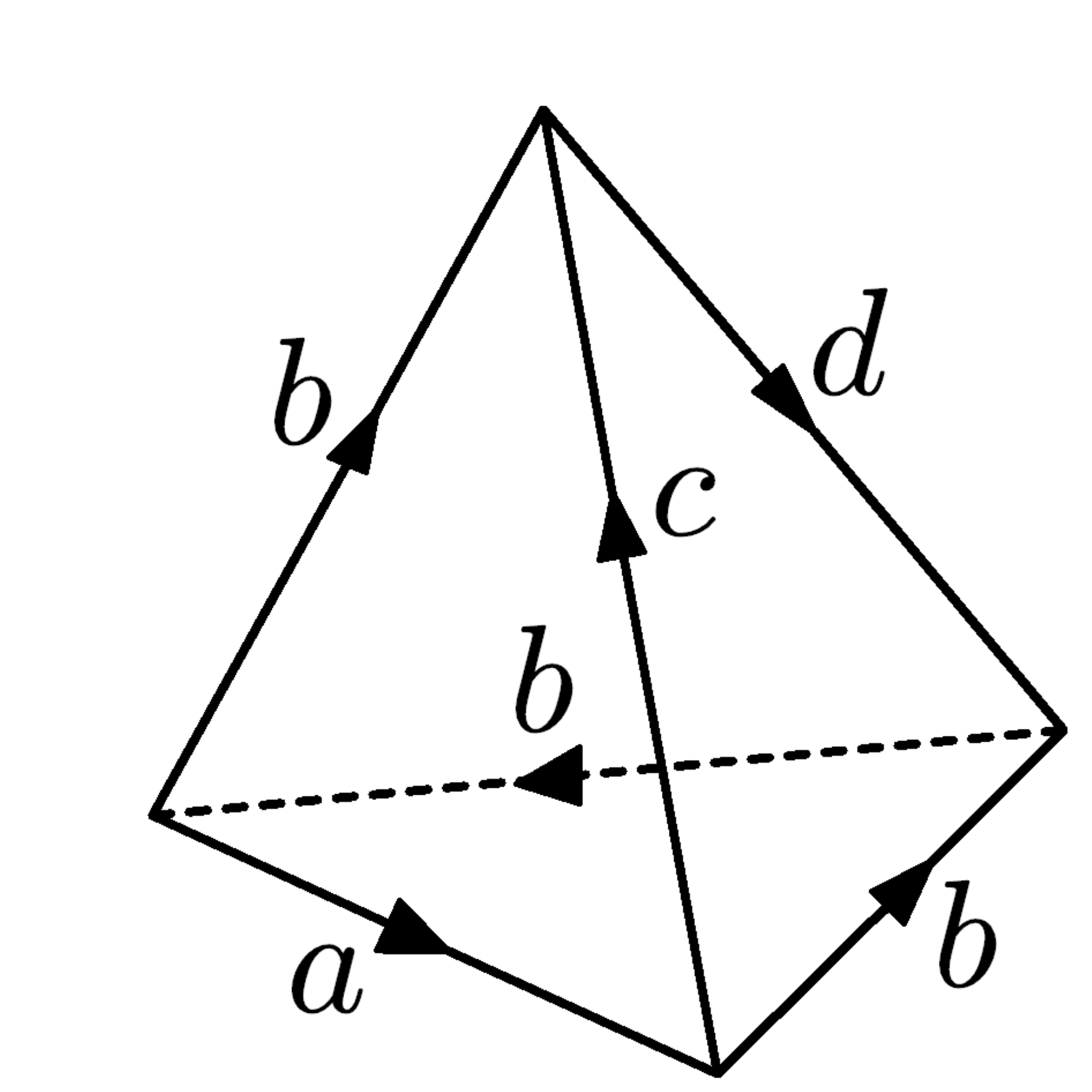}
\includegraphics[width=2.6cm]{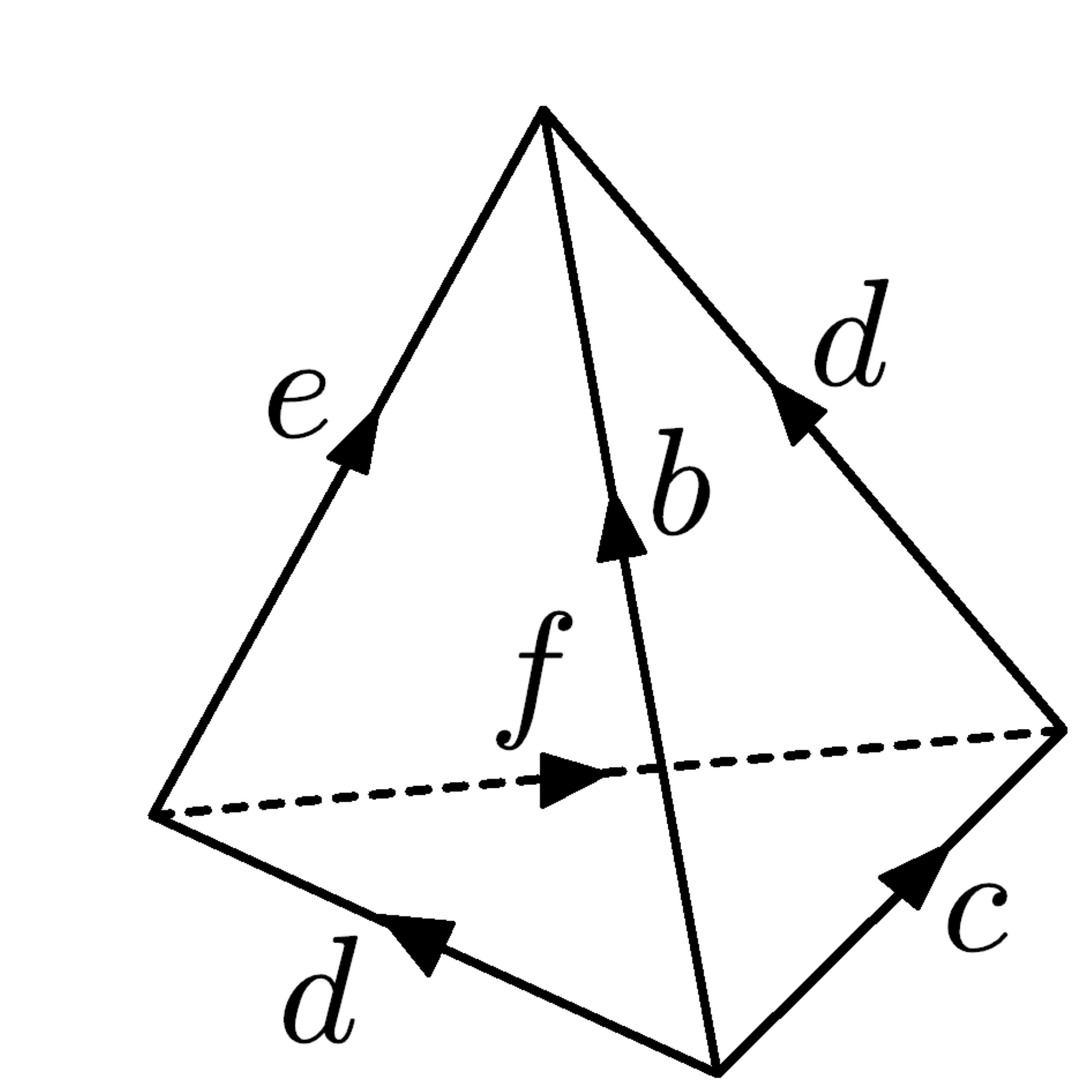} 
\includegraphics[width=2.6cm]{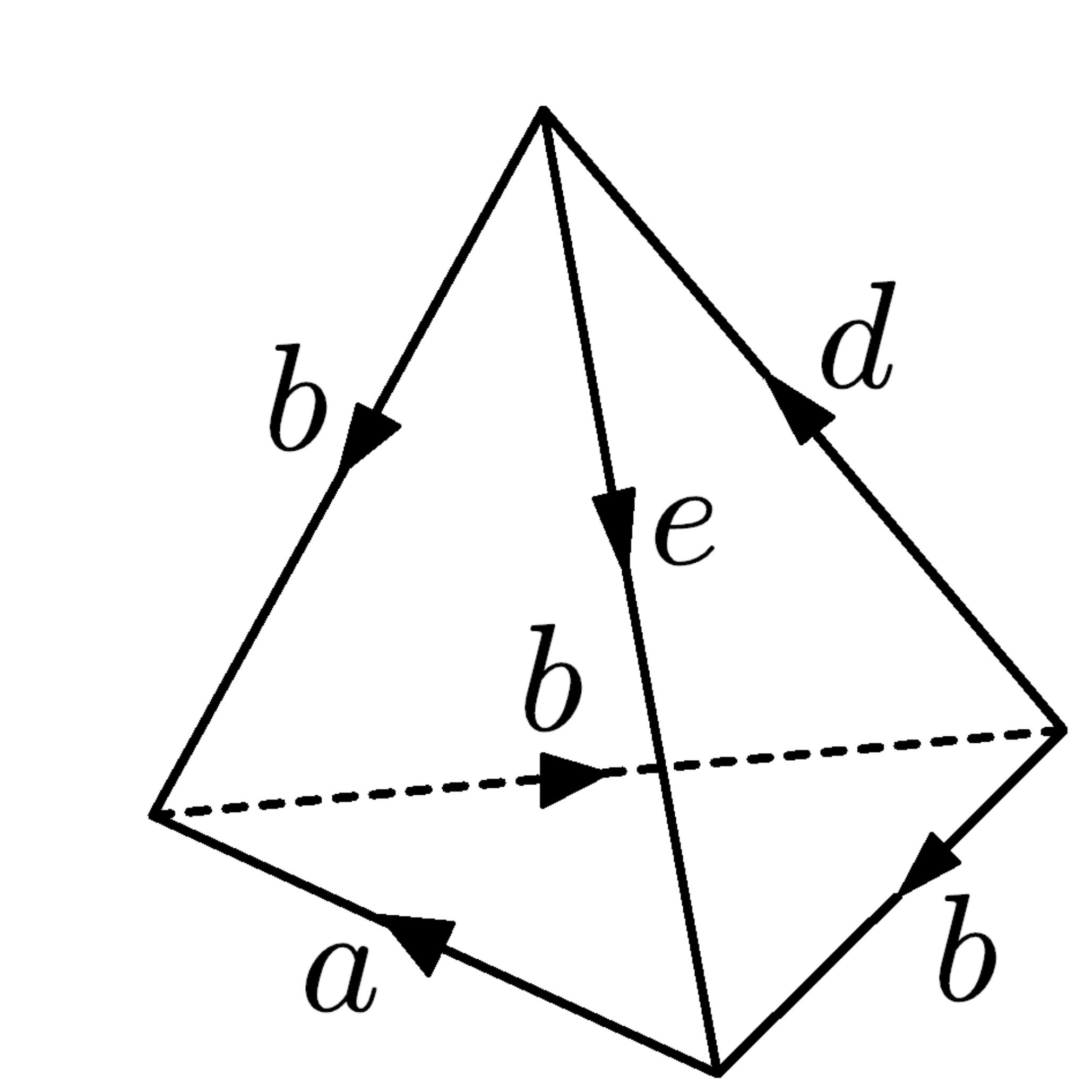}

\vspace{0.3cm}

\hspace{1cm}
\includegraphics[width=2.6cm]{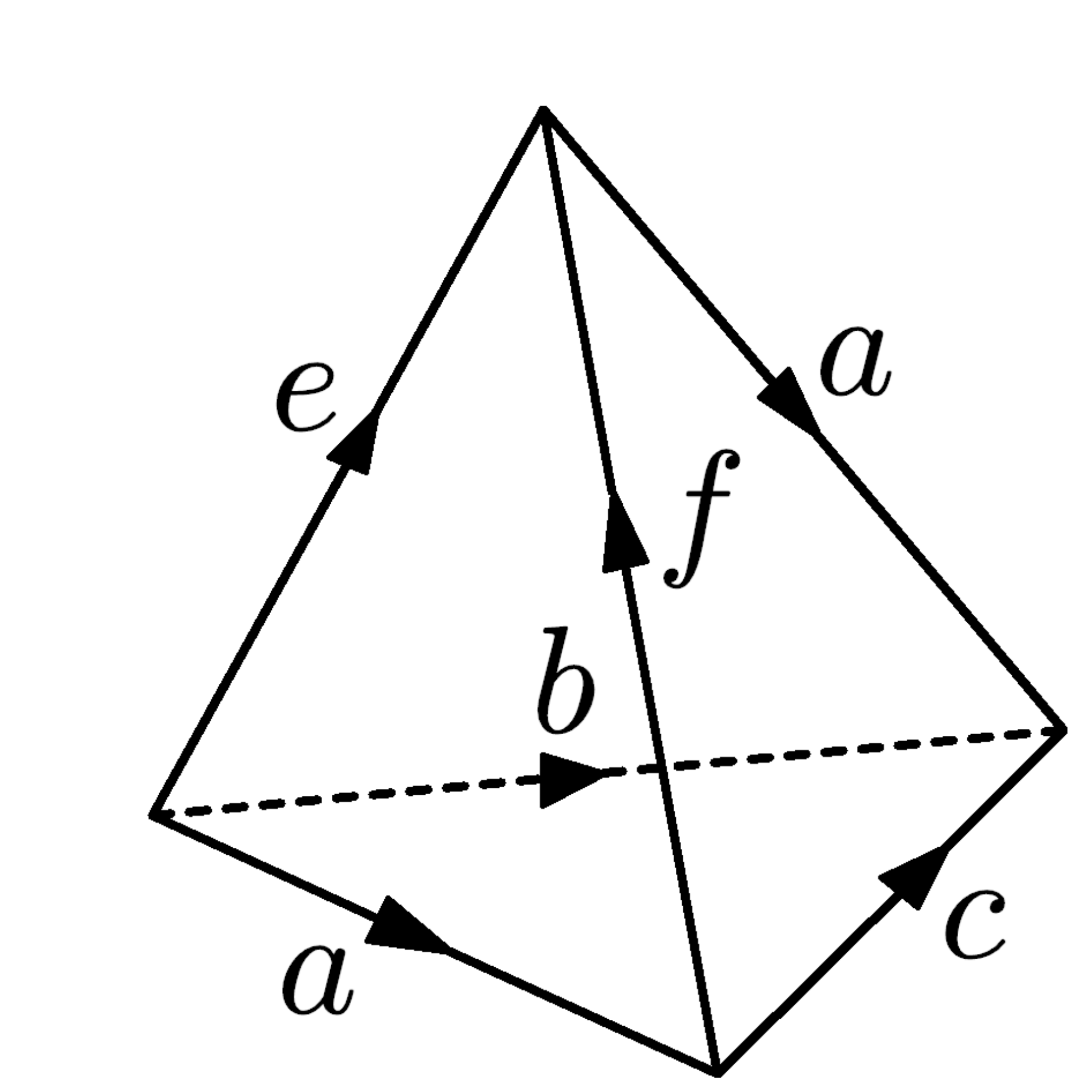} 
\includegraphics[width=2.6cm]{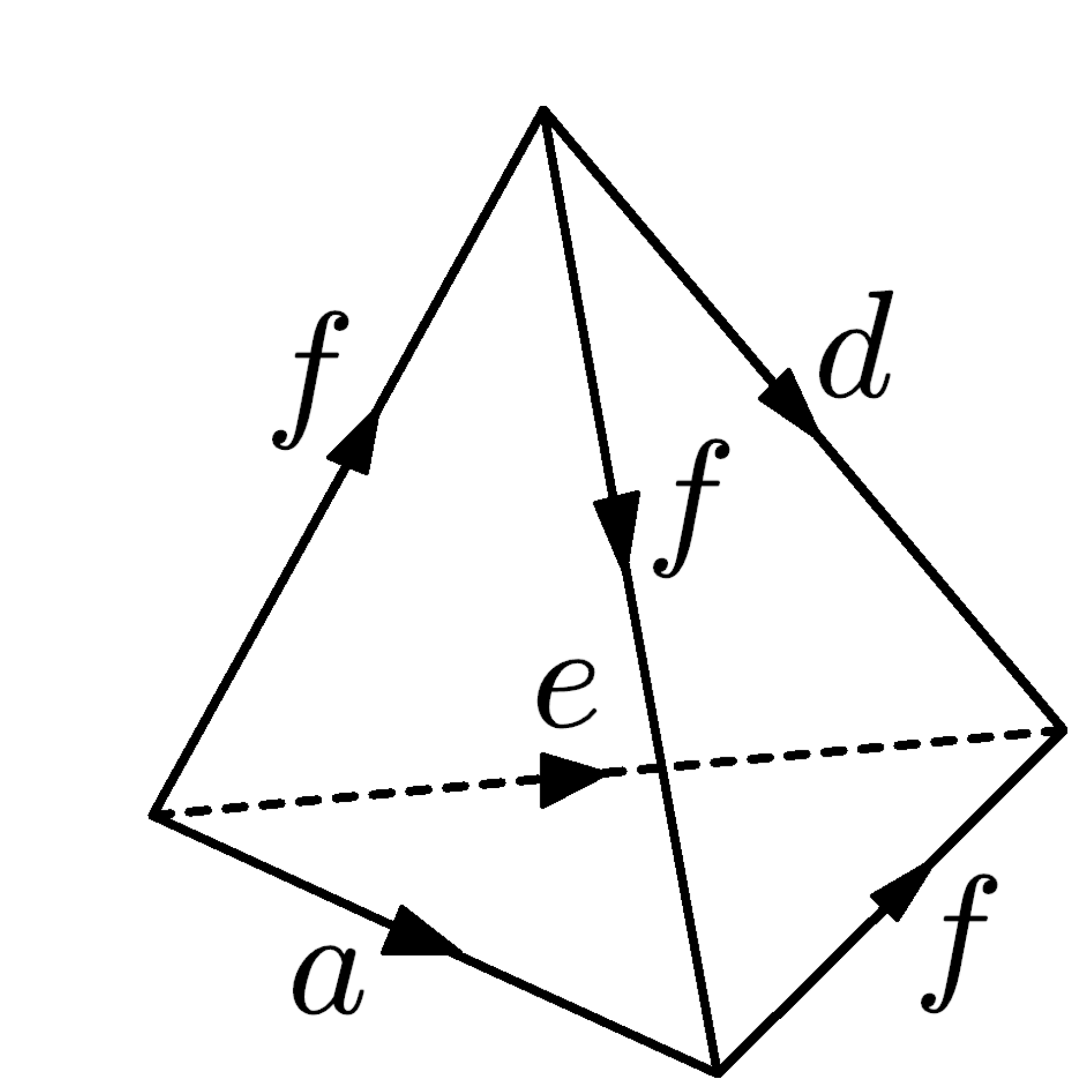}
\includegraphics[width=2.6cm]{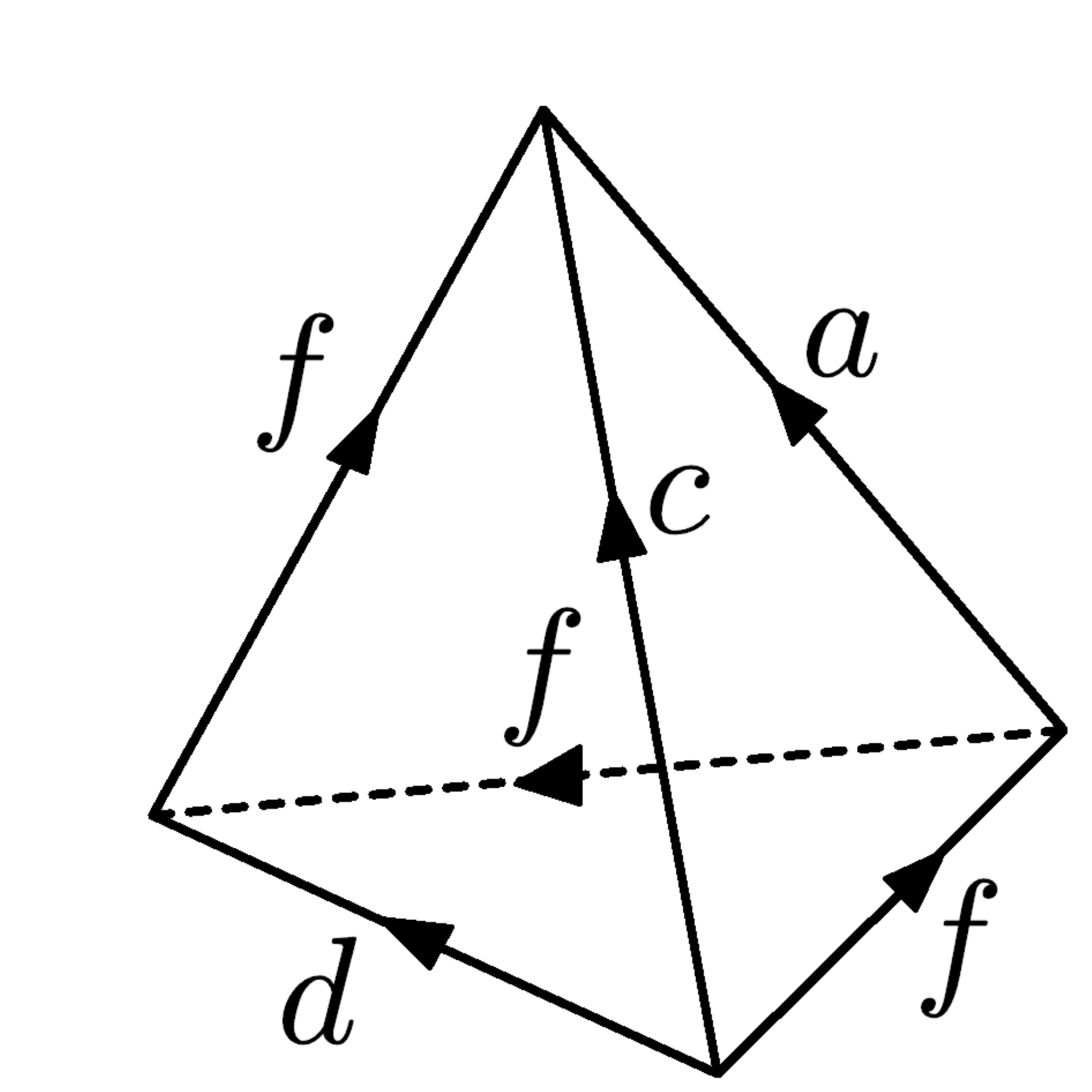} 
\caption{A minimal ideal triangulation of $s778$.}
\end{figure}

\begin{figure}
\centering
$s788$ \hspace{0.6cm} 
\includegraphics[width=2.6cm]{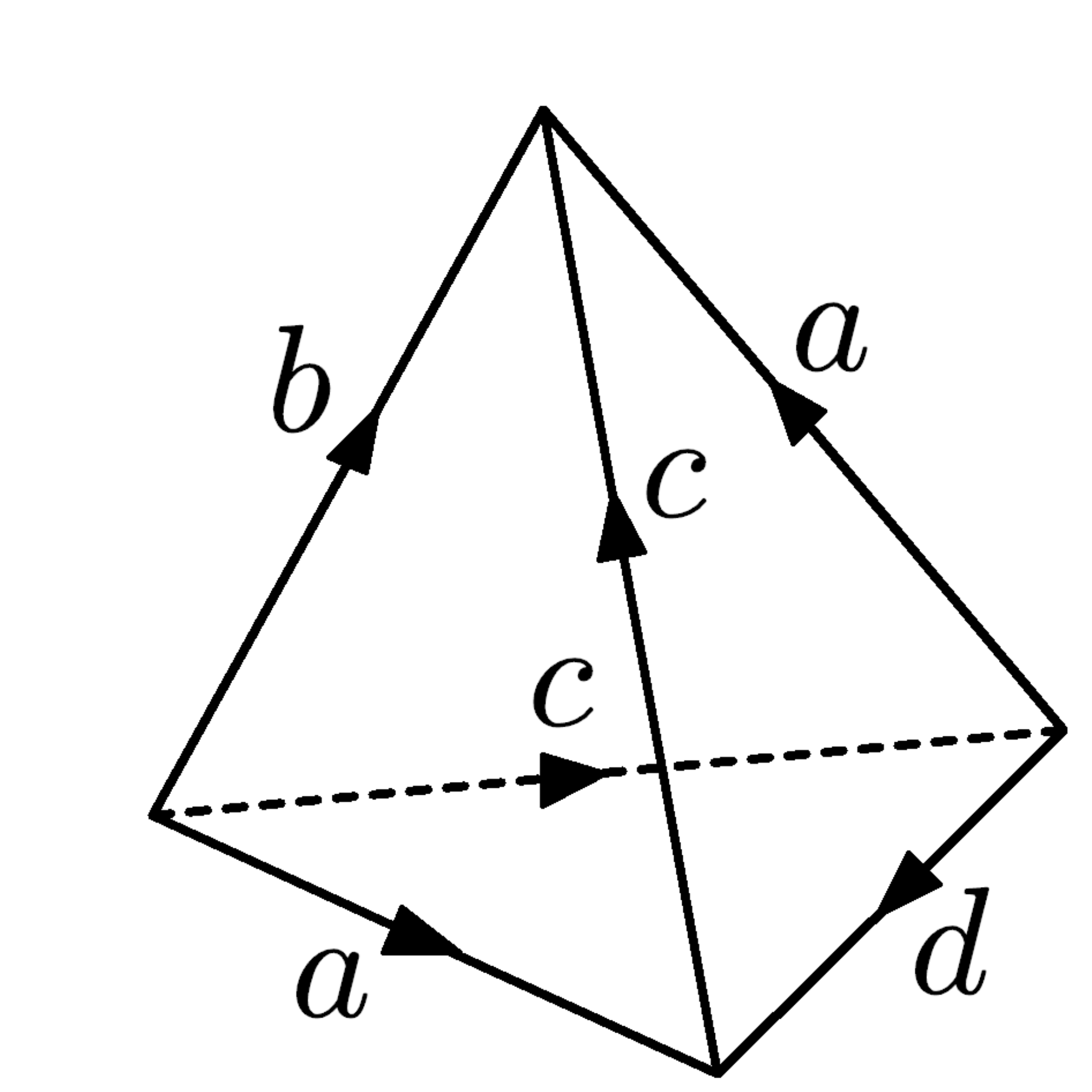}
\includegraphics[width=2.6cm]{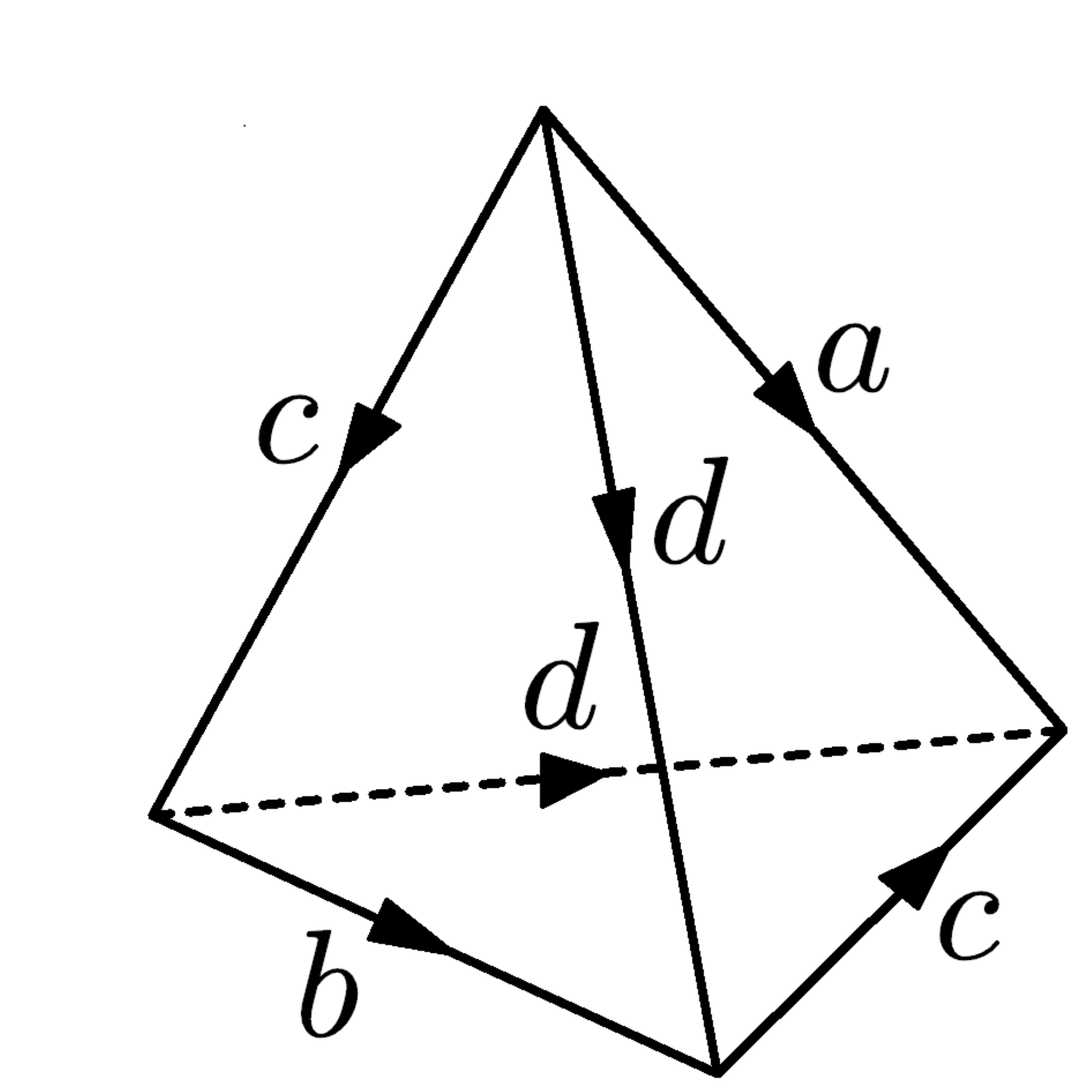} 
\includegraphics[width=2.6cm]{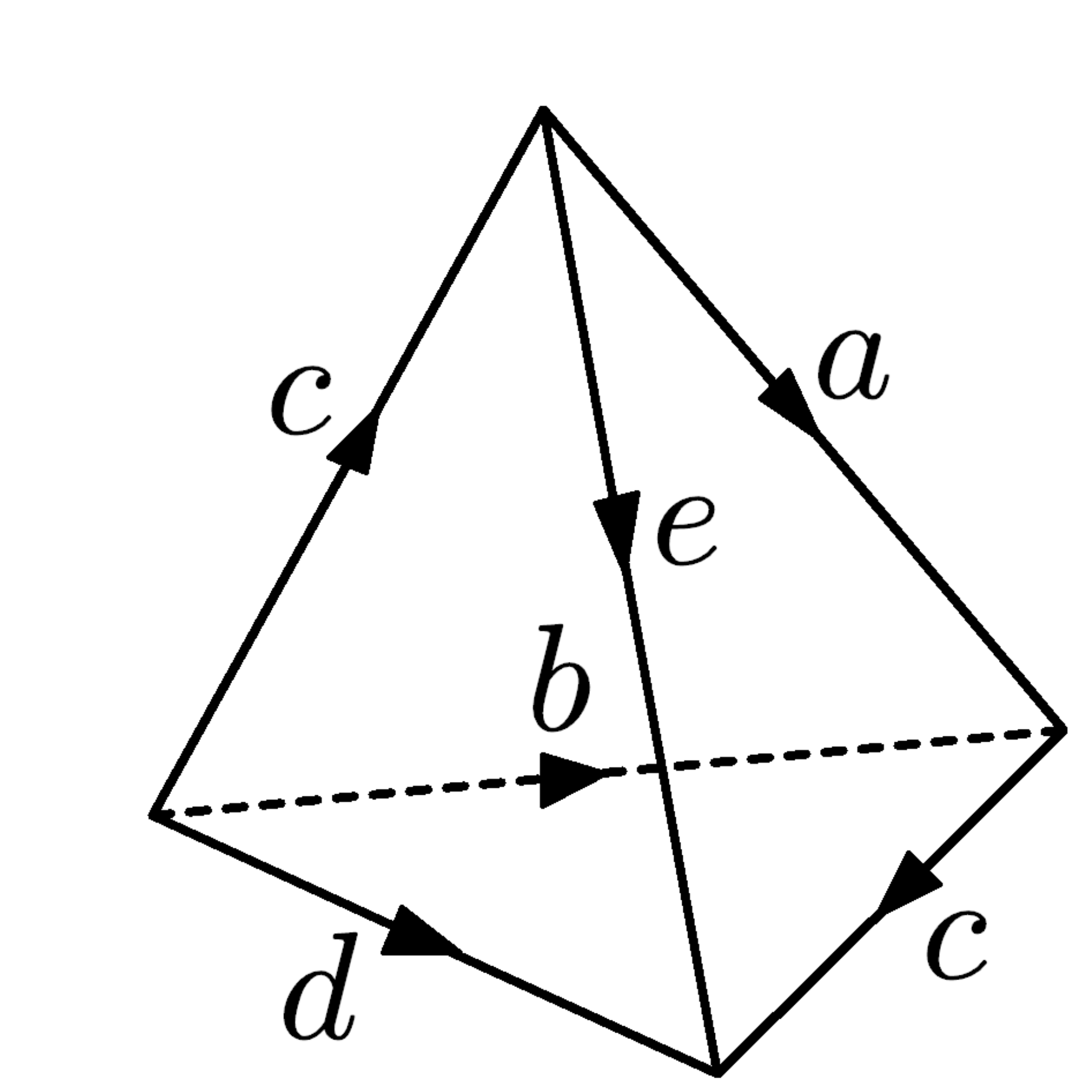}

\vspace{0.3cm}

\hspace{1cm}
\includegraphics[width=2.6cm]{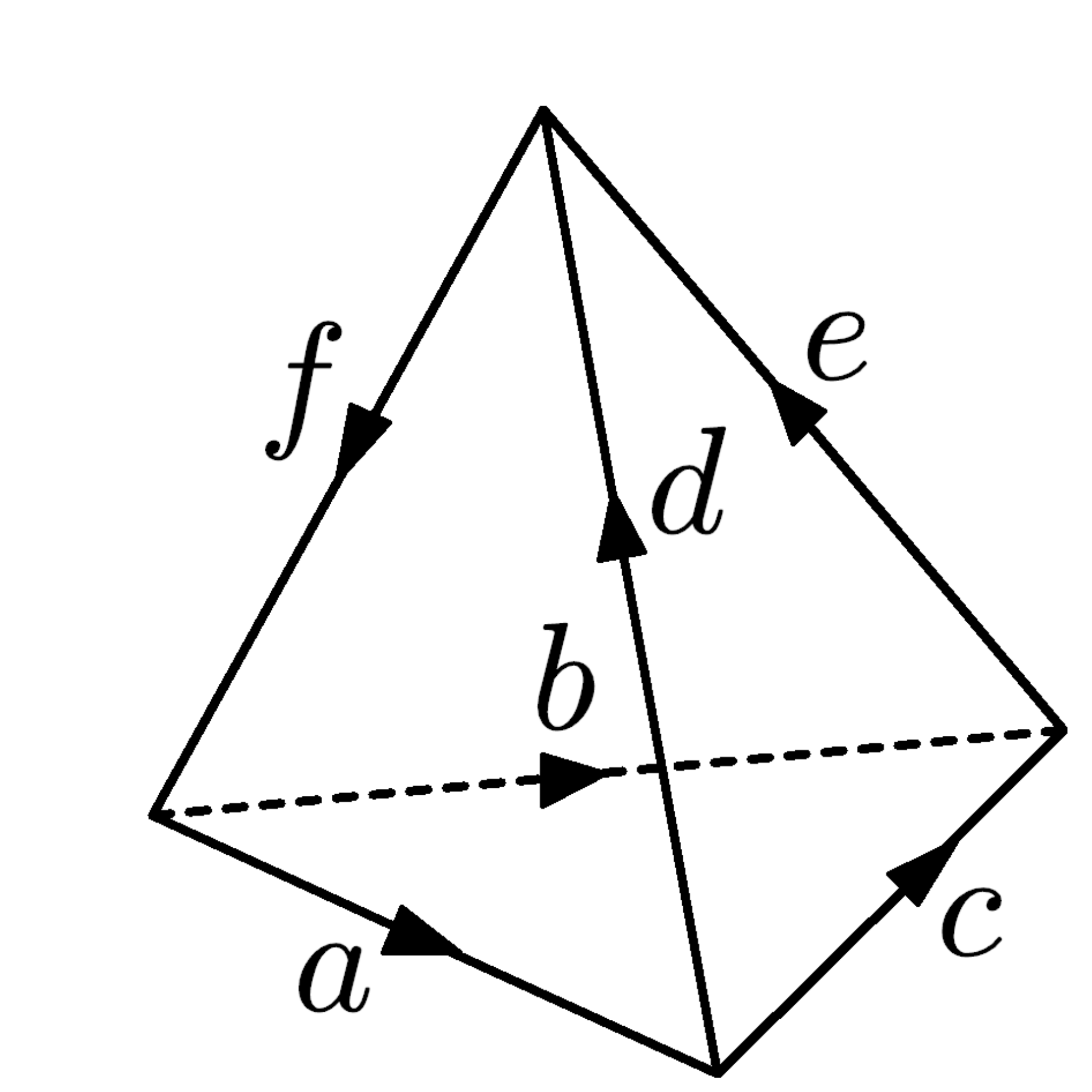} 
\includegraphics[width=2.6cm]{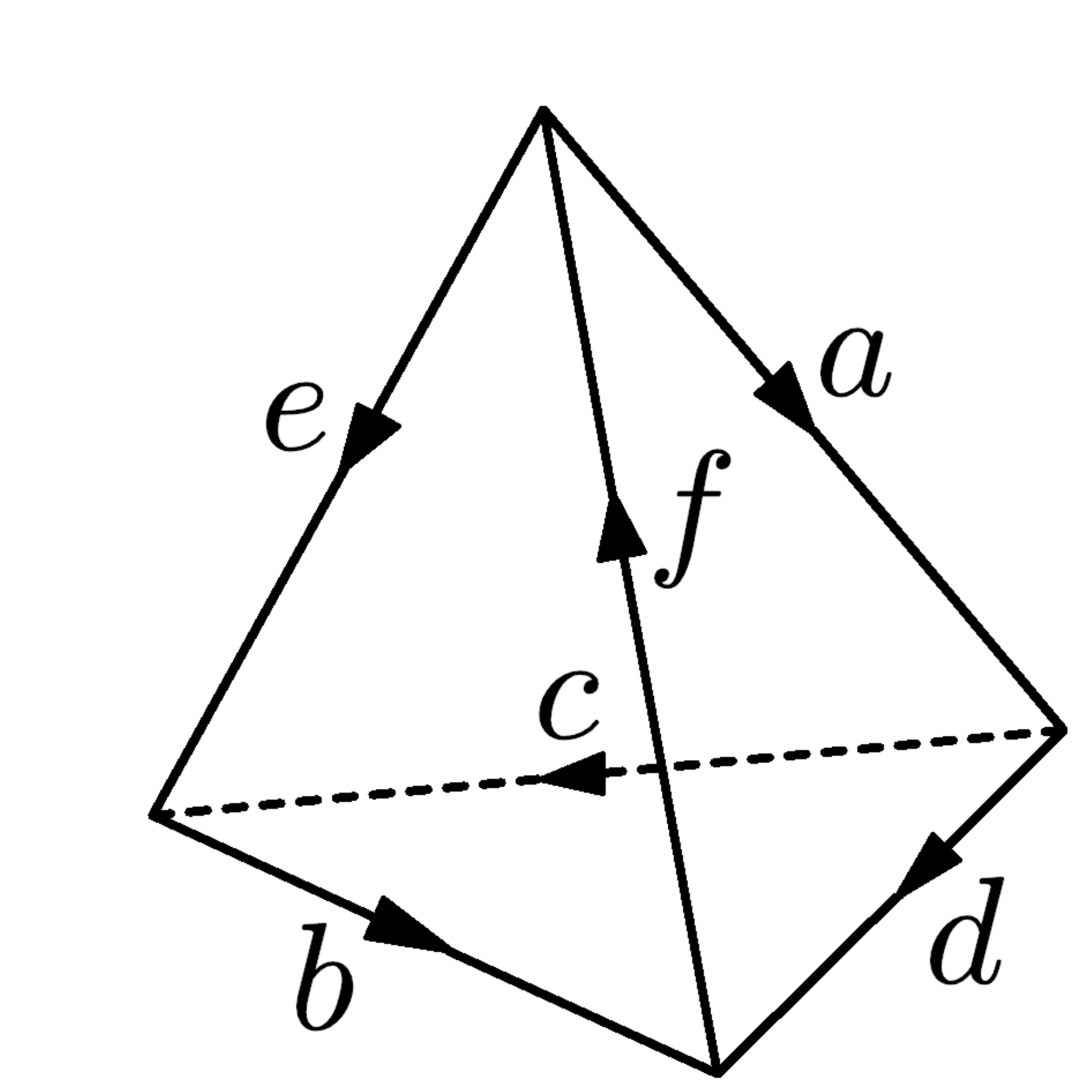}
\includegraphics[width=2.6cm]{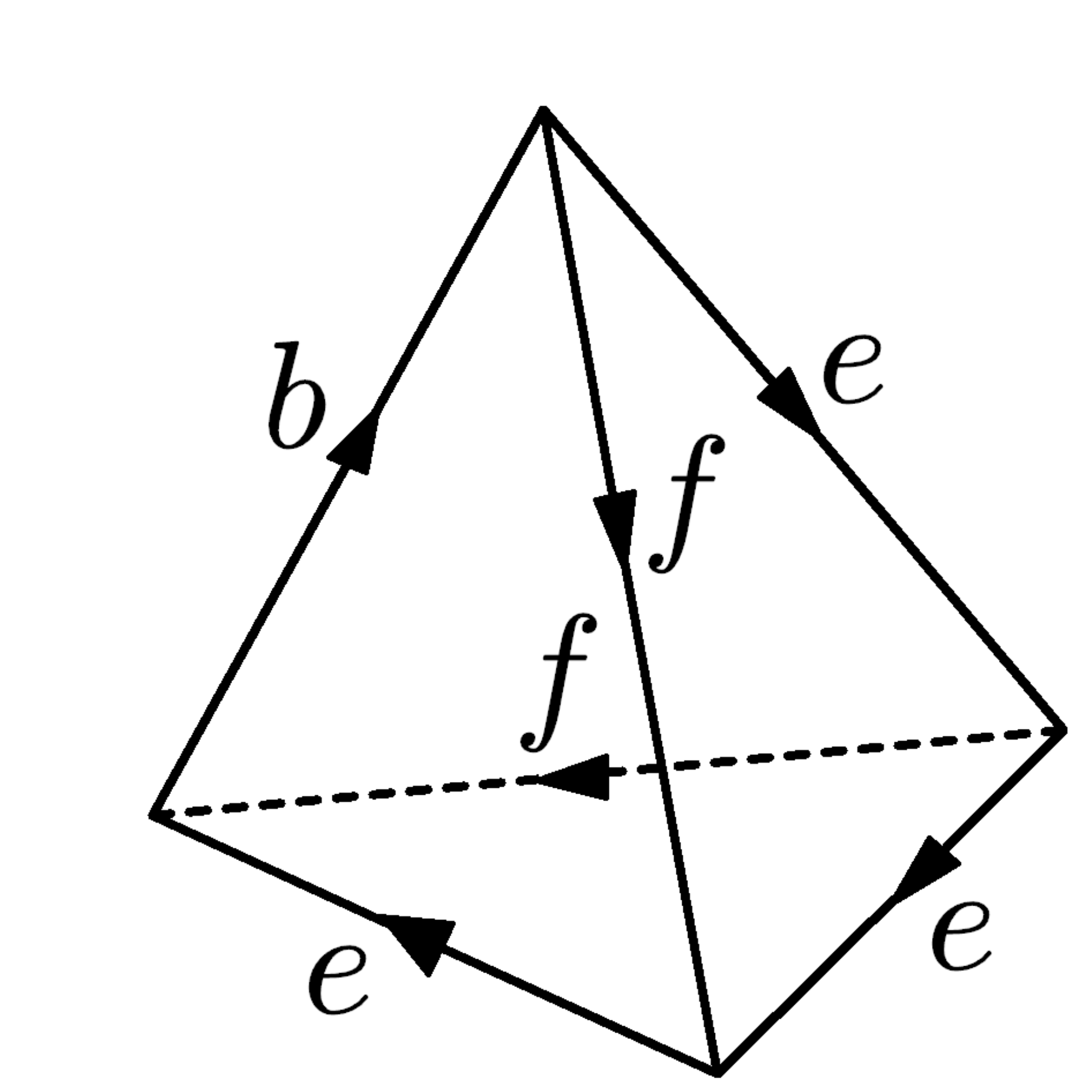} 
\caption{A minimal ideal triangulation of $s788$.}
\end{figure}

(4) $s778$ and $s788$

According to Regina \cite{Burton} and SnapPy \cite{Culler}, $s778$ and $s788$ are cusped orientable 3-manifolds with the minimal ideal triangulations 
shown in Figure 14 and 15 respectively.  
Their hyperbolic volumes, homology groups and $SO(3)$ Turaev-Viro invariants \cite{Kirby} at $r = 5$ 
are as follows:
$${\rm Vol}(s778) = {\rm Vol}(s788) \approx 5.33349,$$
$$H_1(s778 ;\mathbb{Z}) = H_1(s788 ; \mathbb{Z}) = \mathbb{Z} \oplus \mathbb{Z}_{12},$$
$$TV(s778) = 6-2\sqrt{5},\quad TV(s788) = \frac{5-\sqrt{5}}{2}.$$

The minimal ideal triangulations of $s778$ and $s788$ shown in Figure 14 and 15 do not admit a local order.
In order to assign a local order, transform the ideal triangulations of $s778$ and $s788$ by positive (2,3)-Pachner moves.









\begin{figure}
\centering
\includegraphics[width=2.4cm]{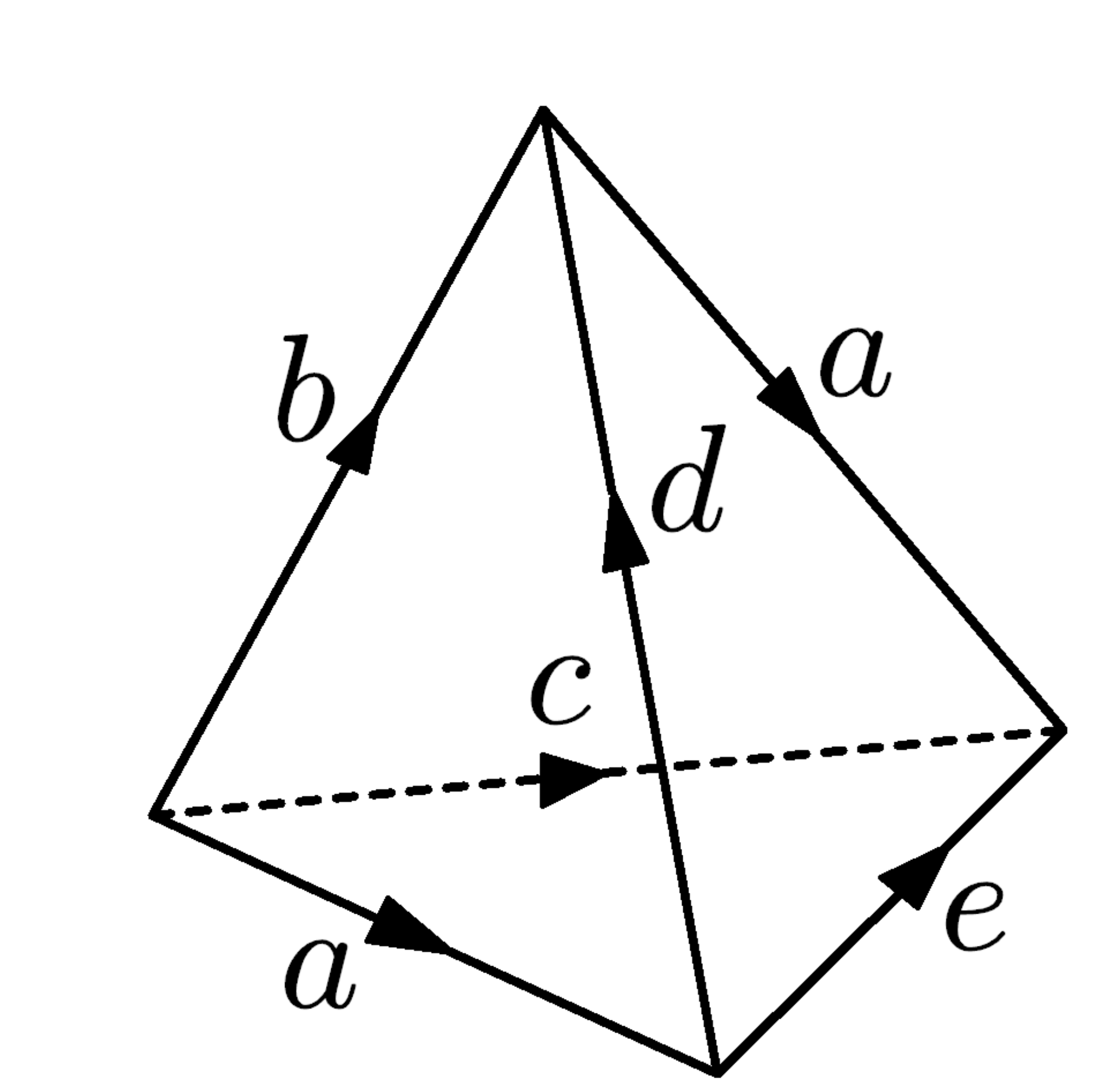}
\includegraphics[width=2.4cm]{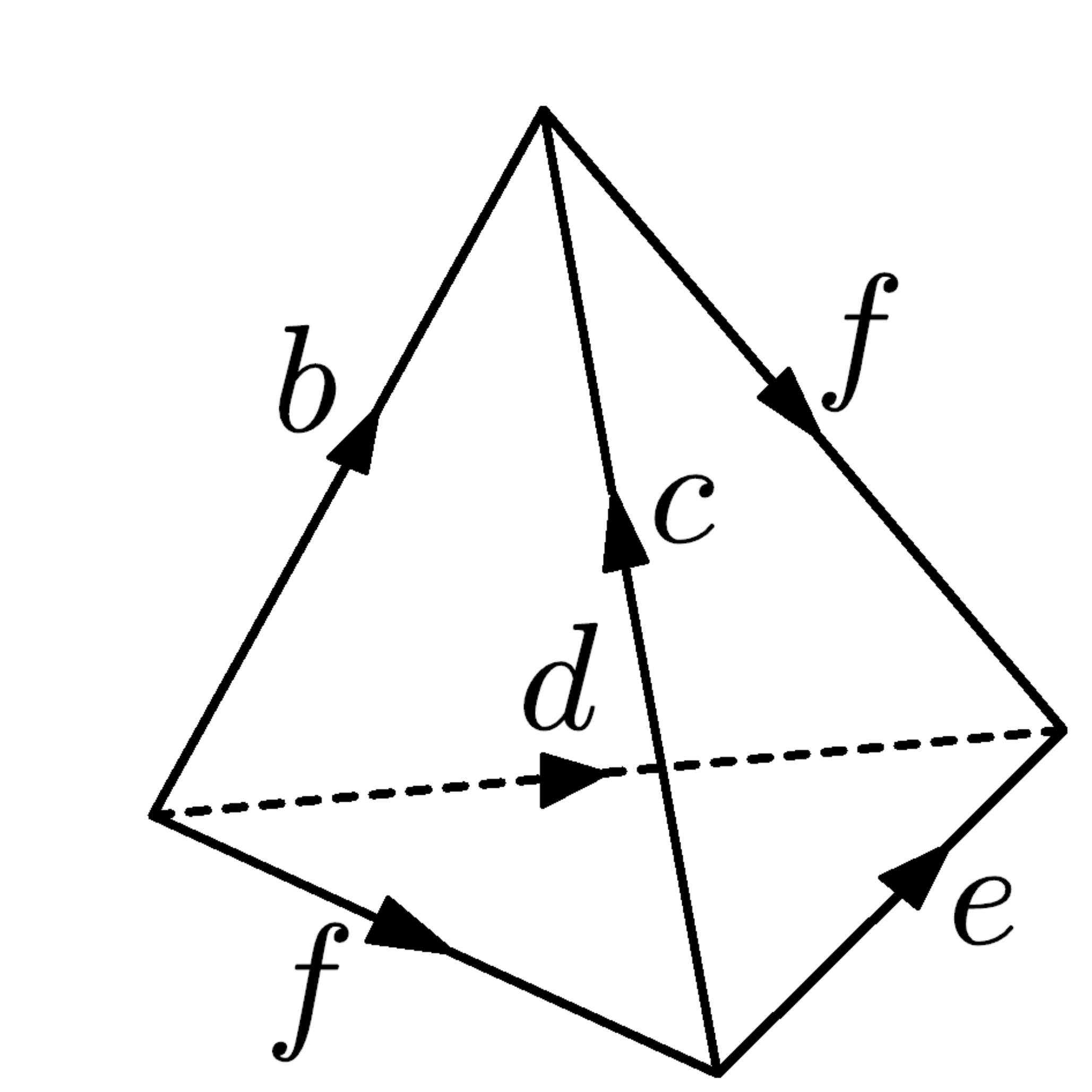} 
\includegraphics[width=2.4cm]{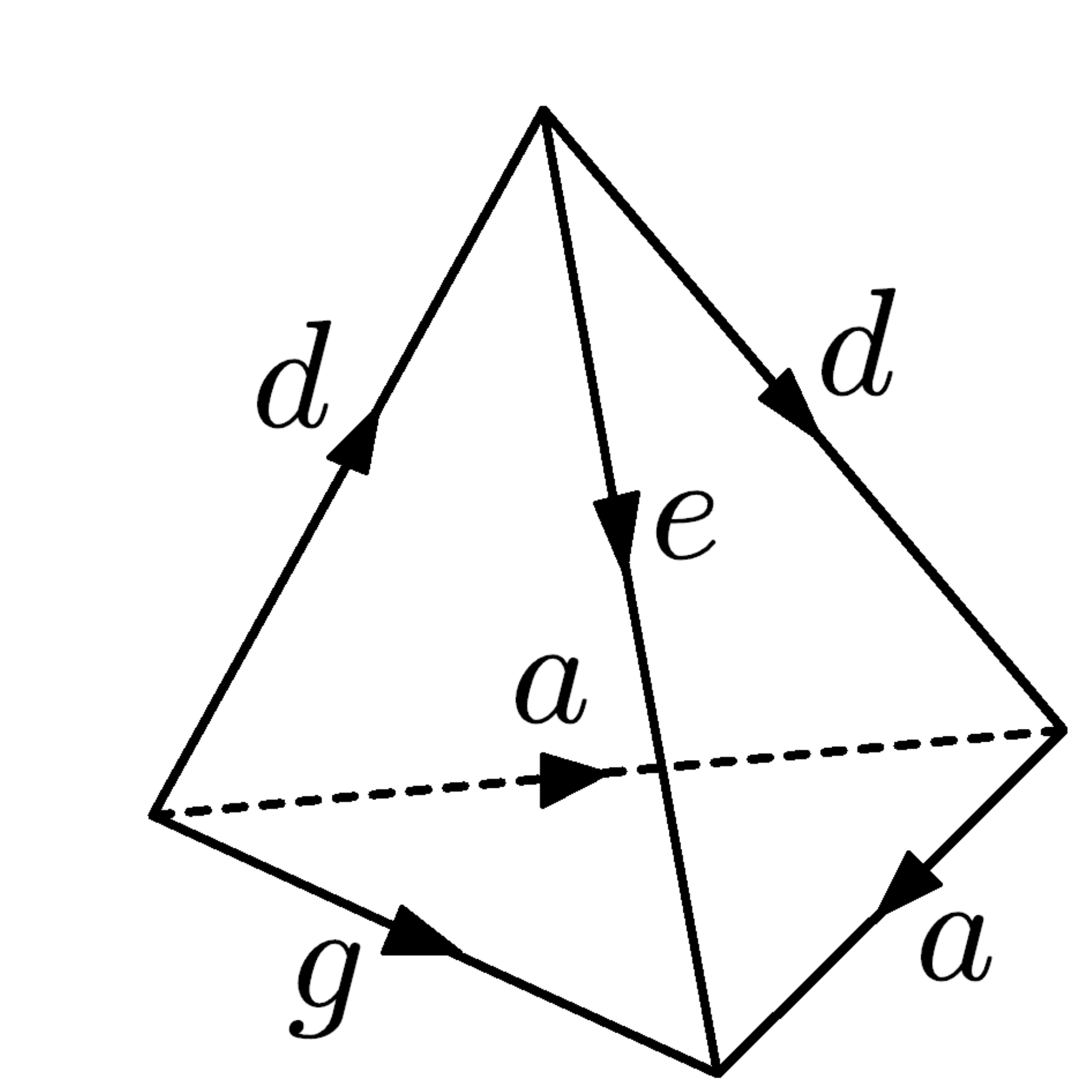}
\includegraphics[width=2.4cm]{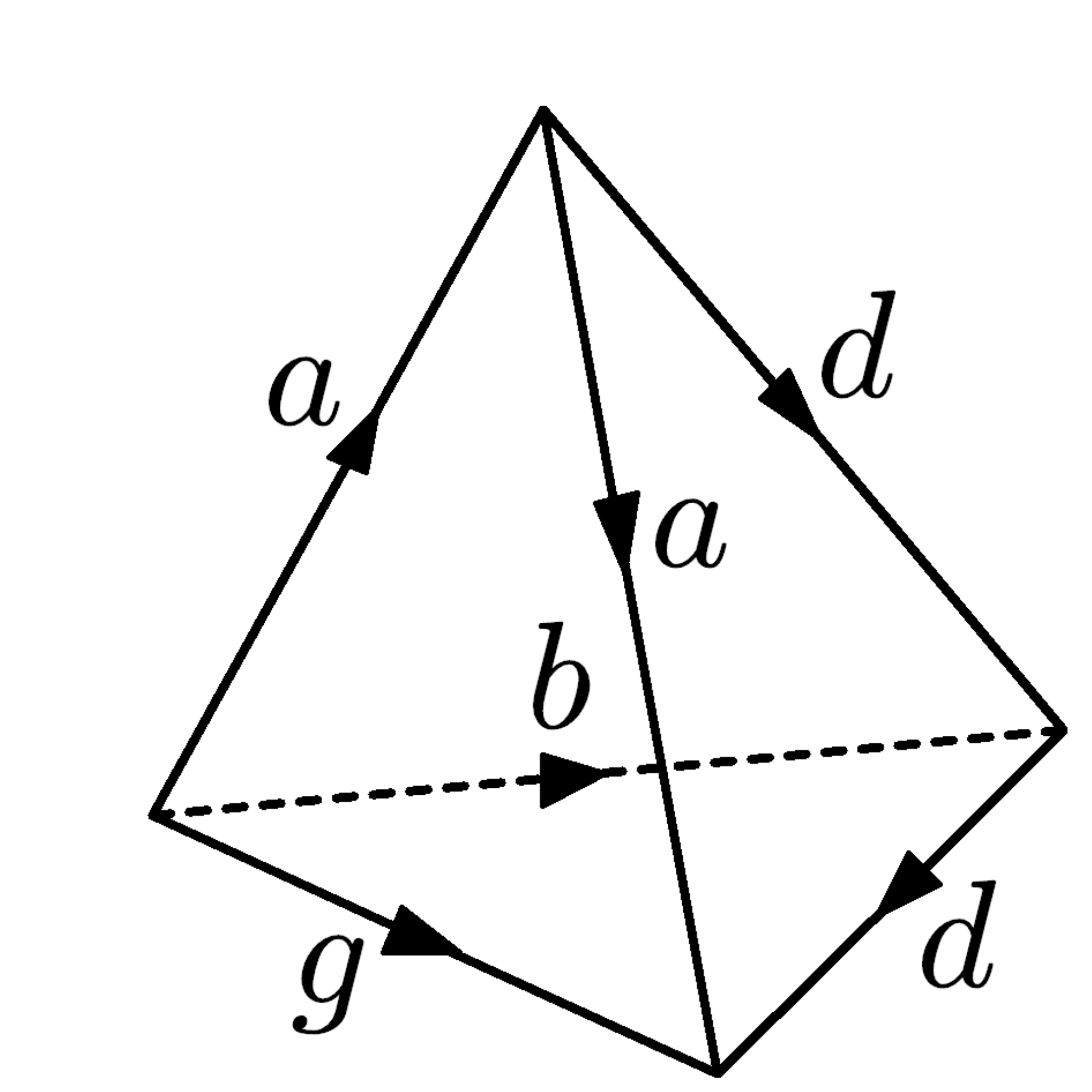}

\vspace{0.15cm}

\includegraphics[width=2.4cm]{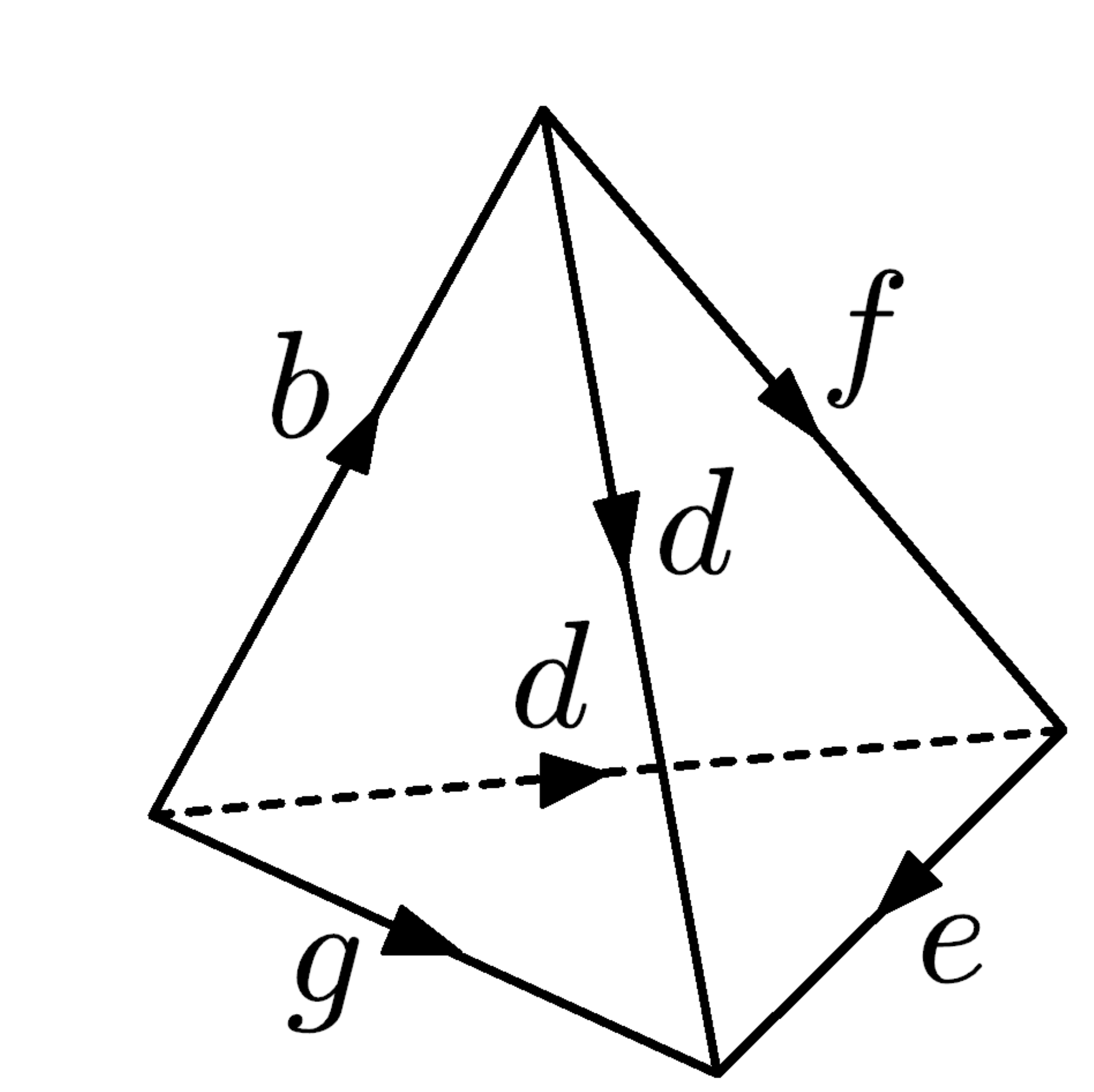}
\includegraphics[width=2.4cm]{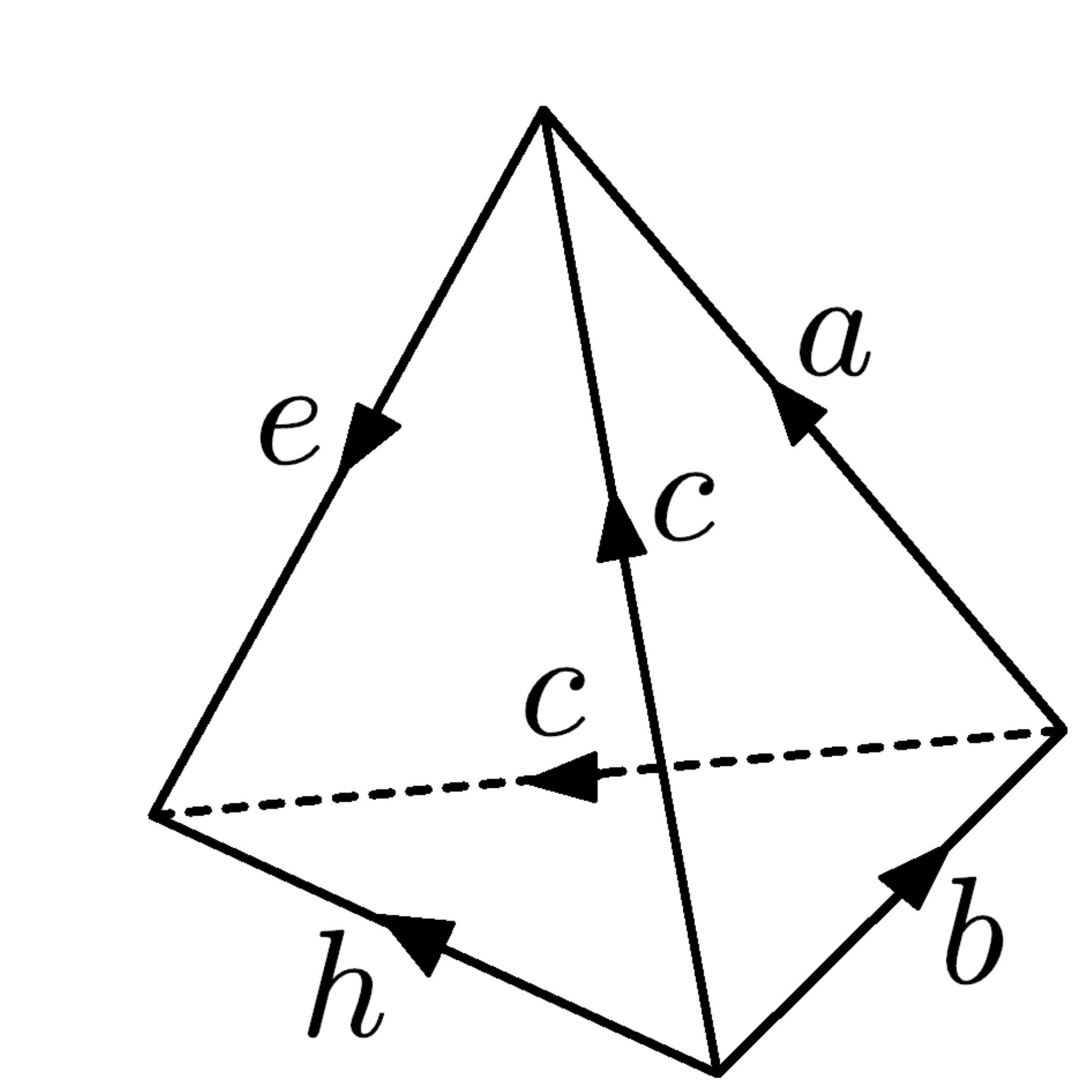}
\includegraphics[width=2.4cm]{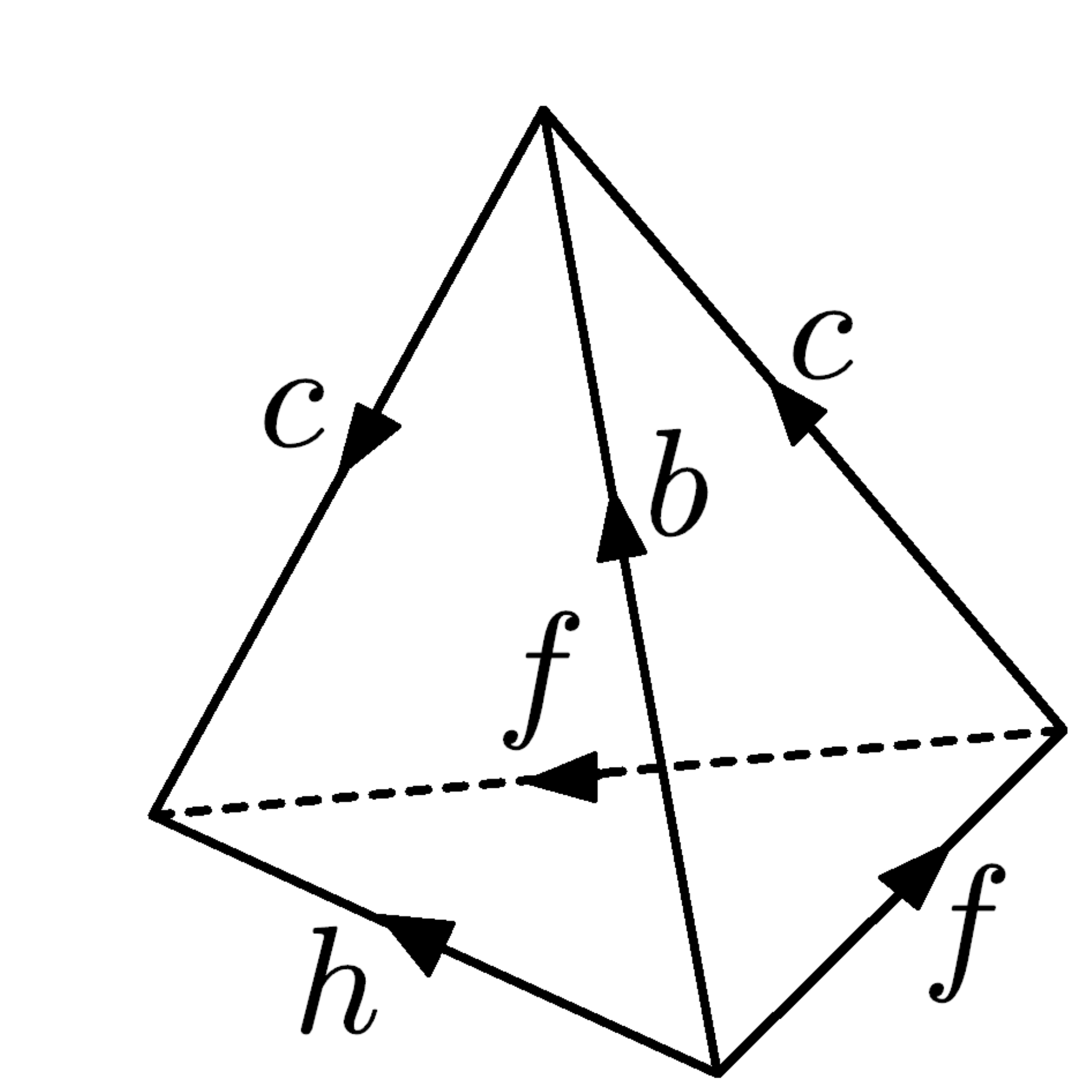}
\includegraphics[width=2.4cm]{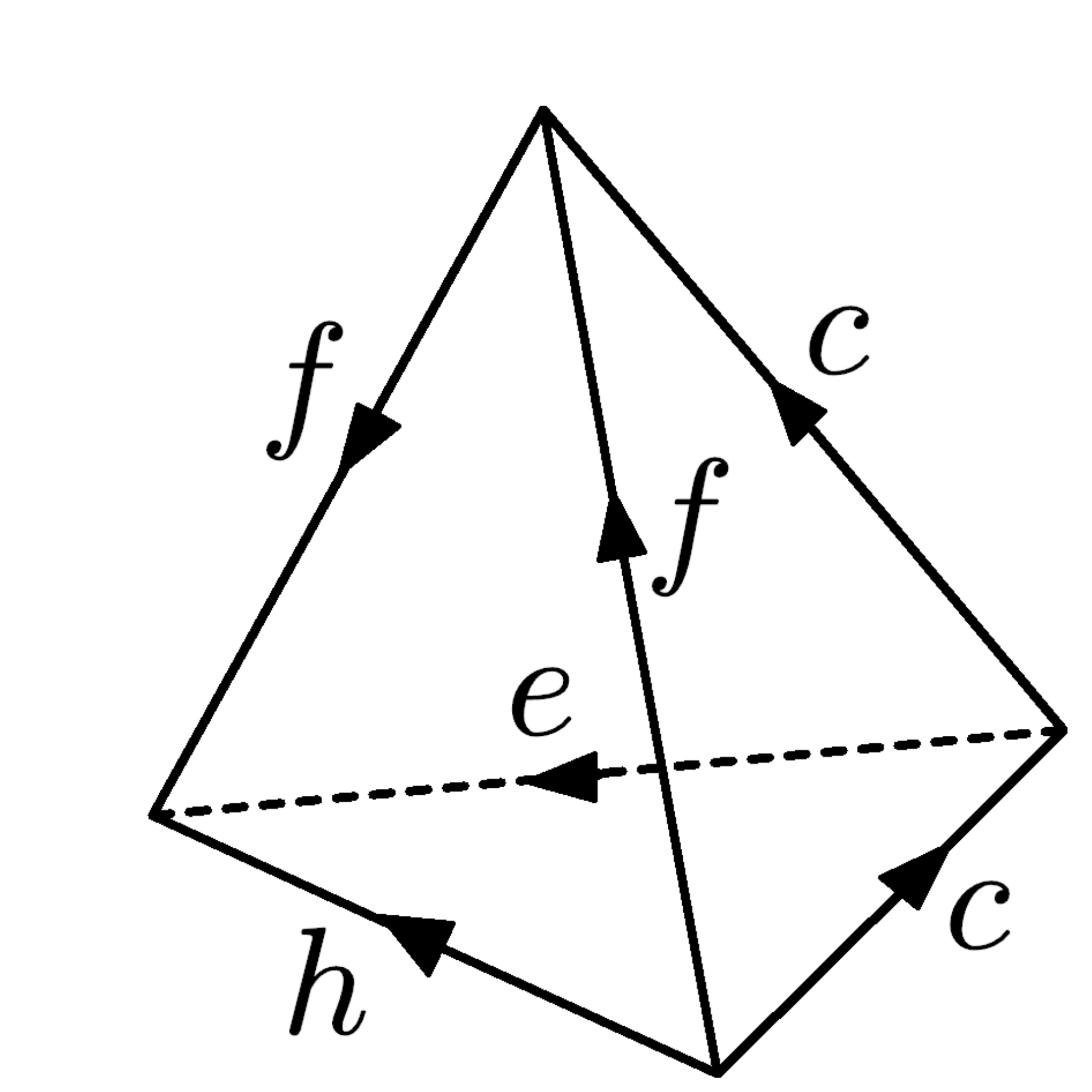}
\caption{An ideal triangulation of $s778$ with a local order.}
\end{figure}

After positive (2,3)-Pachner moves twice, the ideal triangulation of $s778$ which consists of eight ideal tetrahedra admits the local order shown in Figure 16.  
The relations between the colors of edges are the following:

$a=d^2$, $b=e=d^3$, $c=d^5$, $f=d^{10}$, $g=d^4$, $h=d^8$, \;$d^{12}=1.$
 
$$Z(s778) = \sum_{d \in G, d^{12}=1} \alpha(d,d,d^2) \alpha(d^2,d,d) \alpha(d^2,d,d^2)\alpha(d^3,d^2,d^3)$$
$$\hspace{1.5cm} \times \alpha(d^3,d^{10},d^3) \alpha(d^5,d^5,d^{10}) \alpha(d^{10},d^5,d^5) \alpha(d^{10},d^5,d^{10}).$$













\begin{figure}
\centering
\includegraphics[width=2.4cm]{tetra069.pdf}
\includegraphics[width=2.4cm]{tetra070.pdf} 
\includegraphics[width=2.4cm]{tetra071.pdf}
\includegraphics[width=2.4cm]{tetra073.pdf}
\includegraphics[width=2.4cm]{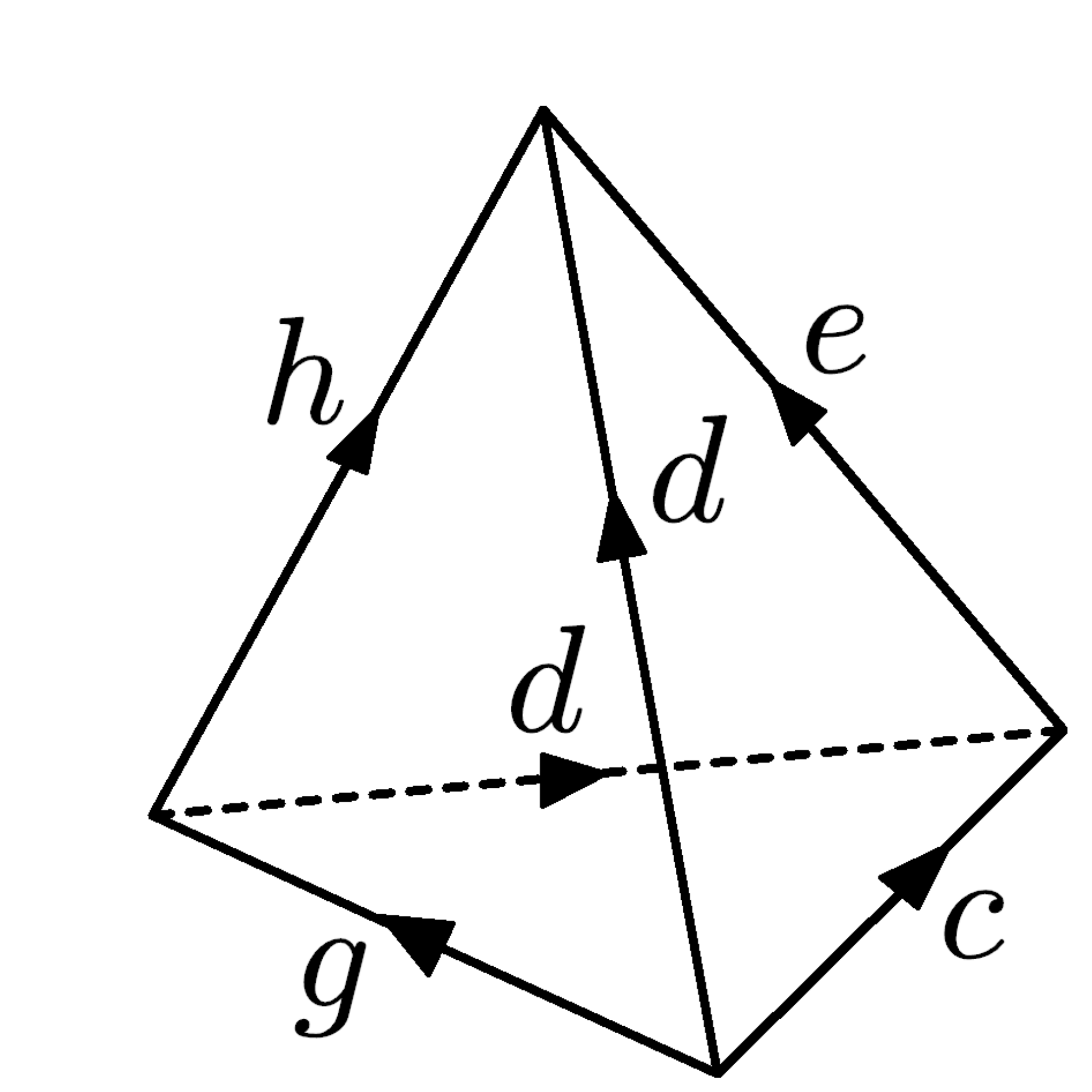}

\vspace{0.3cm}

\hspace{1cm}
\includegraphics[width=2.4cm]{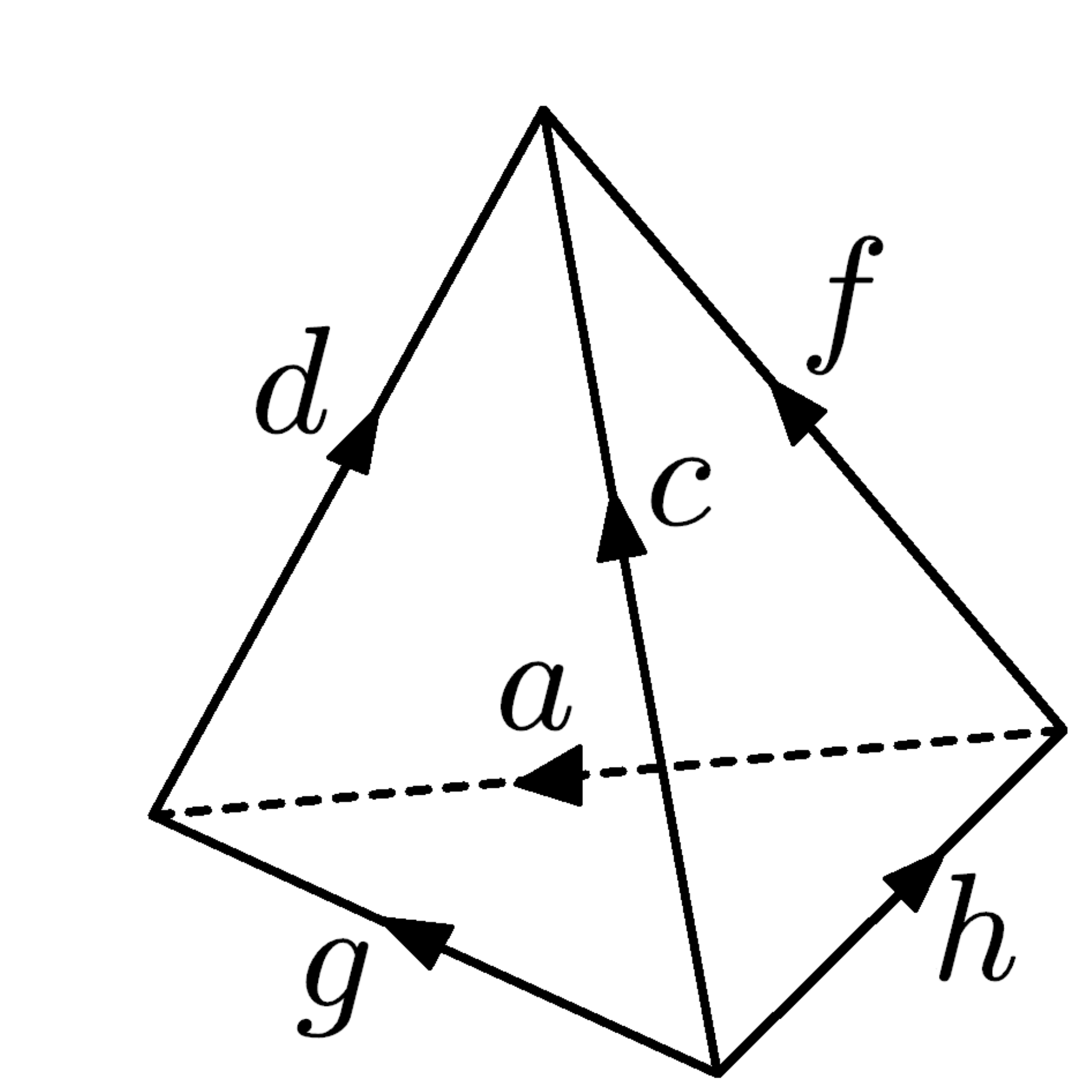}
\includegraphics[width=2.4cm]{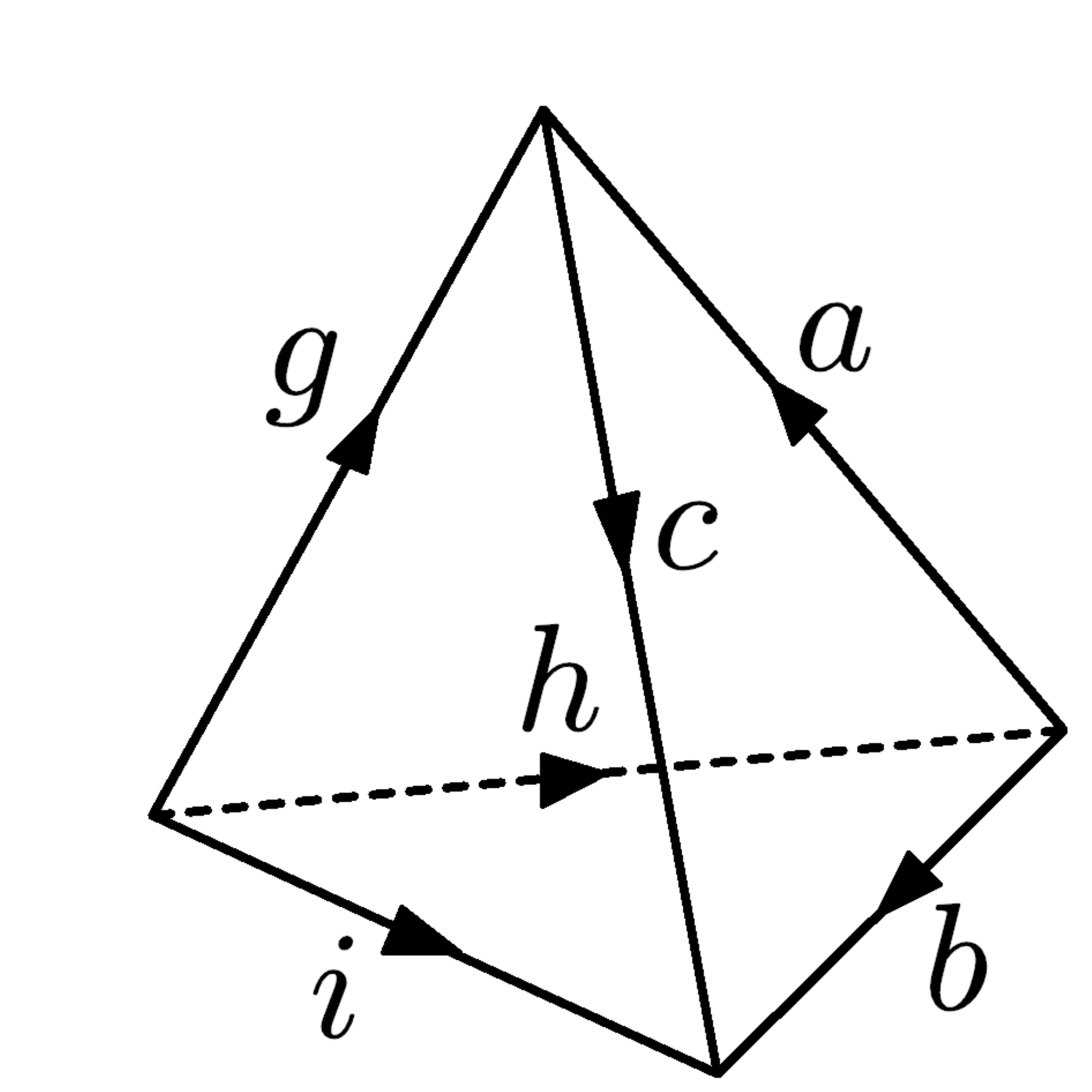}
\includegraphics[width=2.4cm]{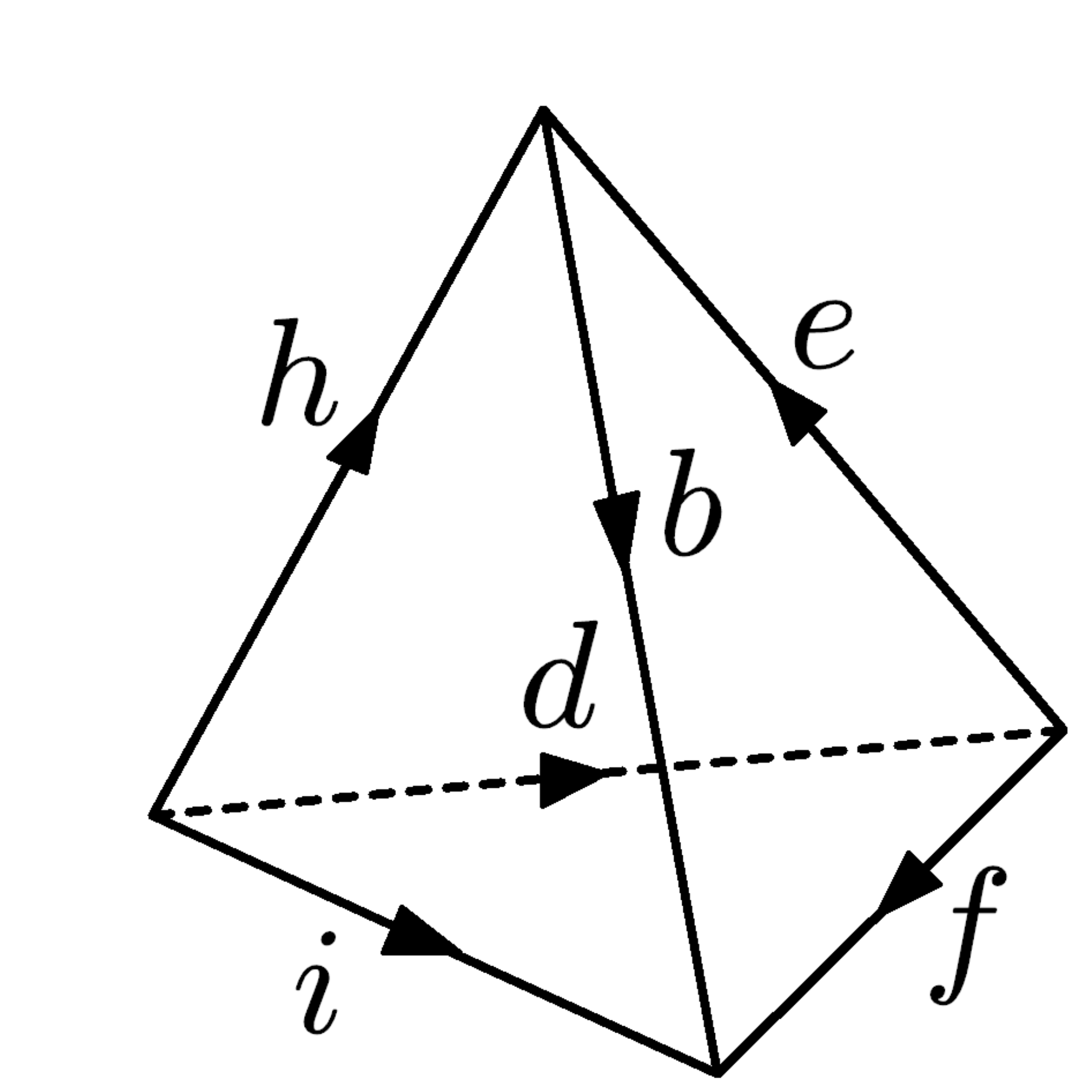}
\includegraphics[width=2.4cm]{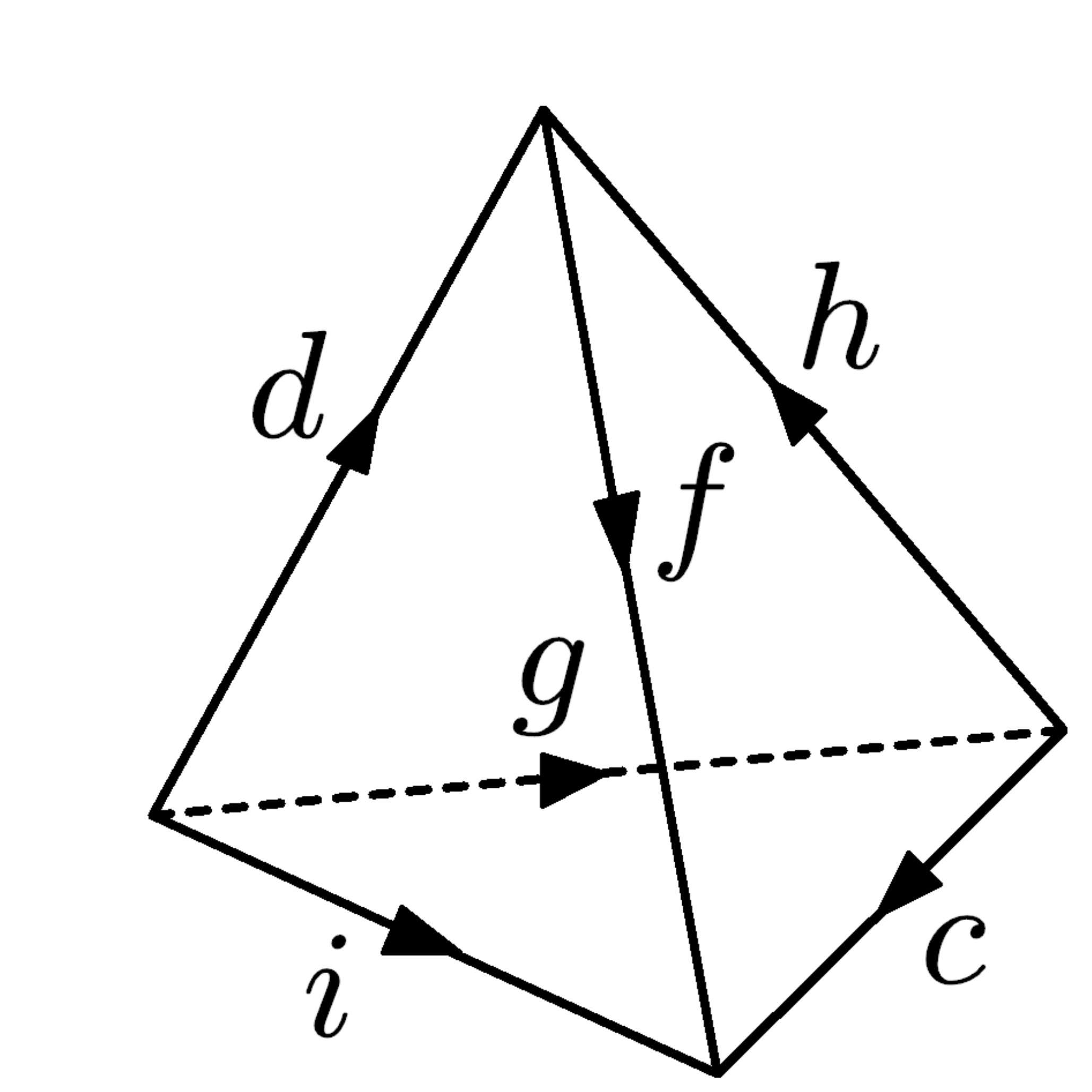}
\caption{An ideal triangulation of $s788$ with a local order.}
\end{figure}

After positive (2,3)-Pachner moves three times, the ideal triangulation of $s788$ which consists of nine ideal tetrahedra admits the local order shown in Figure 17.  
The relations between the colors of edges are the following:

$b=e=a^9$, $c=a^8$, $d=a^5$, $f=a^6$, $g=a^3$, $h=a^2$, $i=a^{-1}$, \;$a^{12}=1.$

$$Z(s788) = \sum_{a \in G, a^{12}=1} \alpha(a^5,a,a^2) \alpha(a^6,a^2,a^3) \alpha(a^8,a,a^2) \alpha(a^8,a,a^8)^{-1}$$
$$\times \alpha(a^8,a^5,a^8)^{-1} \alpha(a^8,a^9,a^8)^{-1} \alpha(a^9,a^5,a^3)^{-1} \alpha(a^9,a^8,a)^{-1} \alpha(a^9,a^9,a^5).$$

In order to confirm $Z(s778) \neq Z(s788)$, we calculate $Z(s778)$ and $Z(s788)$ for 
$G = \mathbb{Z}_{12}$ and a generator $\alpha$ of $H^3(\mathbb{Z}_{12}, U(1)) \cong  \mathbb{Z}_{12}$. 

$$Z(s778) = -6,\quad Z(s788) = 3-2\sqrt{3}.$$


Hence the generalized DW invariants distinguish $s778$ and $s788$.

\end{document}